\documentclass[10pt,11pt]{amsart}
\usepackage{eurosym}
\usepackage[utf8]{inputenc}  
\usepackage[T1]{fontenc}
\usepackage{amsfonts,enumerate,amsmath,amssymb,amsthm}
\usepackage[colorlinks=true, pdfstartview=FitV, pagebackref=true]{hyperref}
\usepackage{pgf,pstricks}
\usepackage[english]{babel}
\usepackage[numbers,sort]{natbib}
\usepackage{colortbl}
\usepackage[absolute]{textpos}
\usepackage{stmaryrd}
\usepackage{subcaption}
\usepackage{dsfont}
\usepackage{float}
\usepackage{pstricks,tikz}
\usepackage{graphicx}
\usepackage{multido}
\usepackage{tikz,tikzrput}
\usepackage{float}
\usepackage{ytableau}
\usepackage{geometry}

\newtheorem{theorem}{Theorem}[section]
\newtheorem{corollary}[theorem]{Corollary}

\newtheorem{lemma}[theorem]{Lemma}
\newtheorem{proposition}[theorem]{Proposition}
\theoremstyle{definition}

\newtheorem{Ass}{Assumption}

\DeclareMathOperator{\Var}{Var}

\newcommand{\N}{\mathbb N}

\renewcommand{\P}{\mathbb P}
\newcommand{\E}{\mathbb E}
\newcommand{\cD}{{\mathcal D}}
\newcommand{\cF}{{\mathcal F}}
\newcommand{\given}{\, \big|\, }
\title{}
\usetikzlibrary{arrows.meta}

\newcommand{\cercleDiscretFleches}[3]{
    
    \begin{tikzpicture}
        \def\radius{3.5} 
        \pgfmathtruncatemacro{\totalNodes}{#1 + 1}
        \pgfmathsetmacro{\nodeSize}{max(2.5, min(8, 180/\totalNodes))}
        \pgfmathsetmacro{\fontSize}{\nodeSize*1.4}
        
        \draw[thick, gray!30] (0,0) circle (\radius);
        
        \foreach \i in {0,...,#1} {
            \pgfmathsetmacro{\angle}{(\i)*(360/\totalNodes)}
            
            \edef\fillcolor{white}
            \foreach \j in {#2} {
                \ifnum\i=\j \xdef\fillcolor{blue!50} \fi
            }
            
            \node[circle, draw=black, fill=\fillcolor, 
                  minimum size=\nodeSize mm, 
                  inner sep=0.2pt,
                  font=\fontsize{\fontSize}{\fontSize}\selectfont] 
                (n\i) at (\angle:\radius) {\i};
        }
        
        \foreach \startNode in {#3} {
            \pgfmathsetmacro{\nextNode}{int(mod(\startNode+1, \totalNodes))}
            
            \pgfmathsetmacro{\bendAngle}{100/\totalNodes + 15}
            
            \draw[-{Stealth[length=2mm]}, blue, thick] 
                (n\startNode) to[bend right=\bendAngle] (n\nextNode);
        }
    \end{tikzpicture}
}

\begin{document}

\title[]{Universality of cutoff for independent random walks on the circle conditioned not to intersect}

\author[]{Anna Ben-Hamou}
\author[]{Pierre Tarrago}
\address[]{Sorbonne University, LPSM \\
4 place Jussieu, 75005 Paris, France}
\email[]{anna.ben\_hamou@sorbonne-universite.fr}
\email[]{pierre.tarrago@sorbonne-universite.fr}
\begin{abstract}
In the present paper, we consider a class of Markov processes on the discrete circle which has been introduced by König, O'Connell and Roch. These processes describe movements of exchangeable interacting particles and are discrete analogues of the unitary Dyson Brownian motion: a random number of particles jump together either to the left or to the right, with trajectories conditioned to never intersect. We provide asymptotic mixing times for stochastic processes in this class as the number of particles goes to infinity, under a sub-Gaussian assumption on the random number of particles moving at each step. As a consequence, we prove that a cutoff phenomenon holds independently of the transition probabilities, subject only to the sub-Gaussian assumption and a minimal aperiodicity hypothesis. Finally, an application to dimer models on the hexagonal lattice is provided.
\end{abstract}
\maketitle

\section{Introduction}
In this paper, we study the mixing times of a class of Markov chains describing particles evolving on a discrete circle without intersecting, which is a generalization of discrete random processes introduced by König, O'Connell and Roch in \cite[Section 5]{konig2002non}. All Markov operators in this class commute and thus have the same invariant measure. These operators have a deep connection with the quantum cohomology of the Grassmannian \cite{guilhot2022quantum}, an algebra describing certain intersections of curves in the Grassmannian manifold. As a byproduct of this relation, they are explicitly diagonalizable and yield an integrable system. Using steepest descent methods and explicit combinatorial properties of the symmetric group, it is then possible to obtain effective upper and lower bounds on the total variation distance between the Markov chains at given times and their invariant measure. As a consequence, we prove that a cutoff phenomenon always occurs under sub-Gaussian and aperiodicity assumptions, independently of the local weights of the Markov chain.

\subsection{Notations and statement of the result}
For $k\leq n$ two positive integers, let $B_{k,n}$ be the set of decreasing sequences of $\llbracket 0,n-1\rrbracket$ of length $k$. An element $I=\{I(1)>\dots >I(k)\}\in B_{k,n}$ will be viewed as the positions of $k$ particles on a cycle with $n$ sites, numbered anti-clockwise from $0$ to $n-1$, and will often be associated with
\[
\xi(I)=(\omega^{I(1)-\frac{k-1}{2}},\dots, \omega^{I(k)-\frac{k-1}{2}})\, ,
\]
where $\omega=e^{i2\pi/n}$. Now for $\ell\in\llbracket 1,k\rrbracket$, let $A^{(\ell)}$ be the adjacency matrix of the oriented graph on $B_{k,n}$ corresponding to $\ell$ particles moving one step anti-clockwise on the circle without intersecting (note that if two neighboring particles are followed by an empty site, then they can simultaneously jump of one step). 

For $\ell\in \llbracket -k,-1\rrbracket$, we set $A^{(\ell)}=\,^t\left[A^{(-\ell)}\right]$, which corresponds to clockwise moves. See Figure \ref{Fig:movement} for an example of possible moves and the resulting corresponding graph in small examples. Remark that the case $\ell< 0$ was not considered in \cite{konig2002non}. Although it does not bring any technical or insightful novelty, it is necessary to consider conditioned version of the ASEP on the circle, see Section \ref{subsec:condASEP}.

\begin{figure}[h!]
\scalebox{0.8}{\cercleDiscretFleches{9}{3,4,6,7,9}{3, 4, 9}}
\hspace{1cm}
\includegraphics[scale=0.4]{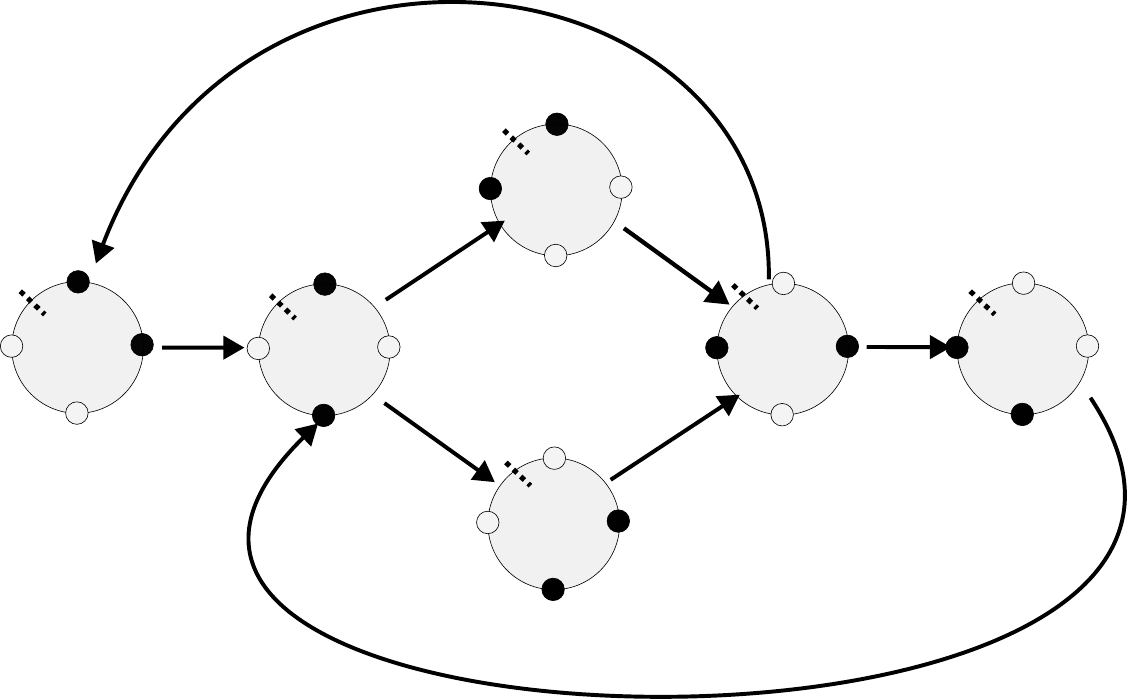}
\caption{\label{Fig:movement}Possible move corresponding to $A^{(3)}$ on $B_{5,10}$ and graph corresponding to the adjacency matrix $A^{(-1)}$ on $B_{2,4}$ (the dashed line marks the origin of the circle).}
\end{figure}

Two elements of $B_{k,n}$ will play a particular role:
\[
 I_0= (k-1,\dots,1,0)\qquad\text{and}\qquad I_1=  I_0+(1,0,\dots,0)\, .
 \]
Letting $e_\ell$ be the $\ell^{\text{th}}$ elementary polynomial, the Perron-Frobenius eigenvalue of $A^{(\ell)}$ is given by
\[
\alpha^{(\ell)}_{I_0}=e_\ell(1,\omega,\dots,\omega^{k-1})=\omega^{\frac{\ell(\ell-1)}{2}}\frac{(1-\omega^k)\dots(1-\omega^{k-\ell+1})}{(1-\omega)\dots(1-\omega^\ell)}\, ,
\]
associated with the eigenvector $\varphi_{I_0}$ given by 
\[
\varphi_{I_0}(I)=\frac{V(\xi(I))}{V(\xi(I_0))}\, ,
\]
where $V(\xi(I))$ is the Vandermonde determinant:
\[
V(\xi(I))=\prod_{i<j} \left(\omega^{I(j)}-\omega^{I(i)}\right)\, .
\]
The Doob $h$-transform of $A^{(\ell)}$ is then defined by
\[
Q^{(\ell)}(I,J)=\frac{1}{\alpha^{(\ell)}_{I_0}}\frac{\varphi_{I_0}(J)}{\varphi_{I_0}(I)} A^{(\ell)}(I,J)\, .
\]
This transition matrix approximates the dynamic of a Markov chain on $B_{k,n}$ where, at each step, $\ell$ particles make a step (anti-clockwise if $\ell\geq 1$, clockwise if $\ell\leq -1$), conditionally on the event that no collision ever occurs. All the matrices $Q^{(\ell)}$ share the same stationary distribution given by
\[
\mu(I)=\frac{|V(\xi(I))|^2}{n^k}\, \cdot 
\]

Let $p=(p_{-k},\ldots,p_{-1},p_0,p_1,\dots,p_{k})$ be a probability vector $\llbracket -k,k\rrbracket$. We are interested in the mixing time of the Markov chain on $B_{k,n}$ with transition kernel
\begin{equation}\label{eq:def-P}
P=\sum_{\ell=-k}^{k} p_\ell Q^{(\ell)}\, ,
\end{equation}
where $Q^{(0)}$ is the identity matrix. The eigenvalues are given by
\[
\lambda_J=\sum_{\ell=-k}^{k}p_\ell\frac{e_\ell(\xi(J))}{e_\ell(\xi(I_0))}
\]
(see section~\ref{sec:spectral-gap}) and we introduce the spectral parameter 
\begin{equation}\label{eq:def_spectral_gap}
\gamma=1-|\lambda_{I_1}|.
\end{equation}
Beware that $\gamma$ may not be the spectral gap of the random walk, depending on how the random walk is close to being periodic. However, under our assumptions, $\gamma$ will give the main contribution to the mixing time, see Theorem \ref{thm:main}. Assuming that the gcd of the support of $p$ is equal to $1$, the chain $P$ is irreducible and aperiodic. This implies that, given any initial probability distribution $\mu_0$ on $B_{k,n}$, the distribution $\mu_0 P^t$ converges towards $\mu$ as $t$ goes to infinity. Let us introduce the quantity 
$$\mathcal{D}(t)=\sup_{\mu_0\in\mathcal{P}\left(B_{k,n}\right)}d_{TV}\left(\mu_0P^t,\mu\right)\, .$$
This quantity quantifies the convergence towards the stationary distribution and is the main object of study of the present manuscript. For $\epsilon>0$, introduce the mixing time 
$$t_{\epsilon}=\inf(t\geq 0, \mathcal{D}(t)\leq \epsilon).$$
We will consider the following regime. In what follows, $X$ denotes a random variable with distribution~$p$.
\begin{Ass}[Number of particles]\label{ass:number-particles}
There exists $\eta\in (0,1/2)$ such that 
\[
\eta\leq \frac{k}{n}\leq 1-\eta\, .
\]
\end{Ass}

\begin{Ass}[Expectation]\label{ass:far-k}
There exists $\delta>0$ such that 
\[
\delta\leq \E[|X|]\leq  (1-\delta)k\, .
\]
\end{Ass}
Due to the symmetries of the system, the upper bound of the latter assumption could be replaced by $k-\delta$ up to small changes in the proofs, see the second paragraph of Section \ref{subsec:related_work}.
\begin{Ass}[Sub-Gaussianity]\label{ass:X-subgaussian}
There exists $K_{g}>0$ such that
$( |X|-\mathbb{E}[|X|])$ is sub-Gaussian with parameter $K_g\sqrt{\E[|X|]}$.
\end{Ass}
We guess that this assumption may be weaken in the case where $\E[|X|]$ is of order smaller than $k$. For example, one expects that a sub-exponentiality condition instead of sub-gaussianity (with the same order of parameters) would suffice if $\E[|X|]$ is of order $1$.

\begin{Ass}[Aperiodicity]\label{ass:X-aperiodicity}
There exist $K_{a}>0$ such that the Fourier transform $\Phi_p$ of the variable $X$ satisfies, for all $\theta\in [-\pi,\pi]$,
\[\left\vert\Phi_{p}(\theta)\right\vert\leq 1-K_{a}\min\{\E[|X|]\theta^2\, ,\, 1\}\, .
\]
 \end{Ass}
The latter assumption means that the random walk $X$ must be aperiodic and have a variance of order at least $\E[|X|]$. 

For $\eta\in (0,1/2)$ and $\delta,K_g>0$, let us denote by $\mathcal{A}(\eta,\delta,K_g,K_a)$ the set of sequences of kernels $(P^{(n)})_{n\geq 1}$ of the form~\eqref{eq:def-P} such that for all $n\geq 1$,  Assumption~\ref{ass:number-particles} is satisfied with constant $\eta$, and the probability distribution $p$ over $\{-k,\dots,k\}$ satisfies Assumption~\ref{ass:far-k} with constant $\delta$, Assumption~\ref{ass:X-subgaussian} with parameter $K_g$, and Assumption~\ref{ass:X-aperiodicity} with parameter $K_a$.

\begin{theorem}\label{thm:main}
Let $\eta\in (0,1/2)$ and $\delta,K_g,K_a>0$. Any sequence of kernels belonging to $\mathcal{A}(\eta,\delta,K_g,K_a)$ satisfies
$$\lim_{s\rightarrow +\infty}\liminf_{n\rightarrow +\infty}\mathcal{D}\left(\frac{1}{\gamma}(\log n-s)\right)=1\text{ and }\lim_{s\rightarrow +\infty}\limsup_{n\rightarrow +\infty}\mathcal{D}\left(\frac{1}{\gamma}(\log n+s)\right)=0.$$
\end{theorem}
This theoreom is a consequence of Proposition \ref{prop:lower-bound}, Proposition \ref{prop:upper_bound_1} and \eqref{eq:inequality_l2_l1}. According to Theorem \ref{thm:main}, any sequence of kernels belonging to $\mathcal{A}(\eta,\delta,K_g,K_a)$ exhibits the cutoff phenomenon at time $\tfrac{\log n}{\gamma}$ and with window of order $\tfrac{1}{\gamma}$. As a corollary, we have the following equivalent for the mixing time.

\begin{corollary}\label{cor:mixing_time}
Let $\eta\in (0,1/2)$ and $\delta,K_g,K_a>0$. Any sequence of kernels belonging to $\mathcal{A}(\eta,\delta,K_g,K_a)$ satisfies
$$t_{\epsilon} \underset{n\to +\infty}{\sim}\frac{\log n}{\gamma}\, .$$
\end{corollary}

\subsection{Related work}\label{subsec:related_work}
The transition matrices $Q^{(\ell)}$ can be understood as follows: consider the dynamics of $k$ particles on the circle where at each step $\ell$ particles make a step in the same direction (either all clockwise or all anti-clockwise). Then condition on the very unlikely event that, during the whole process, no particles ever collide. The trajectory of this conditioned process can then be approximated by a Markov chain with transition matrix $Q^{(\ell)}$. The stationary distribution of this process is an instance of hyper-uniform distributions: it is a discrete version of the Circular $\beta$-Ensemble for $\beta=2$. 

There is a much broader class of interacting particles that generalizes the cases studied in this paper. As previously stated, the transition matrices considered here are deeply related to the elementary symmetric functions $e_{\ell}$ with $0\leq \ell\leq n$. Similarly, there exists a transition operator for any Schur function whose associated Young diagram fits into a rectangle of size $k\times n-k$, and for each of these transition operators, a stationary distribution is the discrete Circular Unitary Ensemble, see \cite{guilhot2022quantum}. The reason behind this fact is that they generate a commutative algebra called the quantum cohomology of the Grassmannian, see \cite{bertram}. As a particular case of this that we did not pursue for the sake of clarity, a straightforward generalization would be to also consider a shift of all particles by more than one step; such a generalization allows for the removal of the upper bound in Assumption \ref{ass:far-k}.

Actually, the processes we are considering are discrete versions of the circular unitary Brownian motion whose stationary distribution is the Circular Unitary Ensemble. It has been shown in \cite{meliot2014cut} that this process exhibits a cutoff phenomenon as the number of particles tends to infinity. Many more processes are considered in \cite{meliot2014cut} for various Brownian motions in symmetric space coming from different simple Lie algebras. Discrete versions of such continuous processes also exist for the orthogonal and symplectic groups, see \cite[Section 7]{guilhot2023homology}, although no mixing times are known in these cases.

Theorem~\ref{thm:main} can be seen as a universality result for the occurrence of cutoff: under mild conditions on the distribution $p$ of the number of moved particles, the corresponding chain exhibits cutoff, at a time that depends on this distribution only through the induced spectral gap. Such universality phenomena in the mixing behavior of interacting particle systems have already been highlighted. For instance, the result of Salez~\cite{salez2023universality} establishes the universality of cutoff for the zero-range process with bounded monotone rates.

In the case where $p$ is supported on $\{-1,0,1\}$, it is interesting to compare our result with the Simple Exclusion Process, either symmetric or asymmetric. In this process, at each step, a particle attempts to move left or right. If the targeted site is occupied, then the move is censored and the chain stays in the same configuration. The stationary distribution is then the uniform distribution on configurations. In the symmetric case with $k=\alpha n$ particles, $\alpha\in (0,1)$, the result of Lacoin~\cite{lacoin2017simple,lacoin2016cutoff} shows that the process has cutoff at time $\tfrac{n^3\log n}{4\pi^2}$. Recently, Schmid and Sly~\cite{schmid2022mixing} showed the totally asymmetric exclusion process on the circle mixes in time of order $n^{5/2}$, without cutoff. In comparison, the chains described below in sections~\ref{subsec:constant} (with $k_0=1$) and~\ref{subsec:condASEP}, both corresponding to one single particle moving at a time, have cutoff at a time of order $n^2\log n$. Stationarity is thus attained much shortly than for simple exclusion processes. Note that the invariant distribution $\mu$ can be simulated by generating a random matrix in the unitary group (for example by diagonalizing a random matrix from the Gaussian Unitary Ensemble) and outputting its eigenvalues. Typically, this method has complexity $n^3$. Hence, provided that simulating a transition from the chains described in~\ref{subsec:constant} and~\ref{subsec:condASEP} has time complexity much smaller than $n$, then those chains can provide an efficient way to simulate from~$\mu$. Also note that for chains $P$ of the form~\eqref{eq:def-P} with a spectral gap~$\gamma$ of order~$n^{-1}$ such as the periodic dimer model described in section~\ref{subsec:dimer}, the mixing time can be as small as order~$n\log n$. 

\subsection{Structure of the paper and specific notations} Apart from the next subsection which is provides applications of Theorem \ref{thm:main}, the remaining sections are devoted to the proof of this theorem. Section \ref{sec:spectral-gap} gives precise estimates on the spectral gap $\gamma$ which plays a prominent role in the statement of the theorem. We also relate in this section $\gamma$ to each individual spectral gap $\gamma_{\ell}$ corresponding to the matrix $Q^{(\ell)}$. Section \ref{sec:lower-bound} is then devoted to the proof of the lower bound of Theorem \ref{thm:main}, which is embodied by Proposition \ref{prop:lower-bound}. Section \ref{sec:upper-bound} provides then the upper bound of the theorem in Proposition \ref{prop:upper_bound_1}. The proof of this proposition relies on many estimates on the individual eigenvalues $\lambda_{J}^{\ell}$ of the matrices $Q^{(\ell)}$. A summary of these estimates is given in Proposition \ref{prop:estimates_J} of this section. The proof of the estimates in the case where $\ell$ is larger than a power of $\log n$ is given in Section \ref{sec:estimates_superlog}, while the proof in the case of $\ell$ being of order a power of $\log n$ or smaller is given in Section \ref{sec:estimates_lowerlog}.

Throughout this paper, we use several constants $c,C>0$ which are always assumed to depend at most only on $\delta,\eta, K_g$ and $K_a$. We then write $A\asymp B$ to say that $c\leq \frac{A}{B}\leq C$ for some constant $c,C>0$ and $O(B)$  to denote a quantity $A\in\mathbb{C}$ such that $\vert A\vert \leq C\vert B\vert$ for some constant $C>0$. To emphasize the fact that $A,B\in \mathbb{R}$, we write $O_{\mathbb{R}}(B)$ instead of $O(B)$.
\subsection{Application to simple models}
\subsubsection{Constant Markov kernel}\label{subsec:constant}
Let $p=(p_{-k_0},\ldots,p_{k_0})$ be a probability vector on $\llbracket -k_0,k_0\rrbracket$ generating an aperiodic random walk on $\mathbb{Z}$, and set $p^{(n)}=p$ for $k\geq k_0, n\geq k$. Then, $p^{(n)}$ satisfies Assumptions~\ref{ass:far-k}, \ref{ass:X-subgaussian} and~\ref{ass:X-aperiodicity}. Moreover, by Proposition \ref{prop:spectral-gap}, as $k,n\rightarrow +\infty$ with $\frac{k}{n}\rightarrow \theta\in ]0,1[$; the corresponding spectral gap $\gamma$ satisfies 
\begin{align*}
\gamma=&2\left(1+O\left(\frac{1}{n}\right)\right)\frac{\sin\left(\tfrac{\pi}{n}\right)}{\sin\left(\tfrac{k\pi}{n}\right)}\sum_{\ell=1}^{k_0}(p_\ell+p_{-\ell}) \sin\left(\tfrac{\ell\pi}{n}\right)\sin\left(\tfrac{(k-\ell)\pi}{n}\right)=2\left(1+O\left(\frac{1}{n}\right)\right)\frac{\pi^2\mathbb{E}\left[\left\vert X\right\vert\right]}{n^2}.
\end{align*}
Hence, by Corollary \ref{cor:mixing_time}, for all $\epsilon>0$,
$$t_{\epsilon}\underset{n\rightarrow +\infty}{\sim}\frac{n^2\log n}{2\pi^2\mathbb{E}\left[\left\vert X\right\vert\right]}.$$

\subsubsection{Conditioned ASEP}\label{subsec:condASEP} 
Let $\alpha,\beta\geq 0$ be such that $\alpha+\beta>0$ and set $\epsilon=1-\frac{\alpha+\beta}{k}$. Let $k\geq 1$ be such that $\epsilon>0$ and set $A=\frac{1}{k}\alpha A^{(-1)}+\epsilon A^{(0)}+\frac{1}{k}\beta A^{(1)}$. The matrix $A$ is the transition matrix of $k$ particles moving on a circle of size $n$ independently with kernel $\frac{1}{k}\alpha\delta_{-1}+\epsilon\delta_{0}+\frac{1}{k}\beta\delta_{1}$ and killed when intersecting. Then, setting 
$$\alpha_{I_0}=\frac{1}{k}\alpha\overline{e_{1}}(\xi_{n}(I_0)))+\epsilon+\frac{1}{k}\beta e_{1}(\xi_{n}(I_0)))=(\alpha+\beta)\frac{\sin\left(\frac{k\pi}{n}\right)}{k\sin\left(\frac{\pi}{n}\right)}+\epsilon,$$
the matrix $P=\frac{1}{\alpha_{I_0}}\Delta^{-1}A\Delta$
is the transition matrix of the Doob's h-transform of the killed process, which is a natural form of conditioning the particles to never intersect each other. Remark that we can write 
$$P=p_{-1}Q^{(-1)}+p_0Q^{(0)}+p_{1}Q^{(1)}$$
with 
$$p_{-1}=\frac{\alpha \overline{e_{1}}(\xi_{n}(I_0)))}{k\alpha_{I_0}}=\frac{\alpha}{\alpha+\beta+\epsilon\frac{k\sin\left(\frac{\pi}{n}\right)}{\sin\left(\frac{k\pi}{n}\right)}}, \quad p_{0}=\frac{\epsilon}{(\alpha+\beta)\frac{\sin\left(\frac{k\pi}{n}\right)}{k\sin\left(\frac{\pi}{n}\right)}+\epsilon},\quad p_{1}=\frac{\beta }{\alpha+\beta+\epsilon\frac{k\sin\left(\frac{\pi}{n}\right)}{\sin\left(\frac{k\pi}{n}\right)}}.$$

 Set $p(n)=p$ for $n\geq 1$. Then, $p$ satisfies Assumptions~\ref{ass:far-k}, \ref{ass:X-subgaussian} and~\ref{ass:X-aperiodicity}. Let $\theta\in]0,1[$. Then, by Theorem \ref{thm:main}, as $k,n\rightarrow +\infty$ with $\frac{k}{n}\rightarrow \theta$,
$$\liminf_{s\rightarrow -\infty}\liminf_{n\rightarrow +\infty}\mathcal{D}\left(\frac{1}{\gamma}(\log k-s)\right)=1\text{ and }\limsup_{s\rightarrow +\infty}\limsup_{n\rightarrow +\infty}\mathcal{D}\left(\frac{1}{\gamma}(\log k+s)\right)=0$$
with, using Proposition \ref{prop:spectral-gap} below,
$$\gamma=2\left(1+O\left(\frac{1}{n}\right)\right)\frac{\pi^2}{n^2}(p_{-1}+p_{1}).$$
Since 
$$p_{-1}+p_{1}=\frac{\alpha+\beta}{\alpha+\beta+\frac{\theta\pi}{\sin\left(\theta \pi\right)}}\left(1+O\left(\frac{1}{n}\right)\right),$$
we deduce that for all $\epsilon>0$, the mixing time $t_{\epsilon}=\min(t\geq 1, \mathcal{D}(t)\leq \epsilon)$ is asymptotically given by 
$$t_{\epsilon}\sim\frac{n^2\log n}{2\pi^2}\left(1+\frac{\theta\pi}{(\alpha+\beta)\sin(\theta \pi)}\right).$$
Remark that the mixing time only depends on $\alpha+\beta$. The mixing behaviour of the conditioned ASEP is thus very different from the reflected ASEP, for which the mixing time depends on both $\alpha$ and $\beta$. In the specific case of TASEP, where $\alpha=0$, there is still a cut-off phenomenon in the conditioned version while it has been proven in that there is no cut-off in the reflected version, see \cite{schmid2022mixing}.

\subsubsection{Periodic dimer model on the hexagonal lattice}\label{subsec:dimer} 
In this section, we give an application of our results to the case of dimer models on the hexagonal lattice in the case of an infinite cylinder. See \cite{kenyon} for a general introduction to dimer models and more specifically \cite[Section 4.2]{kenyon} for the hexagonal model we are considering. The relation between this model and non-colliding random walk is known for a long time, see \cite{wu1968remarks}. In our case, the corresponding random walk is called the noncolliding Bernoulli random walk on the circle, see \cite{konig2002non} (see also \cite{gorin2019universality} for the non circular version): this particular random walk is included in the class of Markov processes of the present paper and corresponds to the choice of probability distribution $p_{\ell}=\binom{k}{\ell}p^\ell(1-p)^{k-\ell}$ for some $0\leq p\leq 1$. .

Let us briefly recall the definition of the dimer model on the hexagonal lattice and its relation to the non-colliding Bernoulli random walk. Denote by $H_{n,m}$ the doubly periodic hexagonal lattice of length $m$ and width $n$, whose set of vertices is given by $V_{n,m}=\{ai+be^{-i\pi/6}, a,b\in \mathbb{Z},a+b\not=-1\mod 3\}/\{n\sqrt{3}\mathbb{Z}+m\frac{3}{2}i\mathbb{Z}\}$ and there is an edge from $v$ to $w$ is $\vert v-w\vert=1$. We denote by $E$ the edge set of $H_{n,m}$ and each $e=\{v,w\}\in E$ has a natural orientation $\ell\in\{0,1,2\}$ given by the 
condition of $w-v\in\mathbb{R}e^{i\pi/2+2\ell i\pi/3}$. For $\ell\in\{0,1,2\}$, let us denote by $E_\ell$ the subset of edges of orientation $\ell$.

A dimer configuration on $H_{n,m}$ is a subset $\mathcal{E}\subset E$ such that each $v\in V$ is adjacent to exactly one element of $\mathcal{E}$. For any vector $a=(a_{0},a_{1},a_{2})\in\mathbb{R}_{\geq 0}$, one associates to each dimer configuration $\mathcal{E}\subset E$ a weight 
$$w(\mathcal{E})=\prod_{\ell\in\{0,1,2\}}a_\ell^{\#\mathcal{E}\cap E_{\ell}}.$$
This turns the set $\Sigma$ of dimer configurations on $H_{n,m}$ into a probability space with 
$$\mathbb{P}_{a}(\mathcal{E})=\frac{w(\mathcal{E})}{Z_{n,m}(a)},$$
where $Z_{n,m}(a)$ is the partition function of the system.

For $0\leq t\leq m$, denote by $V_{t}\subset V$ the slice 
$$V_t=\left\{v\in V, \Im v=\frac{3}{2}t\right\}=\frac{3t}{2}i+\sqrt{3}(\mathbb{Z}/n\mathbb{Z}).$$
By the latter description, each set $V_{t}$ is isomorphic to $\mathbb{Z}/n\mathbb{Z}$ by an isomorphism $\Phi$ sending $\frac{3t}{2}i+(r\mod[n])\sqrt{3}$ to $(r+t)[n]$.

Let us denote by $E_t$ the set of edges of type $1$ or $2$ adjacent to elements of  $V_t$. The graph structure yields that for any dimer configuration $\mathcal{E}$, the number 
$$k=\mathcal{E}\cap E_t$$
is independent of $1\leq t\leq m$. Moreover, the full configuration $\mathcal{E}$ is uniquely defined by the data of $\bigcup_{t=1}^mE_t$. In particular, if we set $\mathcal{J}_{t}\subset V_t$ the subset of vertices adjacent to edges of $E_t$, then the sequence $(\mathcal{J}_t)_{0\leq t\leq m-1}$ fully determines $\mathcal{E}$. Since $\#V_t=k$ for all $0\leq t\leq m-1$, the sequence $\Psi(\mathcal{E}):=(J_t)_{0\leq t\leq m-1}=(\Phi(\mathcal{J}_t))_{0\leq t\leq m-1}$ is a process on $B_{k,n}$.

For any two configurations $I_1,I_2\in B_{k,n}$, one has 
$$\sum_{\substack{\mathcal{E}\in\Sigma\\ \Psi(\mathcal{E})_0=I_1,\Psi(\mathcal{E})_t=I_2}}w(\mathcal{E})= \left[A^t\right]_{I_1,I_2}\left[A^{m-t}\right]_{I_2,I_1},$$
with $A=\sum_{\ell=0}^{k}a_1^{\ell}a_2^{k-\ell}A^{(\ell)}$. Then, as $m$ goes to infinity, an asymptotic computation yields that 
$$\mathbb{P}_a(J_0=I_1,J_t=I_2)=(P^T)^t_{I_1I_2},$$
where $P=\frac{1}{\alpha_{I_0}}\Delta^{-1}A\Delta$.

\begin{corollary}
Let $\epsilon >0,\eta,\delta\in(0,1)$. Then, uniformly on $k\in [n\eta,n(1-\eta)]$ and $a_1,a_2\in\left[\delta,\frac{1}{\delta}\right]$,
$$t_{\epsilon}\underset{n,k\rightarrow +\infty}{\sim}\frac{nr\log n}{4\pi a_1\sin\left(\frac{k\pi}{n}-\theta\right)},$$
where $r,\theta$ are such that $a_2+a_1e^{\frac{i\pi}{n}}=re^{i\theta_0}$.
\end{corollary}
Before proving this corollary, let us briefly explain its meaning in terms of dimer models. Since the work of Cohn, Keynon and Propp \cite{cohn2001variational}, it has been conjectured that the statistic behavior of the dimer model on a simply connected bounded domain locally looks like a dimer model on a torus for some special choice of weights $(a_0,a_1,a_2)$ depending on the location in the domain in a conformal way. This conjecture has been recently proven by Aggarwal, \cite{aggarwal2023universality}. In our setting, we provide on a toy model a threshold after which all information from the boundary condition is lost in the local statistical behavior of the model. It would be interesting to see whether such phenomenon appears on dimer models on the complement of a simply connected domain. After rescaling, it is expected that there is a curve at distance approximately $\log n$ from the inner domain after which the dimer model locally behave like a given Gibbs measure independently of the shape of the inner domain.
\begin{proof}

A computation yields 
$$\alpha_{I_0}=\sum_{\ell=0}^ke_{\ell}(\xi(I_0))a_1^{\ell}a_2^{k-\ell}=\prod_{j=1}^k\left(a_2+a_1\xi(I_0)_j\right),$$
and the kernel $P$ can be written $P=\sum_{\ell=0}^kp_{\ell}Q^{(\ell)}$ with 
$p_{\ell}=\frac{e_{\ell}(\xi(I_0))}{\alpha_{I_0}}a_1^{\ell}a_2^{k-\ell}$. 

As a first consequence, 
$$\lambda_{I_1}=\frac{\sum_{\ell=0}^ke_{\ell}(\xi(I_1))a_1^{\ell}a_2^{k-\ell}}{\sum_{\ell=0}^ke_{\ell}(\xi(I_0))a_1^{\ell}a_2^{k-\ell}}=\frac{\prod_{j=1}^k\left(a_2+a_1\xi(I_1)_j\right)}{\prod_{j=1}^k\left(a_2+a_1\xi(I_0)_j\right)}=\frac{a_2+a_1e^{\frac{i\pi(k+1)}{n}}}{a_2+a_1e^{\frac{i\pi(k-1)}{n}}},$$
since $\xi(I_1)_j=\xi(I_0)_j$ for $2\leq j\leq k$ and $\xi(I_1)_1=e^{\frac{i\pi(k+1)}{n}}$ while $\xi(I_0)_1=e^{\frac{i\pi(k-1)}{n}}$.
Hence, writing $a_2+a_1e^{i\pi k/n}=re^{i\theta_0}$,
\begin{align*}
\gamma=&1-\left\vert\frac{a_2+a_1e^{\frac{i\pi(k+1)}{n}}}{a_2+a_1e^{\frac{i\pi(k-1)}{n}}}\right\vert\\
=&\frac{4\pi}{n}\frac{a_1}{r}\sin\left(\frac{k\pi}{n}-\theta_0\right)+o\left(\frac{1}{n}\right).
\end{align*}

Next, the Fourier transform of $p$ is given by 
\begin{align*}
\Phi_{p}(\theta)=&\sum_{\ell=0}^ke^{i\ell \theta}\frac{e_{\ell}(\xi(I_0))}{\alpha_{I_0}}a_1^{\ell}a_2^{k-\ell}=\frac{\prod_{j=1}^k\left(a_2+a_1e^{i \theta}\xi(I_0)_j\right)}{\prod_{j=1}^k\left(a_2+a_1\xi(I_0)_j\right)}=\prod_{j=1}^k\frac{a_2+a_1e^{i\theta}\xi(I_0)_j}{a_2+a_1\xi(I_0)_j}.
\end{align*}
Taking the logarithm and a Riemann integral for the $\mathcal{C}^2$-function $u\mapsto \log\left(\frac{a_2+a_1e^{i(\theta+\pi u)}}{a_2+a_1e^{i\pi u}}\right)$ (whose first and second derivatives are bounded uniformly on $(u,\theta)\in [-(1-\eta),1-\eta]\times \left([-\pi\eta/2,\pi\eta/2]+i\mathbb{R}\right)$) yields 
$$\log \Phi_p(\theta)=k\left(\int_{-\frac{k}{n}}^{\frac{k}{n}}\log\left(\frac{a_2+a_1e^{i(\theta+\pi u)}}{a_2+a_1e^{i\pi u}}\right)du+O\left(\frac{1}{k^2}\right)\right).$$
In particular, writing $\tau=\frac{k}{n}$,
$$\Phi_p(\theta)=\exp\left(kF(\tau,\theta)+O\left(\frac{1}{k}\right)\right),$$
with $F(\tau,\theta)=\int_{-\tau}^{\tau}\log\left(\frac{a_2+a_1e^{i(\theta+\pi u)}}{a_2+a_1e^{i\pi u}}\right)du$ being $C^\infty$ on $[\eta,1-\eta]\times ]-\pi\eta/2,\pi\eta/2[+i\mathbb{R}$ and $O\left(\frac{1}{k}\right))$ being uniform on this range of parameters. 

Recall that $\log \Phi_p(\theta)=i\mathbb{E}[X]\theta-\frac{\Var(X)}{2}\theta^2+O\left( \theta^3\right)$, with $O\left(\theta^3\right)$ only depending on the third derivative of $\log \Phi_p$ on $[0,\theta]$. Hence, we deduce that 
$$\mathbb{E}[X]=\frac{k}{i}\partial_{\theta} F(\tau,0)+O\left(\frac{1}{k}\right),$$
and 
$$\Phi_{X-\mathbb{E}[X]}(\theta)= \exp\left(k\left(F(\tau,\theta)-\partial_{\theta}F(\tau,0)\right)+O\left(\frac{1}{k}\right)\right).$$
and there exists a band $[-\theta_0,\theta_0]+i\mathbb{R}$ only depending on $\eta$ on which
$$\left\vert \Phi_{X-\mathbb{E}[X]}(i\theta)\right\vert \leq \exp\left(k\partial^2_\theta F(0,\theta)\theta^2+O\left(\frac{1}{k}\right)\right)\leq \exp\left(K^2\mathbb{E}[X]\theta^2\right),$$
with $K^2=\sup_{\tau\in[\eta,1-\eta]}\frac{\partial_\theta F(\tau,0)}{\partial^2_\theta F(\tau,0)}$. This implies that $X-\mathbb{E}[X]$ is $K\sqrt{\mathbb{E}[X]}-$ sub-Gaussian, up to a universal numeric constant, see \cite[Proposition 2.6.1]{vershynin2018high}, and Assumption \ref{ass:X-subgaussian} holds.

Next, remark that $\Re F(\tau,\theta)$ is decreasing on $[0,\pi]$. Hence, the Taylor expansion of $F(\tau,\theta)$ in $\theta$ around $(\tau,0)$ yields Assumption \ref{ass:X-aperiodicity} with a constant independent of $\tau \in [\eta,1-\eta]$.
\end{proof}

\section{Spectral gap}
\label{sec:spectral-gap}

Let us first present a brief picture of the eigen-elements of the matrices $A^{(\ell)}$, $Q^{(\ell)}$ and $P$. 

Recall that for $I\in B_{k,n}$, we let $\xi(I)=(\omega^{I(1)-\frac{k-1}{2}},\dots, \omega^{I(k)-\frac{k-1}{2}})$, where $\omega=e^{i2\pi/n}$, and that
\[
 I_0= (k-1,\dots,1,0)\qquad\text{and}\qquad I_1=  I_0+(1,0,\dots,0)\, .
 \]

By~\cite[Lemma 4.3]{guilhot2022quantum}, all the matrices $A^{(\ell)}$ share a common orthogonal basis of eigenvectors $(\varphi_J)_{J\in B_{k,n}}$ (for the standard scalar product), given by
\[
\varphi_J(I)=\overline{S_I(\xi(J))}\, ,
\]
where $S_I$ denotes the Schur function 
$$S_I(x_1,\ldots,x_k)=\frac{\det\left(x_i^{I_j}\right)}{\det\left(x_i^{j-1}\right)},$$
and the corresponding eigenvalues for $A^{(\ell)}$, $\ell\geq 0$ are
\[
\alpha^{(\ell)}_J=e_\ell(\xi(J))
\]
where $e_\ell$ denotes the $\ell^{\text{th}}$ elementary symmetric polynomial. The eigenvalue $\alpha_{J}^{(\ell)}$ for $-k\leq \ell\leq -1$ is then simply 
$$\alpha_{J}^{(\ell)}=\overline{\alpha_{J}^{(-\ell)}}.$$

Letting $\Delta$ be the diagonal matrix with diagonal entries given by the eigenvector $\varphi_{I_0}$, one may then write $Q^{(\ell)}$ as
\[
Q^{(\ell)}=\frac{1}{\alpha^{(\ell)}_{I_0}}\Delta^{-1}A^{(\ell)}\Delta\, .
\]
We have
\begin{equation}\label{eq:Delta-phi}
Q^{(\ell)}\Delta^{-1}\varphi_J=\frac{1}{\alpha^{(\ell)}_{I_0}}\Delta^{-1}A^{(\ell)}\varphi_J=\frac{\alpha^{(\ell)}_{J}}{\alpha^{(\ell)}_{I_0}}\Delta^{-1}\varphi_J\, .
\end{equation}
Moreover, it is not hard to check that the eigenvectors $\Delta^{-1}\varphi_J$ are orthogonal  for the weighted scalar product $\langle\cdot,\cdot\rangle_\mu$. After normalization, one obtains an orthonormal (for$\langle\cdot,\cdot\rangle_\mu$)  basis of eigenvectors $(f_J)$ of $Q^{(\ell)}$ given by
\[
f_J(I)=\frac{|V(\xi(J))|}{|V(\xi(I_0))|}\frac{\overline{S_I(\xi(J))}}{\overline{S_I(\xi(I_0))}}=\frac{|V(\xi(J))|}{|V(\xi(I_0))|}\frac{\overline{S_J(\xi(I))}}{\overline{S_J(\xi(I_0))}}\, ,
\]
with the corresponding eigenvalues
\begin{equation}\label{eq:spectral_gap_symetry}
\lambda^{(\ell)}_J=\frac{e_\ell(\xi(J))}{e_\ell(\xi(I_0))}\quad \text{ and} \quad \lambda^{(-\ell)}_J=\overline{\lambda^{(\ell)}_J}\, , \quad \ell\geq 0\, .
\end{equation}
The eigenvalues of $P=\sum_{\ell=-k}^{k} p_\ell Q^{(\ell)}$ are then
\[
\lambda_J=\sum_{\ell=-k}^{k}p_\ell\frac{e_\ell(\xi(J))}{e_\ell(\xi(I_0))}\, \cdot 
\]
We now prove several properties regarding the spectral gap $\gamma=1-|\lambda_{I_1}|$.
\begin{proposition}\label{prop:spectral-gap}
We have
\[
\lambda_{I_1}=1-\sum_{\ell=0}^{k}p_\ell\omega^{k-\ell}(1-\omega)\frac{1-\omega^\ell}{1-\omega^k}  -\sum_{\ell=-1}^{-k}p_{\ell}\omega^{-k-\ell}(1-\omega^{-1})\frac{1-\omega^\ell}{1-\omega^{-k}}\, .
\]
Moreover, under Assumptions~\ref{ass:number-particles}, \ref{ass:far-k} and~\ref{ass:X-subgaussian}, we have
\[
\gamma=2\left(1+O\left(\frac{1}{n}\right)\right)\frac{\sin\left(\frac{\pi}{n}\right)}{\sin\left(\frac{k\pi}{n}\right)}\sum_{\ell=1}^k(p_{\ell}+p_{-\ell})\sin\left(\frac{\ell\pi}{n}\right)\sin\left(\frac{(k-\ell)\pi}{n}\right)\, \cdot 
\]
\end{proposition}
This result may be proven for much weaker hypotheses than Assumption \ref{ass:X-subgaussian}. For example, a moment condition $k-\frac{\E[X^2]}{\E[ X ]}\gg \log(k)$ suffices to get an asymptotic formula up to an error $o\left(\frac{1}{\log k}\right)$, which is small enough to get our asymptotic results.
\begin{proof}[Proof of Proposition~\ref{prop:spectral-gap}]
For all $\ell\in \llbracket 0,k\rrbracket$, we have
\begin{align*}
\lambda^{(\ell)}_{I_1}&=\frac{e_\ell(\xi(I_1))}{e_\ell(\xi(I_0))}= \frac{e_\ell(1,\omega,\dots,\omega^{k-2},\omega^k)}{e_\ell(1,\omega,\dots,\omega^{k-1})}\\
&= \frac{e_\ell(1,\omega,\dots,\omega^{k-1})+(\omega^k-\omega^{k-1})e_{\ell-1}(1,\omega,\dots,\omega^{k-2})}{e_\ell(1,\omega,\dots,\omega^{k-1})}\\
&=1-\omega^{k-1}(1-\omega)\frac{e_{\ell-1}(1,\omega,\dots,\omega^{k-2})}{e_\ell(1,\omega,\dots,\omega^{k-1})}\, \cdot 
\end{align*}
Using the $q$-identity
\[
e_\ell(1,q,\dots,q^{k-1})=q^{\frac{\ell(\ell-1)}{2}}{k\choose \ell}_q=q^{\frac{\ell(\ell-1)}{2}}\frac{(1-q^k)\dots (1-q^{k-\ell+1})}{(1-q)\dots(1-q^\ell)}\, ,
\]
we obtain
\[
\frac{e_{\ell-1}(1,\omega,\dots,\omega^{k-2})}{e_\ell(1,\omega,\dots,\omega^{k-1})}=\omega^{-\ell+1}\frac{1-\omega^\ell}{1-\omega^k}\, \cdot
\]
Hence, for $\ell\in\llbracket 0,k\rrbracket$,
\begin{equation}\label{eq:expression_lambda_I1_l}
\lambda^{(\ell)}_{I_1}=1-\omega^{k-\ell}(1-\omega)\frac{1-\omega^\ell}{1-\omega^k}\, .
\end{equation}
In particular, $\lambda^{(0)}_{I_1}=1$. And for $\ell\in\llbracket -k,-1\rrbracket$
\[\lambda^{(\ell)}_{I_1}=\overline{\lambda^{(-\ell)}_{I_1}}=1-\omega^{-k-\ell}(1-\omega^{-1})\frac{1-\omega^\ell}{1-\omega^{-k}}\, .
\]
Hence,
\begin{equation}
\lambda_{I_1}=1-\sum_{\ell=1}^{k}p_\ell\omega^{k-\ell}(1-\omega)\frac{1-\omega^\ell}{1-\omega^k} -\sum_{\ell=-1}^{-k}p_{\ell}\omega^{-k-\ell}(1-\omega^{-1})\frac{1-\omega^\ell}{1-\omega^{-k}}\, \cdot 
\end{equation}
Now, using that $\omega^j-1=2i\omega^{j/2}\sin\left(\frac{j\pi}{n}\right)$ and writing $s_{k,n}=\dfrac{\sin\left(\frac{\pi}{n}\right)}{\sin\left(\frac{k\pi}{n}\right)}$, we see that
\[
\lambda_{I_1}= 1+2is_{k,n}\sum_{\ell=1}^{k}p_\ell\omega^{\frac{k-\ell+1}{2}}\sin\left(\tfrac{\ell\pi}{n}\right)+2is_{k,n}\sum_{\ell=-1}^{-k}p_\ell\omega^{\frac{-k-\ell-1}{2}}\sin\left(\tfrac{\ell\pi}{n}\right)\, , 
\]
so that
\begin{align*}
|\lambda_{I_1}|^2&= \left(1-2s_{k,n}\sum_{\ell=1}^{k}p_\ell\sin\left(\tfrac{(k-\ell+1)\pi}{n}\right)\sin\left(\tfrac{\ell\pi}{n}\right)-2s_{k,n}\sum_{\ell=-1}^{-k}p_\ell\sin\left(\tfrac{(-k-\ell-1)\pi}{n}\right)\sin\left(\tfrac{\ell\pi}{n}\right)\right)^2\\
& \hspace{2cm}+\left(2s_{k,n}\sum_{\ell=1}^{k}p_\ell\cos\left(\tfrac{(k-\ell+1)\pi}{n}\right)\sin\left(\tfrac{\ell\pi}{n}\right)+2s_{k,n}\sum_{\ell=-1}^{-k}p_\ell\cos\left(\tfrac{(-k-\ell-1)\pi}{n}\right)\sin\left(\tfrac{\ell\pi}{n}\right)\right)^2\\
&= 1-4s_{k,n}\E\left[\sin\left(\frac{(k-|X|+1)\pi}{n}\right)\sin\left(\frac{|X|\pi}{n}\right)\right]+4s_{k,n}^2\E\left[\sin\left(\frac{(k-|X|+1)\pi}{n}\right)\sin\left(\frac{|X|\pi}{n}\right)\right]^2\\
& \hspace{1cm}+4s_{k,n}^2\E\left[\cos\left(\frac{(k-|X|+1)\pi}{n}\right)\sin\left(\frac{X\pi}{n}\right)\right]^2\\
&= 1-4s_{k,n}\E\left[\sin\left(\frac{(k-|X|)\pi}{n}\right)\sin\left(\frac{|X|\pi}{n}\right)\right]+O\left(\frac{\E[|X|]}{n^3}\right)\, .
\end{align*}
Since $\frac{\sin(\theta)x}{\theta}\leq \sin(x)\leq x$ for $x\in[0,\theta]$, we have
$$s_{k,n}\E\left[\sin\left(\frac{(k-|X|)\pi}{n}\right)\sin\left(\frac{|X|\pi}{n}\right)\right]\asymp \frac{\E[\vert X\vert(k-\vert X\vert)\vert]}{n^3},\, $$
where $\asymp$ means equality up to a numeric constant only depending on $\eta>0$ such that $\frac{k}{n}\in ]\eta,1-\eta[$. By Assumption \ref{ass:X-subgaussian}, we have $\E[X^2]\leq \E[X]^2 +2K_g^2\E[|X|]$. Hence,
\[
\E[|X|(k-|X|)]\geq k\E[|X|]-\E[X]^2-2K_g^2\E[|X|]\geq \E[|X|]\left(k-|\E[X]|-2K_g^2\right)\, .
\]
Using Assumption~\ref{ass:far-k}, we have $k-|\E[X]|\geq \delta k$, hence $\E[|X|]=O\left(\frac{\E[|X|(k-|X|)]}{n}\right)$. Therefore,
\[
|\lambda_{I_1}|^2=1-4s_{k,n}\left(1+O\left(\frac{1}{ n}\right)\right)\sum_{\ell=1}^{k}(p_\ell+p_{-\ell})\sin\left(\tfrac{(k-\ell)\pi}{n}\right)\sin\left(\tfrac{\ell\pi}{n}\right)\, ,
\]
and
\[
\gamma=2s_{k,n}\left(1+O\left(\frac{1}{ n}\right)\right)\sum_{\ell=1}^{k}(p_\ell+p_{-\ell}) \sin\left(\tfrac{\ell\pi}{n}\right)\sin\left(\tfrac{(k-\ell)\pi}{n}\right)\, .
\]
\end{proof}
The proof of the latter proposition shows that
\begin{equation}\label{eq:average_gamma}
\gamma=\left(1+O\left(\frac{1}{ n}\right)\right)\sum_{\ell=1}^k(p_{\ell}+p_{-\ell})\gamma_{\ell}=\left(1+O\left(\frac{1}{n}\right)\right)\mathbb{E}\left[\gamma_{X}\right],
\end{equation}
where, for $-k \leq \ell\leq k$, $\gamma_{\ell}=1-\left\vert \lambda_{I_1}^{(\ell)}\right\vert$ is the spectral gap associated to $Q^{(\ell)}$ and satisfies
\begin{equation}\label{eq:expression_gammal}
\gamma_{\ell}=2\frac{\sin\left(\frac{\pi}{n}\right)\sin\left(\frac{|\ell|\pi}{n}\right)\sin\left(\frac{(k-|\ell|)\pi}{n}\right)}{\sin\left(\frac{k\pi}{n}\right)}+O\left(\frac{|\ell|}{n^3}\right)\, .
\end{equation}
Since the right hand-side is a continuous function of $\ell$, one expects that the sub-gaussian property of $X$ translates into a sub-gaussian property of $\gamma_X$. 
\begin{lemma}\label{lem:gamma_sub_gaussian}
Under Assumptions~\ref{ass:number-particles}, \ref{ass:far-k} and~\ref{ass:X-subgaussian}, there exists $L>0$ only depending on $\eta$ such that $(\gamma_X-\gamma)$ is sub-gaussian with parameter $\frac{LK_g\sqrt{\gamma}}{k}$.
\end{lemma}
\begin{proof}
Remark first that $\gamma_{\ell}=\gamma_{-\ell}$ for $0\leq \ell\leq k$ by \eqref{eq:spectral_gap_symetry}. Thus, $\gamma_{X}=\gamma_{\vert X\vert}$ and for all $\theta>0$,
$$\E\left[\exp\left(\frac{(\gamma_X-\E[\gamma_X])^2}{\theta^2}\right)\right]=\E\left[\exp\left(\frac{(\gamma_{|X|}-\E[\gamma_{|X|}])^2}{\theta^2}\right)\right].$$
Let $Y$ be an independent copy of the random variable $X$. By Jensen Inequality, we have, for all $\theta>0$,
\begin{align*}
\E\left[\exp\left(\frac{(\gamma_{|X|}-\E[\gamma_{|X|}])^2}{\theta^2}\right)\right]=\E\left[\exp\left(\frac{(\gamma_{|X|}-\E[\gamma_{|Y|}])^2}{\theta^2}\right)\right]\leq \E\left[\exp\left(\frac{(\gamma_{|X|}-\gamma_{|Y|})^2}{\theta^2}\right)\right]\, .
\end{align*}
By expression~\eqref{eq:expression_gammal}, the function $\ell\mapsto \gamma_\ell$ is Lipschitzian with constant $\tfrac{L}{n^2}$, for some constant $L$ depending only on $\eta$. Hence,
\begin{align*}
\E\left[\exp\left(\frac{(\gamma_X-\E[\gamma_X])^2}{\theta^2}\right)\right]&\leq \E\left[\exp\left(\frac{L^2(|X|-|Y|)^2}{n^4\theta^2}\right)\right]\\
&\leq \E\left[\exp\left(\frac{2L^2((|X|-\E[|X|])^2+(|Y|-\E[|Y|])^2)}{n^4\theta^2}\right)\right]\\
&=\E\left[\exp\left(\frac{2L^2(|X|-\E[|X|])^2}{n^4\theta^2}\right)\right]^2\leq \E\left[\exp\left(\frac{4L^2(|X|-\E[|X|])^2}{n^4\theta^2}\right)\right]\, .
\end{align*}
By Assumption~\ref{ass:X-subgaussian}, the latter quantity is less than~$2$ for $\theta\geq \tfrac{2LK_g\sqrt{\E[|X|]}}{n^2}$. The conclusion of the Lemma then follows from the facts that $\gamma-\E[\gamma_{X}]=O\left(\tfrac{\gamma}{n}\right)$ and that, under Assumptions~\ref{ass:number-particles}, ~\ref{ass:far-k} and~\ref{ass:X-subgaussian}, we have
\begin{equation}\label{eq:equiv_gamma_expect}
\gamma\asymp \frac{\E[|X|(k-|X|)]}{n^3}\asymp \frac{\E[|X|]}{n^2}\, \cdot 
\end{equation}
\end{proof}
\section{Lower bound in Theorem~\ref{thm:main}}\label{sec:lower-bound}

The goal of this section is to prove the following proposition, which yields the lower bound in Theorem~\ref{thm:main}.

\begin{proposition}\label{prop:lower-bound}
Under Assumptions~\ref{ass:number-particles}, ~\ref{ass:far-k} and~\ref{ass:X-subgaussian}, the total-variation distance of $P$ satisfies
\[
\liminf_{n\to +\infty}\cD\left(\frac{\log n+s}{\gamma}\right)\,\underset{s\to -\infty}{\rightarrow}\, 1\, .
\]
\end{proposition}

Let us first recall a classical lower bound on the total-variation distance. In what follows, the variance of a complex-valued random variable $Z$ is $\Var(Z)=\E[|Z-\E[Z]|^2]$.

\begin{lemma}\label{lem:lower-bound-tv}
Let $(X_t)_{t\in\N}$ be an irreducible Markov chain on $\Omega$. For any function $\Phi:\Omega\to \mathbb{C}$, we have
\[
\cD(t)\geq \max_{x\in\Omega}\frac{\left|\E_x\left[\Phi(X_t)\right]-\E\left[\Phi(X)\right]\right|^2}{2\Var_x\left(\Phi(X_t)\right)+ 2\Var\left(\Phi(X)\right)+\left|\E_x\left[\Phi(X_t)\right]-\E\left[\Phi(X)\right]\right|^2}\, ,
\]
where $X$ is a random variable with the stationary distribution of the chain. Moreover, if $\Phi$ is chosen as an eigenvector associated to an eigenvalue $\lambda\neq 1$, then we have
\[
\left|\E_x\left[\Phi(X_t)\right]-\E\left[\Phi(X)\right]\right|^2=|\lambda|^{2t}|\Phi(x)|^2\, .
\]
\end{lemma}
\begin{proof}[Proof of Lemma~\ref{lem:lower-bound-tv}]
Using the characterization of the total variation distance by couplings, we have
\[
\cD_x(t) \geq \inf \P_{x}(X_t\neq X)\geq \inf\P_{x}\left(\Phi(X_t)\neq \Phi(X)\right)\, ,
\]
where the infimum is taken over all the possible couplings of $X_t$ (when starting from $X_0=x$) and $X$. Using Paley-Zygmund Inequality, we have
\begin{align*}
\P_x\left(\Phi(X_t)\neq \Phi(X)\right)&= \P_x\left( \left|\Phi(X_t)-\Phi(X)\right|>0\right)\\
&\geq \frac{\E_x\left[\left|\Phi(X_t)-\Phi(X)\right|\right]^2}{\E_x\left[\left|\Phi(X_t)-\Phi(X)\right|^2\right]}\,.
\end{align*}
For the numerator, we use
\[
\E_x\left[\left|\Phi(X_t)-\Phi(X)\right|\right]\geq \left|\E_x\left[\Phi(X_t)\right]-\E\left[\Phi(X)\right]\right|\, .
\]
For the denominator, letting $Z_t=\Phi(X_t)-\E_x[\Phi(X_t)]$ and $Z=\Phi(X)-\E[\Phi(X)]$, we have
\begin{align*}
\E_x\left[\left|\Phi(X_t)-\Phi(X)\right|^2\right] &= \E_x\left[\left|Z_t-Z +\E_x[\Phi(X_t)]-\E[\Phi(X)]\right|  ^2\right]\\
&=\Var_x\left(\Phi(X_t)\right)+ \Var\left(\Phi(X)\right)+\left|\E_x\left[\Phi(X_t)\right]-\E\left[\Phi(X)\right]\right|^2-2\E_x\left[\mathfrak{Re}\left(Z_tZ\right)\right]\\
&\leq 2\Var_x\left(\Phi(X_t)\right)+ 2\Var\left(\Phi(X)\right)+\left|\E_x\left[\Phi(X_t)\right]-\E\left[\Phi(X)\right]\right|^2\, ,
\end{align*}
where for the last inequality, we used that
\begin{align*}
-\E_x\left[\mathfrak{Re}\left(Z_tZ\right)\right]&\leq \E_x\left[\left|Z_tZ\right|\right]\leq \sqrt{\Var_x\left(\Phi(X_t)\right)\Var\left(\Phi(X)\right)}\leq\frac{1}{2}\left( \Var_x\left(\Phi(X_t)\right)+\Var\left(\Phi(X)\right)\right)\, .
\end{align*}
We obtain that, for all couplings,
\[
\P_x\left(\Phi(X_t)\neq \Phi(X)\right)\geq \frac{\left|\E_x\left[\Phi(X_t)\right]-\E\left[\Phi(X)\right]\right|^2}{2\Var_x\left(\Phi(X_t)\right)+ 2\Var\left(\Phi(X)\right)+\left|\E_x\left[\Phi(X_t)\right]-\E\left[\Phi(X)\right]\right|^2}\, ,
\]
and since the right-hand side only depends on the marginal distributions of $X_t$ and $X$, the same lower bound applies to $\cD_x(t)$. Taking the maximum over $x\in\Omega$, yields the desired lower bound on $\cD(t)$.

For the second statement of the lemma, note that if $\Phi$ is an eigenvector associated to an eigenvalue $\lambda\neq 1$, then
\[
\E\left[\Phi(X_{t+1})\given\cF_t\right]=\lambda\Phi(X_t)\, .
\]
Hence, by induction, $\E_x\left[\Phi(X_t)\right]=\lambda^t\Phi(x)$. Moreover, $\E[\Phi(X)]=0$, since $\E[\Phi(X)]=\E[P\Phi(X)]=\lambda\E[\Phi(X)]$ and $\lambda\neq 1$.
\end{proof}

We will apply Lemma~\ref{lem:lower-bound-tv} for the kernel $P$ with the (unnormalized) eigenvector $\Phi=\Delta^{-1}\varphi_{I_1}$ introduced in~\eqref{eq:Delta-phi}. We may first note that, since $\Phi(I_0)=1$, the absolute difference between expectations when starting from $I_0$ is equal to $|\lambda_{I_1}|^t=(1-\gamma)^t$. It remains to control the variance. To this end, we will need the following lemma.

\begin{lemma}\label{lem:lambda_I2}
Let $I_2=I_0+(-1,0,\dots,0,1)$. We have
\[
\lambda_{I_2}=1-4\frac{\sin\left(\tfrac{\pi}{n}\right)}{\sin\left(\tfrac{(k-1)\pi}{n}\right)}\sum_{\ell=0}^{k} (p_\ell+p_{-\ell})\sin\left(\tfrac{\ell\pi}{n}\right)\sin\left(\tfrac{(k-\ell)\pi}{n}\right)\, \cdot
\]
In particular, under Assumptions~\ref{ass:number-particles}, \ref{ass:far-k} and~\ref{ass:X-subgaussian}, we have
\[
\lambda_{I_2}=1-2\left(1+O\left(\frac{1}{n}\right)\right)\gamma\, .
\]
\end{lemma}
\begin{proof}[Proof of Lemma~\ref{lem:lambda_I2}]
For all $\ell\in \llbracket 0,k\rrbracket$, we have
\begin{align*}
\lambda^{(\ell)}_{I_2}&= \frac{e_\ell(\omega^{-1},\omega,\dots,\omega^{k-2},\omega^k)}{e_\ell(1,\omega,\dots,\omega^{k-1})}\, .
\end{align*}
Using that 
$$e_\ell(\omega^{-1},\omega,\dots,\omega^{k-2},\omega^k)=e_\ell(1,\omega,\dots,\omega^{k-1})+(\omega^{-1}-1)e_{\ell-1}(\omega,\dots,\omega^{k-2})+(\omega^k-\omega^{k-1})e_{\ell-1}(\omega,\dots,\omega^{k-2}),$$ 
we have
\begin{align*}
\lambda^{(\ell)}_{I_2}&=1-(1-\omega^{-1}-\omega^k+\omega^{k-1})\frac{e_{\ell-1}(\omega,\dots,\omega^{k-2})}{e_\ell(1,\dots,\omega^{k-1})}\\
&= 1-\omega^{\ell-1}(1-\omega^{-1}-\omega^k+\omega^{k-1})\frac{e_{\ell-1}(1,\dots,\omega^{k-3})}{e_\ell(1,\dots,\omega^{k-1})}\\
&= 1-\omega^{\ell-1}(1-\omega^{-1}-\omega^k+\omega^{k-1})\frac{\omega^{\tfrac{(\ell-1)(\ell-2)}{2}}{k-2\choose \ell-1}_\omega}{\omega^{\tfrac{\ell(\ell-1)}{2}}{k\choose \ell}_\omega}\\
&= 1-(1-\omega^{-1}-\omega^k+\omega^{k-1})\frac{(1-\omega^\ell)(1-\omega^{k-\ell})}{(1-\omega^k)(1-\omega^{k-1})}\\
&= 1-\omega^{-\tfrac{k-1}{2}}(1-\omega^{-1}-\omega^k+\omega^{k-1})\frac{\sin\left(\tfrac{\ell\pi}{n}\right)\sin\left(\tfrac{(k-\ell)\pi}{n}\right)}{\sin\left(\tfrac{k\pi}{n}\right)\sin\left(\tfrac{(k-1)\pi}{n}\right)}\\
&=1-4\frac{\sin\left(\tfrac{\pi}{n}\right)}{\sin\left(\tfrac{(k-1)\pi}{n}\right)}\sin\left(\tfrac{\ell\pi}{n}\right)\sin\left(\tfrac{(k-\ell)\pi}{n}\right)\, ,
\end{align*}
where for the last equality we used that
\begin{align*}
\omega^{-\tfrac{k-1}{2}}(1-\omega^{-1}-\omega^k+\omega^{k-1})= (\omega^{1/2}-\omega^{-1/2})(\omega^{-k/2}-\omega^{k/2})= 4\sin\left(\tfrac{\pi}{n}\right)\sin\left(\tfrac{k\pi}{n}\right)\, .
\end{align*}
In particular, $\lambda_{I_2}^{(\ell)}=\lambda_{I_2}^{(-\ell)}$ for $\ell\in\llbracket -k,-1\rrbracket$. We thus have
\[
\lambda_{I_2}=1-4\frac{\sin\left(\tfrac{\pi}{n}\right)}{\sin\left(\tfrac{(k-1)\pi}{n}\right)}\sum_{\ell=0}^{k} (p_\ell+p_{-\ell})\sin\left(\tfrac{\ell\pi}{n}\right)\sin\left(\tfrac{(k-\ell)\pi}{n}\right)\, \cdot
\]
\end{proof}

\begin{proof}[Proof of Proposition~\ref{prop:lower-bound}]
Let $\Phi=\Delta^{-1}\varphi_{I_1}$ and recall that
\[
\Phi(I)=\frac{\overline{S_I(\xi(I_1))}}{\overline{S_I(\xi(I_0))}}=\frac{\overline{S_{I_1}(\xi(I))}}{\overline{S_{I_1}(\xi(I_0))}}\, ,
\]
Writing $\Var\left(\Phi(X_t)\right)=\E[|\Phi(X_t)|^2]-|\E[\Phi(X_t)]|^2$, we already know that $|\E[\Phi(X_t)]|^2=(1-\gamma)^{2t}$. Moreover,
\[
\E[|\Phi(X_t)|^2]=\frac{\E[S_{I_1}(\xi(X_t))\overline{S}_{I_1}(\xi(X_t))]}{|S_{I_1}(\xi(I_0))|^2}\, .
\]
Letting $I'_1=I_0+(-1,0,\dots,0)$ and using Pieri's formula, we have
\[
S_{I_1}\overline{S}_{I_1}=S_{I_1}S_{I'_1}=S_{I_2}+S_{I_0}\, ,
\]
where $ I_2 = I_0 +(1,0,\dots,0,-1)$. Since $S_{I_0}$ is constant equal to~$1$ and since $\dfrac{S_{I_2}}{S_{I_2}(\xi(I_0))}=\overline{\Delta^{-1}\varphi_{I_2}}$, we get
\[
\E[|\Phi(X_t)|^2]=\frac{1}{|S_{I_1}(\xi(I_0))|^2}\left( (\overline{\lambda}_{I_2})^tS_{I_2}(\xi(I_0))+1\right)\, .
\]
Using that $\lambda_{I_2}\in\mathbb{R}$ by Lemma~\ref{lem:lambda_I2} and that $S_{I_2}(\xi(I_0))=|S_{I_1}(\xi(I_0))|^2-1$, we have
\[
\E[|\Phi(X_t)|^2]\leq (\lambda_{I_2})^t+\frac{1}{|S_{I_1}(\xi(I_0))|^2}\, \cdot 
\]
In particular, letting $t$ tend to $+\infty$, we have, for $X\sim\mu$,
\[
\Var\left(\Phi(X)\right)\leq \frac{1}{|S_{I_1}(\xi(I_0))|^2}\, ,\, \cdot 
\]
which is less than $\tfrac{c}{k^2}$ for $n$ large enough and for $c>0$ some universal positive constant. Applying Lemma~\ref{lem:lower-bound-tv}, we see that for $n$ large enough and for all $t\in\N$,
\[
\cD(t)\geq \frac{(1-\gamma)^{2t}}{2(\lambda_{I_2})^t-(1-\gamma)^{2t}+\tfrac{4c}{k^2}}\, \cdot 
\]
And by Lemma~\ref{lem:lambda_I2}, under Assumption~\ref{ass:X-subgaussian} and if $t=\tfrac{\log k +s}{\gamma}$,
\[
\cD(t)\geq \frac{e^{-2s}}{e^{-2s}+4c}+o(1)\, \cdot 
\]
\end{proof}

\section{Upper bound in Theorem~\ref{thm:main}}\label{sec:upper-bound}
To prove the upper bound, we use the inequality 
\begin{equation}\label{eq:inequality_l2_l1}
4\cD(t)^2\leq \cD_2^2(t)\, ,
\end{equation}
where $\cD_2^2(t)$ is the squared $\ell^2(\mu)$-distance defined by
\[
\cD_2^2(t)=\sum_{I\in B_{k,n}}\mu(I)\left(\frac{P^t(I_0,I)}{\mu(I)}-1\right)^2
\]
Using the decomposition in the orthonormal basis $(f_J)_{J\in B_{k,n}}$ (for the weighted scalar product $\langle\cdot,\cdot\rangle_\mu$), we have
\[
\frac{P^t(I_0,I)}{\mu(I)}=\sum_{J\in B_{k,n}}\lambda_J^tf_J(I_0)f_J(I)=1+\sum_{J\neq I_0}\lambda_J^tf_J(I_0)f_J(I)\, ,
\]
so that
\begin{equation}\label{eq:L_2_distance}
\mathcal{D}^2_2(t)=\left\| \sum_{J\neq I_0}\lambda_J^tf_J(I_0)f_J\right\|_{2,\mu}^2=\sum_{J\neq I_0}\vert d(J)\vert^2\vert\lambda_{J}\vert^{2t},
\end{equation}
where $d(J)=f_J(I_0)=\left|\frac{V(\xi(J))}{V(\xi(I_0))}\right|=\left|\frac{\prod_{1\leq i<j\leq k}\sin\left(\frac{\pi(J_j-J_i)}{n}\right)}{\prod_{1\leq i<j\leq k}\sin\left(\frac{\pi(j-i)}{n}\right)}\right|$. The fact that $I_0$ is indeed the worst starting point for the $\ell^2$-distance follows from the inequality $|S_J(\xi(I)|\leq |S_J(\xi(I_0)|$, which implies that $|f_J(I)|\leq |f_J(I_0)|$.

The goal of the section is to show the following proposition, yielding an upper bound on $\mathcal{D}_2(t)$ for $s>0$.
\begin{proposition}\label{prop:upper_bound_1}
Let $\eta\in (0,1/2)$ and $\delta,K_g,K_a>0$. For any sequence of Markov kernels belonging to $\mathcal{A}(\eta,\delta,K_g,K_a)$, there exists $C,c>0$ such that 
$$\mathcal{D}_2\left(\frac{1}{\gamma}(\log n+s)\right) \leq C\left(\Gamma(2s+O(1))+e^{-ck}\right),$$
where $\Gamma(s)=\sum_{\lambda\in\mathcal{P}}e^{-s\vert \lambda\vert}$.
\end{proposition}

Since $\lim_{s\rightarrow +\infty}\Gamma(s)=0$, Proposition~\ref{prop:upper_bound_1} yields the upper bound in Theorem~\ref{thm:main}.

\subsection{Bounds on $\lambda^{\ell}_J$}\label{subsec:summary_bounds_alpha}
To achieve the proof of Proposition \ref{prop:upper_bound_1}, we need to bound $\left\vert\lambda^{\ell}_J\right\vert$ for any $J\in B_{k,n}$ and $-k\leq \ell\leq k$. 

We will give different types of bound depending on the shape of $J$. For $J\in B_{k,n}$, denote by  
$$\mu_J=\sum_{i=1}\delta_{2\pi (J_i-(k-1)/2)/n)}\mod ]-\pi,\pi]$$ 
the corresponding finite measure on $]-\pi,\pi]$ and by $\hat{J}$ the support of $\mu_J$. The set of finite measures on $\mathcal{M}_1(]-\pi,\pi])$ is considered as a metric space with the $1$-Wassertein distance $W_1$ defined by 
$$W_1(\mu,\mu')=\sup\left(\left\vert \int_{-\pi}^\pi f d\mu-\int_{-\pi}^\pi f d\mu'\right\vert,\,f:]-\pi,\pi]\rightarrow\mathbb{R}\text{ $1$-Lipschitz}\right)$$
for $\mu,\mu'\in \mathcal{M}_1(]-\pi,\pi])$. For any pair $\mu=\sum_{i=1}^k\delta_{a_i},\mu'=\sum_{i=1}^k\delta_{b_i}\in\mathcal{M}_1(]-\pi,\pi])$ of discrete measures with support of same cardinality, there exists a transport map $T:Supp(\mu)\rightarrow Supp(\mu')$ such that 
$$W_1(\mu,\mu')=\int_{-\pi}^\pi\left\vert T(x)-y\right\vert d\mu(x)d\mu(y),$$
see \cite[Introduction and Section 7.1]{Villani}. When, $\mu=\mu_J,\mu'=\mu_{J'}$ for some $J,J'\in B_{k,n}$, this relation simplifies to 
$$W_1(\mu_J,\mu_{J'})= \sum_{x\in\widehat{J'}}\left\vert T(x)-x\right\vert,$$
with the condition $T:\widehat{J'}\rightarrow \widehat{J}$ is bijective. In the specific case where $J'=I_0$, let us denote by $T^J$ the corresponding map. As the next lemma shows, we can suppose without loss of generality that $\vert T(x)\vert\geq \vert x\vert$ for all $x\in\widehat{I_0}$.
\begin{lemma}\label{lem:increasing_transport_map}
There exists an optimal transport map $T:\widehat{I_0}\rightarrow \widehat{J}$ such that $\vert T(x)\vert\geq\vert x\vert$ for all $x\in\widehat{I_0}$.
\end{lemma}
\begin{proof}
Let $T$ be any  optimal transport map from $\mu_{I_0}$ to $\mu_J$, and suppose that $\vert T(x)\vert< \vert x\vert$ for some $x\in \widehat{I_0}$. Up to choosing another element satisfying this condition, we can suppose that $\vert T(y)\vert\geq \vert y\vert$ for all $y\in \widehat{I_0}$ with $\vert y\vert<\vert x\vert$. Then, since $\vert T(x)\vert< \vert x\vert$, necessarily $y:=T(x)\in \widehat{I_0}$ and, by the choice of $x$, $\vert T(y)\vert\geq \vert y\vert$. Define $\tilde{T}$ by $\tilde{T}(z)=T(z)$ for $z\not\in\{x,y\}$ and $\tilde{T}(x)=T(y)$, $T(y)=T(x)=y$. Then, by the triangular inequality,
$$\sum_{z\in \widehat{I_0}} \vert \tilde{T}(z)-z\vert=\sum_{z\in \widehat{I_0}} \vert T(z)-z\vert+\vert T(y)-x\vert-\vert T(y)-y\vert-\vert y-x\vert\leq\sum_{z\in \widehat{I_0}} \vert T(z)-z\vert,$$
and $\tilde{T}$ is also an optimal transport map. Remark that now either $\vert\tilde{T}(x)\vert=\vert T(y)\vert\geq\vert y\vert$, and $\vert\tilde{T}(x)\vert=\vert y\vert$ only if $T(y)=-y$ (indeed, $T(y)\not=y$ since $T(x)=y)$). In the latter case, since $T(-y)\not\in\{y,-y\}$, necessarily $\vert \tilde{T}(-y)\vert=\vert T(-y)\vert>\vert -y\vert=\vert y\vert$, and iterating the process yields a optimal bijective transport map $\tilde{T }$ such that $\tilde{T}(x)>T(x)$. Iterating the process, we can find $T$ optimal transport map such that $\vert T(x)\vert\geq\vert x\vert$. Iterating again the process, we can find an optimal transport map $T$ such that  $\vert T(x)\vert\geq\vert x\vert$ for all $x\in \widehat{I_0} $.
\end{proof}

There is a shift action of $\mathbb{Z}$ on $B_{k,n}$ given by 
$$t\cdot(J_1,\ldots,J_{k})=Sort(J_1+t[n],\ldots,J_k+t[n]),$$
where $Sort(x_1,\ldots,x_k)=(x_{i_1}\geq  \ldots\geq x_{i_k})$ is the decreasing rearrangement of $(x_1,\ldots,x_k)$. Remark that for $t\in\mathbb{Z}$,
$$\lambda_{t\cdot J}^{\ell}=\frac{e_{\ell}(\xi(t\cdot J))}{e_{\ell}(\xi(I_0))}=\frac{e^{\frac{2i\pi t\ell}{n}}e_{\ell}(\xi(J)}{e_{\ell}(\xi(I_0))}=e^{\frac{2i\pi t\ell}{n}}\lambda_{J},$$
so that the map $J\mapsto \left\vert\lambda_{J}^{\ell}\right\vert$ factors through $\mathcal{J}:=B_{k,n}/\mathbb{Z}$. For $\omega\in\mathcal{J}$, let us denote by $\left\vert\lambda_{\omega}^{\ell}\right\vert$ the common value of $\left\vert \lambda_{J}^{\ell}\right\vert$ for $J\in\omega$. Likewise, since $J\mapsto d(J)$ factors through $\mathcal{J}$, we denote by $d(\omega)$ the common value of $d(J)$ for all $J\in\omega$.

The set $\mathcal{J}$ will be partitioned into three sets, depending on the distance between an optimally chosen representative of $B_{k,n}/\mathbb{Z}$ and $I_0$. Let $C_1,C_2>0$, and introduce 
$$\mathcal{J}_1:=\left\{\omega\in \mathcal{J}, \forall J\in \omega, \int\mathbf{1}_{\vert T^J(x)\vert>k\pi/n+\frac{C_2}{\log n}}(\vert T^J(x)\vert-\vert x\vert)d\mu_{I_0}(x)\geq C_1\frac{k}{\log k}\right\},$$
$$\mathcal{J}_{2}=\left\{\omega\in \mathcal{J}\setminus \mathcal{J}_1,\forall J\in\omega,\exists x\in I_0,\, T^J(x)\not=x\text{ and }\min(\vert T^J(x)\vert-k\pi/n,k\pi/n-\vert x\vert)\geq\frac{C_2}{\log n} \right\},$$
and finally
$$\mathcal{J}_3=\mathcal{J}\setminus (\mathcal{J}_1\cup \mathcal{J}_2).$$
An informal picture providing typical representative elements of the sets $\mathcal{J}_1, \mathcal{J}_2$ and $\mathcal{J}_3$ are given in Figure \ref{Fig:partition_J}

\begin{figure}[h!]
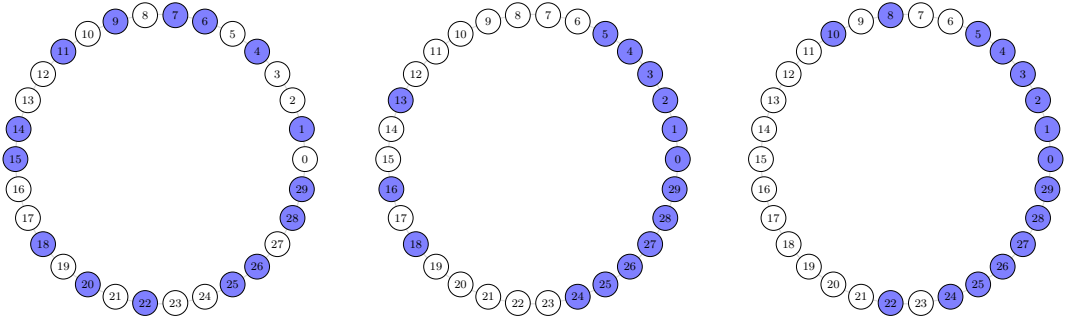

\scalebox{0.55}{\cercleDiscretFleches{29}{1,4,6,7,9,14,15,18,20,22,25,26,28,29,11}{}}\hspace{0.5cm}
\scalebox{0.55}{\cercleDiscretFleches{29}{24,25,26,27,28,29,0,1,2,3,4,5,13,16,18}{}}\hspace{0.5cm}
\scalebox{0.55}{\cercleDiscretFleches{29}{24,25,26,27,28,29,0,1,2,3,4,5,8,10,22}{}}
\caption{\label{Fig:partition_J} Typical representative elements in classes of $\mathcal{J}_1$, $\mathcal{J}_2$ and $\mathcal{J}_3$ for $B_{15,30}$. The last picture is the representative $J_{\omega}$ of a class $\omega\in\mathcal{J}_3$.}
\end{figure}

The elements of $\mathcal{J}_3$ are the closest ones to $I_0$ and will be the main contribution in the sum \eqref{eq:L_2_distance}. They represent equivalence classes $\omega$ including a configuration $J$ for which the corresponding transport map satisfies $T^{J}(x)=x$ if $\vert x\vert\leq \frac{k\pi}{n}-\frac{C_2}{\log n}$ and $\vert T^{J}(x)\vert\leq \frac{k\pi}{n}+\frac{C_{2}}{\log n}$ otherwise. For each such $J$, there exists a unique $t\in\llbracket -\frac{C_2}{\log n},\frac{C_2}{\log n}\rrbracket$ such that $\mu_{t\cdot J}$ has as many positive and negative atoms. Since $t\cdot J\in \omega$, let us set $J_{\omega}=t\cdot J$ to denote this particular representative element. Then, there exist a pair of partitions $\tau^{\omega}=(\mu^{\omega},\mu^{\omega})$ defined by 
$$\mu^{\omega}=\frac{2\pi}{n}Sort\left(T^{J_\omega}(x)-x, x\in \widehat{I_0}\cap \mathbb{R}^+\right),\,\nu^{\omega}=Sort\left(x-T^{J_\omega}(x),x\in \widehat{I_0}\cap \mathbb{R}^-\right).$$
The pair $(\mu^{\omega},\nu^{\omega})$ completely characterizes $\omega$ and satisfies $\max(\mu^{\omega}_1,\nu^{\omega}_1)\leq \frac{2C_2 n}{\pi \log n}$ and $\max(\ell(\mu^{\omega}),\ell(\nu^{\omega}))\leq \frac{2C_2 n}{\pi \log n}$, where $\ell(\mu)$ denotes the length of the partition $\mu$. In the sequel, we write 
$$\vert \tau^{\omega}\vert=\vert \mu^{\omega}\vert+\vert \nu^{\omega}\vert.$$
For exemple, the representative $J_{\omega}$ given in Figure \ref{Fig:partition_J} corresponds to the partition $\mu^{\omega}=(3,2)$ and $\nu^{\omega}=(1)$, yielding $\vert \tau^{\omega}\vert=5+1=6$.

In Section \ref{sec:estimates_superlog} and Section \ref{sec:estimates_lowerlog}, we will prove the following estimates, depending on $\mathcal{J}_{r}, r=1,2,3$.
\begin{proposition}\label{prop:estimates_J}
\begin{enumerate}
\item There exists $\kappa$ such that for all $\omega\in \mathcal{J}_1$,
$$\left\vert \lambda^{\ell}_{\omega}\right\vert=O\left(e^{-\kappa C_1\frac{\gamma_{\ell}k^2}{\log k}}\right).$$
\item Let $c_2,t_0>0$. For $C_2$ large enough, for all $\omega\in \mathcal{J}_2$, there exists $\tilde{\omega}\in\mathcal{J}_3$ such that for $t\geq t_0$
$$d(\omega)^{\frac{\gamma_{\ell}}{t\log k}}\left\vert \lambda^{\ell}_{\omega}\right\vert\leq O\left( e^{\frac{-c_2k\gamma_{\ell}}{t\log k}}\right)d(\tilde{\omega})^{\frac{\gamma_{\ell}}{t\log k}}\left\vert \lambda^{\ell}_{\tilde{\omega}}\right\vert.$$
\item For all $\omega\in\mathcal{J}_3$,
$$\left\vert \lambda^{\ell}_{\omega}\right\vert=e^{-\gamma_{\ell}\vert \tau^\omega\vert\left(1+O\left(\frac{1}{\log k}\right)\right)}.$$
If $\vert \tau^\omega\vert= O\left(\frac{1}{\gamma_{\ell}}\right)$, there exists $J_{\omega}\in\omega$ such that 
$$\lambda^{\ell}_{J_\omega}=e^{-\gamma_{\ell}\vert \tau^{\omega}\vert \left(1+O\left(\frac{1}{\log k}\right)\right)+iO_{\mathbb{R}}\left(\frac{\ell}{n^2}\vert \tau^{\omega}\vert\right )}$$
for all $-k\leq \ell\leq k$.
\end{enumerate}
\end{proposition}
\begin{proof}
Let us first assume that $0\leq \ell\leq \frac{k}{2}$.

(1) This is the content of Lemma \ref{lem:Case_J1} for $\ell\geq c\log^2 k$ and $c$ large enough (depending on $C_1$) and Lemma \ref{lem:bound_J_1_small} for $\ell\leq C\log^2k$ for any $C>0$.

(2) This is the content of Proposition \ref{prop:case_J2} for $\ell \geq \log k^7$ and $c$ large enough (depending on $C_2$) and Proposition \ref{prop:case_j2_small} for $\ell\leq \log k^s$ for any $s>0$.

(3) This is the content of Lemma \ref{lem:Case_J3} for $\ell\geq \log k^5$ and of Proposition \ref{prop:case_J3_small} for $\ell \leq \log k^s$ for $s>0$, and the fact that $\gamma_{\ell}\asymp \frac{\ell}{n^2}$ for $\ell\leq \frac{k}{2}$, see \eqref{eq:expression_gammal}.

If $\frac{k}{2}\leq \ell\leq k$, then the expression \eqref{eq:spectral_gap_symetry} together with the fact that $e_{\ell}(\xi(J))=\overline{e_{k-\ell}(\xi(J))}e_{k}(\xi(J))$ implies that 
$$\vert \lambda_{J}^{\ell}\vert=\vert\lambda_{J}^{k-\ell}\vert\text{ and }\lambda_{J^{\omega}}^{\ell}=\overline{\lambda_{J^{\omega}}^{k-\ell}}e^{i\frac{2\pi(\vert \mu^{\omega}\vert-\vert\nu^{\omega}\vert)}{n}}.$$
The first relation implies directly (1), (2) and the first part of (3). For the second part, this is implied by the second relation together with the fact that
$$\arg(\lambda_{J^{\omega}}^{\ell})=O\left(\frac{\vert \mu^{\omega}\vert+\vert\nu^{\omega}\vert}{n}\right)=O\left(\frac{\ell}{n^2}\vert \tau^{\omega}\vert\right)$$ 
for $k/2\leq \ell\leq k$ and $\eta n\leq k\leq (1-\eta)n$.

The case $\ell\leq 0$ is deduced from the case $\ell\geq 0$ by \eqref{eq:spectral_gap_symetry}.
\end{proof}
\subsection{Proof of the upper bound}

\begin{proof}[Proof of Proposition \ref{prop:upper_bound_1}]
Recall from \eqref{eq:L_2_distance} that 
$$\mathcal{D}_2^2(t)=\sum_{J\in B_{k,n}\setminus \{I_0\}}\vert d(J)\vert^2\vert\lambda_{J}\vert^{2t}$$
with 
$$\lambda_{J}=\sum_{\ell=-k}^kp_{\ell}\frac{\alpha_{J}^{\ell} }{\alpha_{I_0}^{\ell} }.$$
Recall that $\lambda_{J+\mathbf{1}_k}^{\ell} =e^{-\frac{2i\pi \ell}{n}}\lambda_{J}^{\ell}$ and $d(J+\mathbf{1}_k)=d(J)$. Hence, using \eqref{eq:L_2_distance} and factoring through $\mathcal{J}$ yields
\begin{align*}
\mathcal{D}_2^2(t)=\sum_{J\in B_{k,n}\setminus \{I_0\}}\vert d(J)\vert^2\vert\lambda_{J}\vert^{2t}\leq&\sum_{\omega\in\mathcal{J}}\sum_{m=0}^n\vert d(\omega)\vert^2\left\vert\sum_{\ell=-k}^kp_{\ell}e^{-\frac{2i\pi m \ell}{n}}\lambda_{\omega}^{\ell}\right\vert^{2t}-1\\
\leq & S_1+S_2+S_3+S_4,
\end{align*}
where,  
\begin{eqnarray*}
&S_1=\sum_{\omega\in\mathcal{J}_1}(k+1)\vert d(\omega)\vert^2\left\vert\sum_{\ell=-k}^kp_{\ell}\lambda_{\omega}^{\ell}\right\vert^{2t},\quad &S_2=\sum_{\omega\in\mathcal{J}_2}(k+1)\vert d(\omega)\vert^2\left\vert\sum_{\ell=-k}^kp_{\ell}\lambda_{\omega}^{\ell}\right\vert^{2t},\\
&S_3=\sum_{\substack{\omega\in\mathcal{J}_3\\\vert \tau^{\omega}\vert\geq \log n}}(k+1)\vert d(\omega)\vert^2\left\vert\sum_{\ell=-k}^kp_{\ell}\lambda_{\omega}^{\ell}\right\vert^{2t},\quad &S_4=\sum_{\substack{\omega\in\mathcal{J}_3\\\vert \tau^{\omega}\vert\leq \log n}}\sum_{m=0}^k\left\vert\sum_{\ell=-k}^kp_{\ell}e^{-\frac{2i\pi m \ell}{n}}\vert d(\omega)\vert^{\frac{1}{t}}\lambda_{\omega}^{\ell}\right\vert^{2t}-1.
\end{eqnarray*}
 Let $t=\tfrac{\log k +s}{\gamma}$, with $s>0$.
 
\textbf{Bound of $S_1$ :} 
Note first that $d(\omega)\leq \alpha^{k^2}$ for some constant $\alpha>0$ independent of $k$. Then, by Proposition~\ref{prop:estimates_J}(1), 
$$\lambda_{\omega}^{\ell}\leq e^{-\kappa C_1\frac{\gamma_{\ell}k^2}{\log k}}.$$
when $\omega\in \mathcal{J}_1$ for some $\kappa$ only depending on $C_1$. Hence, by Lemma \ref{lem:gamma_sub_gaussian}, there exists $c>0$ such that
\begin{align*}
\sum_{\ell=-k}^kp_{\ell}\lambda_{\omega}^{\ell}\leq\sum_{\ell=-k}^kp_{\ell} e^{-\kappa C_1\frac{\gamma_{\ell}k^2}{\log k}}\leq& e^{-\kappa C_1\frac{\gamma k^2}{\log k}}\sum_{\ell=-k}^kp_{\ell} e^{-\kappa C_1\frac{(\gamma_{\ell}-\gamma)k^2}{\log k}}\\
\leq& e^{-\kappa C_1\frac{\gamma k^2}{\log k}}e^{\frac{\kappa^2C_1^2 c^2 \gamma k^2}{(\log k)^2}} \leq e^{-\kappa C_1\frac{\gamma k^2}{2\log k}}
\end{align*}
for $k$ large enough. Choose $C_1$ large enough so that $\kappa C_1>2\log \alpha$,
$$\vert d(\omega)\vert^2\left\vert\sum_{\ell=-k}^kp_{\ell}\lambda_{\omega}^{\ell}\right\vert^{t}\leq \alpha^{2k^2} e^{-\kappa C_1 k^2}=e^{-\Omega(k^2)}\, .$$
Now, using that $\#\mathcal{J}_1=o(C^k)$ for some constant $C>0$, we get
\begin{equation}\label{eq:final_S1}
S_1\leq (k+1)e^{-\Omega(k^2)}\#\mathcal{J}_1=e^{-\Omega(k^2)}\, .
\end{equation}

\textbf{Bound of $S_2$ :} 
Let $c_2>0$ large enough to choose later. Let $\omega\in \mathcal{J}_2$. By Proposition~\ref{prop:estimates_J}(2) and Proposition~\ref{prop:estimates_J}(3), for $C_2$ large enough (independent of $\omega$), there exists $\tilde{\omega}\in\mathcal{J}_3$ such that 
$$d(\omega)^{\frac{\gamma_{\ell}}{t k\log k}}\vert \lambda_\omega^{\ell} \vert\leq  d(\tilde{\omega})^{\frac{2\gamma_{\ell}}{t\log k}}\exp\left(-\gamma_{\ell}\vert \tau^{\tilde{\omega}}\vert \left(1+O\left(\frac{1}{\log k}\right)\right)-\frac{k\gamma_{\ell}c_2}{t\log k}\right).$$
Therefore,
\begin{align*}
S_2=&\sum_{\omega\in\mathcal{J}_2}(k+1)\left\vert\sum_{\ell=-k}^kp_{\ell}\vert d(\omega)\vert^{\frac{\gamma}{\log k}}\lambda_{\omega}^{\ell}\right\vert^{2\frac{\log k+s}{\gamma}}\\
\leq&\sum_{\omega\in\mathcal{J}_2}(k+1)\left\vert\vert d(\tilde{\omega})\vert^{\frac{\gamma}{\log k}}e^{-kc_2\frac{\gamma}{\log k}}\sum_{\ell=-k}^kp_{\ell}\exp\left(-\gamma_{\ell}\vert \tau^{\tilde{\omega}}\vert\left(1+O\left(\frac{1}{\log k}\right)\right)\right)\right\vert^{2\frac{\log k+s}{\gamma}}
\end{align*}
By Lemma~\ref{lem:gamma_sub_gaussian}, we have
\begin{align}
&\exp\left(-\gamma\vert \tau^{\tilde{\omega}}\vert\left(1+O\left(\frac{1}{\log k}\right)\right)\right)\sum_{\ell=-k}^kp_{\ell}\exp\left(-(\gamma_{\ell}-\gamma)\vert \tau^{\tilde{\omega}}\vert\left(1+O\left(\frac{1}{\log k}\right)\right)\right)\nonumber\\
\leq& \exp\left(-\gamma\vert \tau^{\tilde{\omega}}\vert\left(1+O\left(\frac{1}{\log k}\right)\right)\right)\exp\left(\frac{K^2\gamma}{k^2}\vert \tau^{\tilde{\omega}}\vert^2\left(1+O\left(\frac{1}{\log k}\right)\right)\right)\label{eq:averaging_exp_S3}
\end{align}
For $\tilde{\omega}\in\mathcal{J}_3$, $\vert \tau^{\tilde{\omega}}\vert=O\left(\frac{k^2}{\log k^2}\right)$, so that 
$$\exp\left(\frac{K^2\gamma}{k^2}\vert \tau^{\tilde{\omega}}\vert^2\left(1+O\left(\frac{1}{\log k}\right)\right)\right)= \exp\left(\gamma\vert \tau^{\tilde{\omega}}\vert O\left(\frac{1}{\log k}\right)\right).$$
Hence,
\begin{align*}
S_2\leq& \sum_{\omega\in\mathcal{J}_2}(k+1)\left\vert\vert d(\tilde{\omega})\vert^{\frac{\gamma}{\log k}}e^{-kc_2\frac{\gamma}{\log k}}\exp\left(-\gamma\vert \tau^{\tilde{\omega}}\vert\left(1+O\left(\frac{1}{\log k}\right)\right)\right)\right\vert^{2\frac{\log k+s}{\gamma}}\\
\leq&(k+1)e^{-kc_2}\#\mathcal{J}_2\max_{\omega\in\mathcal{J}_3}d(\omega)^{2}\exp\left(-2\vert \tau^{\omega}\vert \left(\log k+s+O(1)\right)\right).
\end{align*}
Since $\#\mathcal{J}_2\leq C^{k}$ for some constant $C>0$ and $d(\omega)=O(\exp(\vert \tau^{\omega}\vert \log k))$ for $\omega\in\mathcal{J}_3$, for $c_2$ large enough and $s$ large enough, we have 
\begin{equation}\label{eq:final_S2}
S_2=O(e^{-ck}),
\end{equation}
for some $c>0$.

\textbf{Bound on $S_3$ :} By Proposition~\ref{prop:estimates_J}(3) and \eqref{eq:averaging_exp_S3}, we have 
\begin{align*}
S_3\leq& (k+1)\sum_{\substack{\omega\in\mathcal{J}_3\\\vert \tau^{\omega}\vert\geq \log n}}d(\omega)^2\left\vert\sum_{\ell=-k}^kp_{\ell}\lambda_{\omega}^{\ell}\right\vert^{\frac{\log k+s}{\gamma}}
\\
\leq& (k+1)\sum_{\substack{\omega\in\mathcal{J}_3\\\vert \tau^{\omega}\vert\geq \log n}}d(\omega)^2\left\vert\sum_{\ell=-k}^kp_{\ell}\exp\left(-\gamma_{\ell}\vert \tau^{\omega}\vert \left(1+O\left(\frac{1}{\log k}\right)\right)\right)\right\vert^{\frac{\log k+s}{\gamma}}\\
\leq& (k+1)\sum_{\substack{\omega\in\mathcal{J}_3\\\vert \tau^{\omega}\vert\geq \log n}}d(\omega)^2\left\vert\exp\left(-\gamma\vert \tau^{\omega}\vert \left(1+O\left(\frac{1}{\log k}\right)\right)\right)\right\vert^{\frac{\log k+s}{\gamma}}\\
\leq& (k+1)\sum_{\substack{\omega\in\mathcal{J}_3\\\vert \tau^{\omega}\vert\geq \log n}}d(\omega)^2\exp\left(-2(\log k+s)\vert \tau^{\omega}\vert \left(1+O\left(\frac{1}{\log k}\right)\right)\right.
\end{align*}
Since $d(\omega)=O(e^{\vert \tau^{\omega}\vert (\log k-C)})$ for any $C>0$ as $\vert \tau^{\omega}\vert\rightarrow +\infty$, 
\begin{equation}\label{eq:final_S3}
S_3\leq C(k+1)\left(\sum_{\substack{\lambda,\mu\in \mathcal{P}\\\vert\lambda\vert,\vert\mu\vert \geq \log n}}e^{-(s+2)(\vert\lambda\vert+\vert \mu\vert)}\right)\leq C\Gamma(2s),
\end{equation}
where $C>0$ and $\Gamma(s)=\sum_{\lambda\in \mathcal{P}}e^{-s\vert \lambda\vert}$.

\textbf{Bound on S4 :}
For $t=\frac{\log k+s}{\gamma}$, singling out $\omega_0=\overline{I_0}$ in $S_4$ yields
\begin{align*}
S_4=&\sum_{\substack{\omega\in\mathcal{J}_3\\\vert \tau^{\omega}\vert\leq \log n}} d(\omega)\sum_{m=0}^n\left\vert\sum_{\ell=-k}^kp_{\ell}e^{-\frac{2i\pi m \ell}{n}}\lambda_{J_\omega}^{\ell}\right\vert^{\frac{2(\log(k)+s)}{\gamma}}-1\\
=&\left(\sum_{m=0}^n\left\vert \sum_{\ell=-k}^kp_{\ell}e^{-\frac{2i\pi m \ell}{n}}\right\vert^{\frac{2(\log(k)+s)}{\gamma}}-1\right)+\sum_{\substack{\omega\in\mathcal{J}_3\\\vert \tau^{\omega}\vert\leq \log n,\omega\not= \omega_0}} d(\omega)\sum_{m=0}^n\left\vert\sum_{\ell=-k}^kp_{\ell}e^{-\frac{2i\pi m \ell}{n}}\lambda_{J_\omega}^{\ell}\right\vert^{\frac{2(\log(k)+s)}{\gamma}}\\
:=&S_4^1+S_4^2.
\end{align*}
Using Assumption~\ref{ass:X-aperiodicity}, the first term is bounded by
\begin{align*}
S_4^1=\left(\sum_{m=0}^n\left\vert \sum_{\ell=-k}^kp_{\ell}e^{-\frac{2i\pi m \ell}{n}}\right\vert^{\frac{2(\log(k)+s)}{\gamma}}-1\right)=&\sum_{m=1}^n\left\vert \sum_{\ell=-k}^kp_{\ell}e^{-\frac{2i\pi m \ell}{n}}\right\vert^{\frac{2(\log(k)+s)}{\gamma}}\\
=&\sum_{m=1}^n\left\vert \Phi_p\left(\frac{2\pi m}{n}\right)\right\vert^{\frac{2(\log(k)+s)}{\gamma}}\\
\leq &\sum_{m=1}^ne^{-K_a\min\left\{1,\, \frac{4\pi^2\E[|X|]m^2}{n^2}\right\} \frac{2(\log(k)+s)}{\gamma}}\, .
\end{align*}
Using that $\gamma\asymp\frac{\E[|X|]}{n^2}$, we have
\[
S_4^1=\sum_{m=1}^ne^{-\Omega\left(\min\left\{\frac{1}{\gamma},\, m^2\right\}(\log (k)+s\right)}=O\left(e^{-cs}\right)\, 
\]
for some constant $c>0$. For the second term, by Proposition \ref{prop:estimates_J} (3),
\begin{align*}
S_4^2=&\sum_{\substack{\omega\in\mathcal{J}_3\\\vert \tau^{\omega}\vert\leq \log n,\omega\not= \omega_0}} d(\omega)^2\sum_{m=0}^n\Bigg\vert\exp\left(-\gamma\vert \tau^{\omega}\vert \left(1+O\left(\frac{1}{\log k}\right)\right)\right)\\
&\hspace{3cm}\sum_{\ell=-k}^kp_{\ell}e^{-\frac{2i\pi m l}{n}}\exp\left(-(\gamma_{\ell}-\gamma)\vert \tau^{\omega}\vert \left(1+O\left(\frac{1}{\log k}\right)\right)+iO_{\mathbb{R}}\left(\frac{\ell\vert \tau^{\omega}\vert}{n^2}\right)\right)\Bigg\vert^{\frac{2(\log(k)+s)}{\gamma}}.
\end{align*}
For $\vert \tau^{\omega}\vert\leq \log n$, $(\gamma_{\ell}-\gamma)\vert \tau^{\omega}\vert =O(\frac{\log k}{k})$, so that 
\begin{align*}
&\exp\left(-(\gamma_{\ell}-\gamma)\vert \tau^{\omega}\vert \left(1+O\left(\frac{1}{\log k}\right)\right)+iO_{\mathbb{R}}\left(\frac{\ell}{n^2}\right)\right)=1+O\left((\gamma_{\ell}-\gamma)\vert \tau^{\omega}\vert \right)+O\left(\frac{\ell\vert \tau^{\omega}\vert}{n^2}\right),
\end{align*}
and  by Lemma~\ref{lem:gamma_sub_gaussian} and the fact that $\gamma\asymp \frac{\mathbb{E}[X]}{n^2}$, see \eqref{eq:equiv_gamma_expect},
$$\sum_{\ell=-k}^k p_\ell \left(|\gamma_{\ell}-\gamma|+O\left(\frac{\ell}{n^2}\right)\right)=O\left( \frac{\sqrt{\gamma}}{k}+\frac{\mathbb{E}[\vert X\vert]}{n^2}\right) \sqrt{\sum_{\ell=-k}^k p_\ell (\gamma_{\ell}-\gamma)^2}\leq \frac{C\mathbb{E}[\vert X\vert]}{n^2}$$
for some $C>0$. Hence,
\begin{align*}
&\left\vert\sum_{\ell=-k}^kp_{\ell}e^{-\frac{2i\pi m l}{n}}\exp\left(-(\gamma_{\ell}-\gamma)\vert \tau^{\omega}\vert \left(1+O\left(\frac{1}{\log k}\right)\right)+iO_{\mathbb{R}}\left(\frac{\ell}{n^2}\right)\right)\right\vert\leq\Phi_p\left(\frac{2\pi m}{n}\right)+\frac{C\mathbb{E}[\vert X\vert]\vert \tau^{\omega}\vert }{n^2}.
\end{align*}
for some $C>0$. Assumption~\ref{ass:X-aperiodicity} and \eqref{eq:equiv_gamma_expect} then yield
\begin{align*}
\Phi_p\left(\frac{2\pi m}{n}\right)+\frac{C\mathbb{E}[\vert X\vert]\vert \tau^{\omega}\vert }{n^2}\leq& 1-K_a\min\left\{\mathbb{E}[\vert X\vert]\frac{4\pi^2\min(m,n-m)^2}{n^2},1\right\}+\frac{C\mathbb{E}[\vert X\vert]\vert \tau^{\omega}\vert }{n^2}.
\end{align*}
If $\min(m,n-m)^2\geq c\vert \tau^{\omega}\vert$ for some $c>0$ large enough, then for $k$ large enough so that $K_a\geq \frac{C\mathbb{E}[\vert X\vert]\log n}{n^2}$,
$$K_a\min\left\{\mathbb{E}[\vert X\vert]\frac{4\pi^2\min(m,n-m)^2}{n^2},1\right\}-\frac{C\mathbb{E}[\vert X\vert]\vert \tau^{\omega}\vert }{n^2}\geq 0.$$
For such $m$, we then have 
\begin{align*}
&\Bigg\vert\exp\left(-\gamma\vert \tau^{\omega}\vert \left(1+O\left(\frac{1}{\log k}\right)\right)\right)\\
&\hspace{3cm}
\sum_{\ell=-k}^kp_{\ell}e^{-\frac{2i\pi m l}{n}}\exp\left(-(\gamma_{\ell}-\gamma)\vert \tau^{\omega}\vert \left(1+O\left(\frac{1}{\log k}\right)\right)+iO_{\mathbb{R}}\left(\frac{\ell}{n^2}\right)\right)\Bigg\vert^{\frac{2(\log(k)+s)}{\gamma}}\\
\leq &\exp\left(-2\vert \tau^{\omega}\vert \left(\log k+s+O(1)\right)\right).
\end{align*}
Hence, there exists $c>0$ such that
\begin{align}
&\sum_{\min(m,n-m)^2\geq c\vert\tau ^{\omega}\vert}\Bigg\vert\exp\left(-\gamma\vert \tau^{\omega}\vert \left(1+O\left(\frac{1}{\log k}\right)\right)\right)\nonumber\\
&\hspace{3cm}
\sum_{\ell=-k}^kp_{\ell}e^{-\frac{2i\pi m l}{n}}\exp\left(-(\gamma_{\ell}-\gamma)\vert \tau^{\omega}\vert \left(1+O\left(\frac{1}{\log k}\right)\right)+iO_{\mathbb{R}}\left(\frac{\ell}{n^2}\right)\right)\Bigg\vert^{\frac{2(\log(k)+s)}{\gamma}}\nonumber\\
\leq& \sum_{\min(m,n-m)^2\geq c\vert\tau ^{\omega}\vert}\exp\left(-2\vert \tau^{\omega}\vert \left(\log k+s+O(1)\right)\right)\nonumber\\
\leq &\exp\left(-2\vert \tau^{\omega}\vert \left(\log k+s+O(1)\right)\right)\label{eq:S_42_large_m}
\end{align}
for some $c>0$. Next, if $\min(m,n-m)^2\leq c\vert\tau ^{\omega}\vert$, let us use the former bound
\begin{align*}
&\left\vert\sum_{\ell=-k}^kp_{\ell}e^{-\frac{2i\pi m l}{n}}\exp\left(-(\gamma_{\ell}-\gamma)\vert \tau^{\omega}\vert \left(1+O\left(\frac{1}{\log k}\right)+iO(1)\right)\right)\right\vert\\
&\hspace{6cm}\leq \left\vert\sum_{\ell=-k}^kp_{\ell}\exp\left(-(\gamma_{\ell}-\gamma)\vert \tau^{\omega}\vert \left(1+O\left(\frac{1}{\log k}\right)\right)\right)\right\vert\\
&\hspace{6cm}=\exp\left(\gamma\vert \tau^{\omega}\vert O\left(\frac{1}{\log k}\right)\right).
\end{align*} 
Hence, since $\#\{m,\min(m,n-m)^2\leq c\vert\tau ^{\omega}\vert\}=2\sqrt{c\vert\tau ^{\omega}\vert}$,
\begin{align}
&\sum_{\min(m,n-m)^2\leq c\vert\tau ^{\omega}\vert}\Bigg\vert\exp\left(-\gamma\vert \tau^{\omega}\vert \left(1+O\left(\frac{1}{\log k}\right)\right)\right)\nonumber\\
&\hspace{4cm}
\sum_{\ell=-k}^kp_{\ell}e^{-\frac{2i\pi m l}{n}}\exp\left(-(\gamma_{\ell}-\gamma)\vert \tau^{\omega}\vert \left(1+O\left(\frac{1}{\log k}\right)\right)+iO_{\mathbb{R}}\left(\frac{\ell}{n^2}\right)\right)\Bigg\vert^{\frac{2(\log(k)+s)}{\gamma}}\nonumber\\
&\leq 2\sqrt{c\vert\tau ^{\omega}\vert}\exp\left(-2\vert \tau^{\omega}\vert \left(\log k+s+O(1)\right)\right).\label{eq:S_42_low_m}
\end{align}
Putting \eqref{eq:S_42_large_m} and \eqref{eq:S_42_low_m} together yields for $k$ large enough, using the fact that $d(\omega)\leq k^{\vert \tau^{\omega}\vert}$,
\begin{align*}
S_4^2\leq& \sum_{\substack{\omega\in\mathcal{J}_3\\\vert \tau^{\omega}\vert\leq \log n,J\not= I_0}} d(\omega)^2\sqrt{c\vert\tau ^{\omega}\vert}\exp\left(-2\vert \tau^{\omega}\vert \left(\log k+s+O(1)\right)\right)\\
=&\sum_{\substack{\omega\in\mathcal{J}_3\\\vert \tau^{\omega}\vert\leq \log n,J\not= I_0}} \exp\left(-2\vert \tau^{\omega}\vert \left(s+O(1)\right)\right)\\
\leq& C\Gamma(2s+O(1)),
\end{align*}
for some constant $C$ and $O(1)$ depending on the constants $\delta,\eta,K_a$ and $K_g$. Finally
\begin{equation}\label{eq:final_S4}
S_4\leq C\Gamma(2s+O(1))+O\left(e^{-cs}\right) .
\end{equation}
Combining \eqref{eq:final_S1}, \eqref{eq:final_S2}, \eqref{eq:final_S3} and \eqref{eq:final_S4} yields 
\begin{align*}
\mathcal{D}\left(\frac{1}{\gamma}(\log k+s)\right)\leq& e^{-c_2k^2}+e^{-c_2k}+Ce^{-s\log n}\Gamma(2s)+C\Gamma(2s+O(1))+e^{-cs} \\
\leq&  C\left(\Gamma(2s+O(1))+e^{-c'k}\right)
\end{align*}
for some constant $C,c$ only depending on $k$ and the constants of the assumptions.
\end{proof}

\section{Estimates in the super-logarithmic regime}\label{sec:estimates_superlog}
Recall the definition of $C_2,C_1$ in the definition of $\mathcal{J}_1,\mathcal{J}_2$ and $\mathcal{J}_3$ at the end of Section \ref{subsec:summary_bounds_alpha}. 
The goal of this section is to prove Proposition \ref{prop:estimates_J} in the following cases :
\begin{enumerate}
\item $\omega\in\mathcal{J}_1$ and $\ell\geq c\left(\log k\right)^2$ for some $c>0$ depending on $C_2$,
\item $\omega\in\mathcal{J}_2$ for $C_2$ large enough and $\ell\geq (\log n)^7$,
\item $\omega\in\mathcal{J}_3$ and $\ell\geq (\log n)^5$.
\end{enumerate}
Throughout this section, we assume that 
$$0\leq\ell\leq k/2.$$

\subsection{Rewriting the eigenvalues as complex integrals}
Recall the identity
$$\prod_{j=1}^k(1+zx_j)=\sum_{\ell=0}^kz^{\ell} e_{\ell}(x_1,x_2,\ldots,x_k).$$
Hence, by Cauchy integral theorem,
$$\alpha^{\ell} _J=\frac{1}{2\pi}\int_{-\pi}^{\pi}\prod_{j=1}^k(1+re^{i\theta}e^{2i\pi (J(j)-(k-1)/2)/n})r^{-l}e^{-il\theta})d\theta=\frac{1}{2\pi}\int_{-\pi}^{\pi}e^{k \left(g_{J}(re^{i\theta})-\frac{l}{k}\log(r)-i\frac{l}{k}\theta\right)}d\theta$$
for any contour $\mathcal{C}$ around $0$, where
$$g_{J}(z)=\frac{1}{k}\sum_{j=1}^k \log(1+ze^{2i\pi (J(j)-(k-1)/2)/n}).$$

\subsection{Stationary point for the largest eigenvalue}
The goal being to compare $\alpha^{\ell} _J$ to $\alpha^{\ell} _{I_0}$, let us first give an asymptotic of $\alpha^{\ell} _{I_0}$ by using the steepest descent method, see \cite[Section VIII]{flajolet2009analytic}. In this case, setting $g:=g_{I_0}$,
\begin{align*}
g(z)=&\frac{1}{k}\sum_{j=1}^k \log(1+ze^{i\pi(k+1-2j)/n}).
\end{align*}
The first step is to find $r\in \mathbb{R}_{>0}$ such that $rg'(r)-\frac{l}{k}=0$. Let us first gather some technical properties on $g$.
\begin{lemma}\label{lem:study_g}
Set $m(z)=zg'(z)$. Then,
\begin{itemize}
\item $m'(z)>0$ for $z\in\mathbb{R}_{>0}$,
\item $\lim_{z\rightarrow 0}m(z)=0$ and $\lim_{z\rightarrow +\infty,z>0}m(z)=1$.
\end{itemize}
Moreover, for $s\in \mathbb{N}$, $\sup_{z>0}\vert \max(z^{s-1},1)g^{(s)}(z)\vert<+\infty$.
\end{lemma}
\begin{proof}
Derivating $g(z)$ yields
$$m(z)=\frac{1}{k}\sum_{j=1}^k\frac{ze^{i\pi(k+1-2j)/n}}{1+ze^{i\pi(k+1-2j)/n}}=1-\frac{1}{k}\sum_{j=1}^k\frac{1}{1+ze^{i\pi(k+1-2j)/n}}.$$
This directly shows that 
$$\lim_{z\rightarrow 0}m(z)=0,\quad \lim_{z\rightarrow +\infty,z>0}m(z)=1.$$
Then, $m'(z)=\frac{1}{k}\sum_{j=1}^k\frac{e^{i\pi(k+1-2j)/n}}{(1+ze^{i\pi(k+1-2j)/n})^2}$. Since the function $\theta\mapsto\frac{e^{i\theta}}{(1+ze^{i\theta})^2}$ is $C^{\infty}$ on $]-\pi,\pi[$, a Riemann sum formula yields 
\begin{align*}
m'(z)=\frac{ n}{2\pi k}\int_{-k\pi/n}^{k\pi/n}\frac{e^{i\theta}}{(1+ze^{i\theta})^2}d\theta+O\left(\frac{1}{k^2}\right)=&\frac{ n}{2\pi ikz}\left(\frac{1}{1+ze^{-ik\pi/n}}-\frac{1}{1+ze^{ik\pi/n}}\right)+O\left(\frac{1}{k^2}\right)\\
=&\frac{- n}{k\pi z}\Im \frac{1}{1+ze^{ik\pi/n}}+O\left(\frac{1}{k^2}\right)>0. 
\end{align*}
Finally, note first that 
$$g^{(s)}(z)=\frac{1}{k}\sum_{j=1}^k\frac{e^{is\pi(k+1-2j)/n}(-1)^{s-1}(s-1)!}{(1+ze^{i\pi(k+1-2j)/n})^{s-1}}$$
is bounded on $[0,1]$. For $z\geq 1$,
  for $s\in\mathbb{N}_{>0}$ and $z\geq 1$
$$\vert z^{s-1}g^{(s)}(z)\vert =\left\vert\frac{1}{k}\sum_{j=1}^k\frac{z^{s-1}e^{is\pi(k+1-2j)/n}(-1)^{s-1}(s-1)!}{(1+ze^{i\pi(k+1-2j)/n})^{s-1}}\right\vert\leq \frac{(s-1)!}{(\min_{z>0}\vert z^{-1}+e^{ik\pi/n}\vert)^{s-1}},$$
so that $\sup_{z>0}\vert\max(1,z^{s-1}) g^{(s)}(z)\vert<+\infty$. 
\end{proof}

A consequence of Lemma \ref{lem:study_g} is the existence of a stationary point for $f$.
\begin{lemma}\label{lem:stationary_point}
There exists a unique $0<r<2$ such that $rg'(r)-\frac{l}{k}=0$. Moreover, $r=\frac{\sin\left(\frac{l\pi}{n}\right)}{\sin\left(\frac{(k-l)\pi}{n}\right)}+O\left(\frac{1}{k^2}\right)$.
\end{lemma}
Remark that for $l\geq k/2$, a similar result would hold, except for the new range $r\geq 1$.
\begin{proof}
By the previous lemma,
 $z\mapsto zg'(z)$ is strictly increasing and
$$\lim_{z\rightarrow 0} zg'(z)-\frac{l}{k}=-\frac{l}{k},\quad \lim_{z\rightarrow+\infty}zg'(z)-\frac{l}{k}=\frac{k-l}{k}>0.$$
Hence, there exists a unique $r>0$ such that $rg'(r)-\frac{l}{k}=0$. Since $m(z)=\frac{n}{2k\pi}\int_{-\frac{k\pi}{n}}^{\frac{k\pi}{n}}\frac{ze^{i\theta}}{1+ze^{i\theta}}d\theta+O\left(\frac{1}{n^2}\right)$, 
\begin{align*}
m(1)=\frac{n}{2k\pi}\int_{-\frac{k\pi}{n}}^{\frac{k\pi}{n}}\frac{e^{i\theta}}{1+e^{i\theta}}d\theta+O\left(\frac{1}{k^2}\right)=&\frac{n}{2ik\pi}\left[\log(1+e^{\frac{k\pi}{n}})-\log(1+e^{-\frac{k\pi}{n}})\right]+O\left(\frac{1}{k^2}\right)\\
=&\frac{n}{k\pi}\arg(1+e^{\frac{k\pi}{n}})+O\left(\frac{1}{k^2}\right)\\
=&\frac{1}{2}+O\left(\frac{1}{k^2}\right).
\end{align*}
Thus, since $m$ is increasing, the equality $m(r)=\frac{l}{k}$ implies $r\leq 2$ .
Since $\theta\mapsto \frac{1}{1+re^{i\theta}}$ is $\mathcal{C}^2$ on $[-k\pi/n,k\pi/n]$ with second derivative bounded uniformly on $r=O(1)$,
\begin{align*}
rg'(r)=&\frac{1}{k}\sum_{j=1}^{k}\frac{re^{i\pi(k+1-2j)/n}}{1+re^{i\pi(k+1-2j)/n}}\\
=&\frac{n}{2\pi k}\int_{-k\pi/n}^{k\pi/n}\frac{re^{i\theta}}{1+re^{i\theta}}d\theta+O\left(\frac{1}{k^2}\right)\\
=&\frac{n}{2\pi ik}\left(\log(1+re^{ik\pi/n})-\log(1+re^{ik\pi/n})\right)+O\left(\frac{1}{k^2}\right)\\
=&\frac{ n}{\pi k}\text{Arg}\left(1+re^{ik\pi/n}\right)+O\left(\frac{1}{k^2}\right)=\frac{ n}{\pi k}\arctan\left(\frac{r\sin(k\pi/n)}{1+r\cos(k\pi/n)}\right)+O\left(\frac{1}{k^2}\right).
\end{align*}
Hence, the equation $rg'(r)=\frac{l}{k}$ is equivalent, for $l\leq k/2$, to
\begin{equation}\label{eq:first_rel_r_gamma}
\frac{r\sin(k\pi/n)}{1+r\cos(k\pi/n)}=\tan\left(\frac{l\pi}{n}\right)+O\left(\frac{1}{k^2}\right),
\end{equation}
yielding 
$$r=\frac{\tan\left(\frac{l\pi}{n}\right)}{\sin(k\pi/n)-\tan\left(\frac{l\pi}{n}\right)\cos(k\pi/n)}+O\left(\frac{1}{k^2}\right)=\frac{\sin\left(\frac{l\pi}{n}\right)}{\sin\left(\frac{(k-l)\pi}{n}\right)}+O\left(\frac{1}{k^2}\right).$$
This implies the second statement of the lemma.
\end{proof}

In the following, we set 
$$f(\theta)=g\left(r e^{i\theta}\right)-\frac{l}{k}\log r-\frac{l}{k}i\theta.$$
\begin{lemma}\label{lem:study_f}
As $\theta=o(1)$,
$$\Re f(\theta)=f(0)-\frac{\alpha}{2}\theta^2+O(r\theta^4),$$
where $\alpha=r g'(r)+r^2g''(r)>0$, and
$$\Im f(\theta)=O\left(r\theta^3\right).$$
\end{lemma}
\begin{proof}
Since $g(\mathbb{R}_{>0})\subset \mathbb{R} $,
$$g^{(s)}(r)\in\mathbb{R}$$
for $s\in\mathbb{N}$. Hence, using the parity,
$$g(re^{i\theta})=g(r)+ir\theta g'(r)-\frac{rg'(r)+r^2g''(r)}{2}\theta^2+iO\left(\left(g^{(4)}(r) r^4+g^{(3)}(r)r^3+r^2g^{(2)}(r)+rg'(r)\right)\theta^3\right).$$
By the proof of Lemma \ref{lem:stationary_point}, $\alpha:=rg'(r)+r^2g''(r)=\frac{\partial}{\partial r}\left(rg'(r)\right)>0$ and by Lemma \ref{lem:stationary_point}, $g^{(4)}(r) r^4+g^{(3)}(r)r^3+r^2g^{(2)}(r)+rg'(r)=O(r)$. Using that $g'(r)-\frac{l}{k}i\theta=0$ yields
$$\Re f(\theta)=f(0)-\frac{\alpha}{2}\theta^2+O\left(r\theta^4\right),\,\Im f(\theta)=O\left(r\theta^3\right) $$
which implies the result.
\end{proof}

\subsection{Properties of the function $f(\theta)$}

\begin{lemma}\label{lem:first_bound_integral}
As $k$ goes to $\infty$, with $\theta_{0} =\rho \sqrt{\log(k)/(kr)} $ and $\rho$ large enough (independent of $\ell$),
$$\frac{1}{2\pi}\int_{[-\pi,\pi]\setminus[-\theta_0,\theta_0]}e^{k\Re f(\theta)}d\theta=O\left(\frac{r}{k\log k}\right)\alpha_{I_0}^{\ell} .$$
Moreover,
\begin{align*}
\alpha_{I_0}^{\ell} =\frac{1}{2\pi}\int_{-\pi}^{\pi}e^{k\Re f(\theta)}\cos(k\Im f(\theta))d\theta&=\frac{1}{2i\pi}\int_{-\pi}^{\pi}e^{k\Re f(\theta)}d\theta\left(1+O\left(\frac{1}{kr}\right)\right)\\
&=\frac{e^{kf(0)}}{\sqrt{2\pi kr^2f''(0)}}\left(1+O\left(\frac{1}{kr}\right)\right).
\end{align*}
\end{lemma}
\begin{proof}
Remark first that $\Re f(\theta)$ decreases with $\vert \theta\vert$. Indeed, this is true in an interval of size $o(1)$ by Taylor expansion around $r$. Then, derivating $\Re f(\theta)$ with respect to $\theta$ and a comparison series/integral yields that 
$$\frac{\partial}{\partial \theta }\Re f(\theta)=\frac{n}{2\pi k}\log\left(\frac{1+2r\cos (\theta+\frac{\pi k}{n})+r^2}{1+2r\cos(\theta-\frac{\pi k}{n})+r^2}\right)+O\left(\frac{1}{k^2}\right),$$
so that $\frac{\partial}{\partial \theta }f(\theta)<0$ for $\theta>\frac{1}{k}$ and $\frac{\partial}{\partial \theta }\Re f(\theta)>0$ for $\theta<-\frac{1}{k}$ and $k$ large enough. We deduce that 
\begin{equation}\label{eq:real_approx_bound_far_away}
\left\vert\frac{1}{2\pi}\int_{[-\pi,\pi]\setminus[-\theta_0,\theta_0]}e^{kf(\theta)}d\theta\right\vert=O\left( e^{kf(\theta_0)}\right).
\end{equation}

Remark that $f(-\theta)=\overline{f(\theta)}$, so that 
$$\int_{-\theta_0}^{\theta_0}e^{kf(\theta)}d\theta=\int_{-\theta_0}^{\theta_0}\Re \left[e^{kf(\theta)}\right]d\theta=\int_{-\theta_0}^{\theta_0}e^{k\Re f(\theta)}\cos\left[k\Im f(\theta)\right]d\theta.$$
By Lemma \ref{lem:study_g},
$$\Re f(\theta)=f(0)-\frac{\alpha}{2}\theta^2+O(r\theta^4),$$
where $\alpha=\left(r g'(r)+r^2g''(r)\right)>0$, and
$$\Im f(\theta)=O\left(r\theta^3\right).$$
Hence,
\begin{align*}
k\Re f(\theta)+\log\cos\left[k\Im(f(\theta))\right]-f(0)&=-\frac{k}{2}\alpha\theta^2+O(kr\theta^4)+O((kr\theta^3)^2).
\end{align*}
In particular, suppose that $\theta_0=\gamma\sqrt{\log(k)/(k\alpha)}$. Since $\alpha/r\in [1/C,C]$ for some $C>0$ by Lemma \ref{lem:study_g}, by \eqref{eq:real_approx_bound_far_away} and the latter expansion,
$$\left\vert\frac{1}{2\pi}\int_{[-\pi,\pi]\setminus[-\theta_0,\theta_0]}e^{kg(re^{i\theta})}d\theta\right\vert=O\left( \frac{e^{kg_{\ell}(r)}}{\sqrt{kr}}\frac{r}{k\log k}\right),$$
which proves the first assertion.
Then, doing the change of variable $u=\sqrt{k\alpha}\theta$,
\begin{align*}
\int_{-\theta_0}^{\theta_0}e^{kf(\theta)}d\theta=&e^{kf(0)}\int_{-\theta_0}^{\theta_0}e^{-k\alpha\theta^2}d\theta+O\left(e^{kf(0)}\int_{-\theta_0}^{\theta_0}e^{-k\alpha\theta^2}kr\vert \theta\vert^4+(kr\theta^3)^2d\theta\right)\\
=&\frac{e^{kf(0)}}{\sqrt{k\alpha}}\left[\int_{-\sqrt{k\alpha}\theta_0}^{\sqrt{k\alpha}\theta_0}e^{-u^2/2}du+O\left(\frac{kr}{(k\alpha)^{2}}\int_{-\sqrt{k\alpha}\theta_0}^{\sqrt{k\alpha}\theta_0}(\vert u\vert^4+\vert u\vert^6)e^{-u^2/2}du\right)\right]\\
=&\frac{\sqrt{2\pi}e^{kf(0)}}{\sqrt{k\alpha}}\left(1+O\left(\frac{1}{kr}\right)\right),
\end{align*}
where we used again that $\alpha\sim r$. The same reasoning with only $\Re f(\theta)$ instead of $f(\theta)$ yields the same asymptotic expansion, and the second assertion is deduced.
\end{proof}
The following estimate will be useful for small perturbation of $f$.
\begin{lemma}\label{lem:integral_small_pertub_f}
Let $\frac{c}{n}<\epsilon<\frac{C}{\log n}$. Assume that $h=f+\frac{1}{k}\delta$ with $\Vert\delta\Vert_{\infty}\leq O\left(\epsilon\log k\right)$  and $\Vert\delta\Vert_{\infty}^2\leq O\left(\frac{\epsilon}{\log k }\right)$ and such that $\Re \delta(\theta)=-\epsilon\left(1+O\left(\frac{1}{\log n}\right)\right)$ on $[\frac{C}{\log n},\frac{C}{\log n}]$. Then, as long as $l\geq \log n^7$,
$$\frac{1}{2\pi}\int_{-\pi}^{\pi}e^{kh(\theta)}d\theta= e^{-\epsilon\left(1+O\left(\frac{1}{\log k}\right)\right)+iO(\Vert \delta\Vert_{\infty})}\alpha^{\ell} _{I_0},$$
with $O(\cdot)$ only depending on $C$.
\end{lemma}
\begin{proof}
Set $\theta_0=\rho\sqrt{\frac{\log k}{rk}}$ with $\rho$ large enough. First, for $\vert \theta\vert\geq \theta_0$, by Lemma \ref{lem:study_g},
$$k\Re h(\theta)\leq kf(0)-C\rho^2\log n+O\left(\frac{1}{\log n}\right),$$
so that for $\rho$ large enough and by Lemma \ref{lem:first_bound_integral},
\begin{equation}\label{eq:J_3_negligible_theta}
\left\vert\int_{-\pi}^{-\theta_0}e^{kh(\theta)}d\theta\right\vert+\left\vert\int_{\theta_0}^{\pi}e^{kh(\theta)}d\theta\right\vert\leq \alpha^{\ell} _{I_0}O\left(\frac{1}{n^2}\right)\leq \alpha^{\ell} _{I_0}\exp\left(-\epsilon\right)O\left(\frac{\epsilon}{\log k}\right)
\end{equation}
For $\vert \theta\vert\leq \theta_0$, recall from Lemma \ref{lem:study_f} that $\Im f(\theta)=O\left(r\theta^3\right)$ and using that, for $\theta=O\left(\sqrt{\frac{\log k}{kr}}\right)$, $ O\left(kr\theta^3\right)=O\left(\frac{\log n^{3/2}}{\sqrt{kr}}\right)=o(1)$,
\begin{align*}
\Re\left(\frac{1}{2\pi}\int_{-\theta_0}^{\theta_0}e^{kh(\theta)}d\theta\right)=&\frac{1}{2\pi}\int_{-\theta_0}^{\theta_0}e^{k\Re h(\theta)}\cos\left(\Im \left[kh(\theta)\right]\right)d\theta\\
=&\frac{1}{2\pi}\int_{-\theta_0}^{\theta_0}e^{k\Re(f(\theta)-\frac{\epsilon}{k}\left(1+o\left(\frac{1}{\log k}\right)\right)}\cos\left(k\Im f(\theta)+O\left(\Vert\delta\Vert_{\infty}\right)\right)d\theta\\
=&\frac{1}{2\pi}\int_{-\theta_0}^{\theta_0}e^{k\Re f(\theta)-\epsilon\left(1+o\left(\frac{1}{\log k}\right)\right)}\cos\left(k\Im f(\theta)\right)\\
&\hspace{5cm}\left(1+O\left(\Vert\delta\Vert_{\infty}\max\left(\frac{\log n^{3/2}}{\sqrt{kr}},\Vert \delta\Vert_{\infty}\right)\right)\right)d\theta\\
=&\exp\left(-\epsilon\right)\left(1+\epsilon O\left(\frac{1}{\log k}\right)\right)\alpha^{\ell} _{I_0},
\end{align*}
for $\sqrt{kr}\geq \log n^{7/2}$, where we used on the second equality that $\cos(\Im f(\theta))>c$ for some $c>0$ and for $\theta=O\left(\sqrt{\frac{\log k}{kr}}\right)$. Likewise,
\begin{align*}
\Im\left(\frac{1}{2\pi}\int_{-\theta_0}^{\theta_0}e^{kh(\theta)}d\theta\right)=&\frac{1}{2\pi}\int_{-\theta_0}^{\theta_0}e^{k\Re f(\theta)-\epsilon\left(1+o\left(\frac{1}{\log k}\right)\right)}\sin\left(k\Im f(\theta)+O\left(\Vert\delta\Vert_{\infty}\right)\right)d\theta\\
=&\frac{1}{2\pi}\int_{-\theta_0}^{\theta_0}e^{k\Re f(\theta)-\epsilon\left(1+o\left(\frac{1}{\log k}\right)\right)}\Bigg[\cos\left(k\Im f(\theta)\right)O\left(\Vert\delta\Vert_{\infty}\right)\\
&\hspace{3cm}+\sin(k\Im f(\theta))\left(1+O\left(\Vert\delta\Vert_{\infty}^2\right)\right)\Bigg]d\theta
\end{align*}
First, 
\begin{align*}
\frac{1}{2\pi}\int_{-\theta_0}^{\theta_0}&e^{k\Re f(\theta)-\epsilon\left(1+o\left(\frac{1}{\log k}\right)\right)}\left[\cos\left(k\Im f(\theta)\right)O\left(\Vert\delta\Vert_{\infty}\right)+\sin(k\Im f(\theta))O\left(\Vert \delta\Vert_{\infty}^2\right)\right]d\theta\\
&\hspace{2cm}=\exp\left(-\epsilon\left(1+O\left(\frac{1}{\log k}\right)\right)\right)\alpha^{\ell} _{I_0}O\left(\Vert \delta\Vert_{\infty}\right).
\end{align*}
Then, using that of $f(-\theta)=\bar{f(\theta)}$,
\begin{align*}
\frac{1}{2\pi}\int_{-\theta_0}^{\theta_0}e^{k\Re f(\theta)-\epsilon\left(1+o\left(\frac{1}{\log k}\right)\right)}\sin(k\Im f(\theta))d\theta=&\frac{1}{2\pi}\int_{-\theta_0}^{\theta_0}e^{k\Re f(\theta)}\sin(k\Im f(\theta))d\theta\\
+&\frac{1}{2\pi}\int_{-\theta_0}^{\theta_0}e^{k\Re f(\theta)}O\left(\epsilon\right)\sin(k\Im f(\theta))d\theta\\
=&0+O\left(\frac{\epsilon}{\log k}\right)\frac{1}{2\pi}\int_{-\theta_0}^{\theta_0}e^{k\Re f(\theta)}\cos(k\Im f(\theta))d\theta\\
=&\alpha_{I_0}^{\ell} O\left(\Vert\delta\Vert_{\infty}\right),
\end{align*}
where we used on the second equality that $\sin(k\Im f(\theta))=O(\frac{1}{\log n})=O\left(\frac{1}{\log n}\right)\cos(k\Im f(\theta))$ for $\vert \theta\vert\leq \theta_0$ and $l\geq \log n ^5$ and on the last equality that $\epsilon=O\left(\Vert \delta\Vert_{\infty}\right)$. Hence, 
\begin{align*}
\Im\left(\frac{1}{2\pi}\int_{-\theta_0}^{\theta_0}e^{kh(\theta)}d\theta\right)=&\alpha_{I_0}^{\ell} \exp\left(-\epsilon\left(1+O(\frac{1}{\log k}\right)\right)O\left(\Vert\delta\Vert_{\infty}\right).
\end{align*}

Finally, this yields together with \eqref{eq:J_3_negligible_theta}
\begin{align*}
\frac{1}{2\pi}\int_{-\pi}^{\pi}e^{kh(\theta)}d\theta
&=\exp\left(-\epsilon\left(1+O\left(\frac{1}{\log k}\right)\right)+iO(\Vert \delta\Vert_{\infty}\right)\alpha^{\ell} _{I_0},
\end{align*}
where we used that $\Vert\delta\Vert_{\infty}^2=O\left(\frac{\epsilon}{\log n}\right)$.
\end{proof}

\subsubsection*{Case $\mathcal{J}_1$ :}

For $J\in B_{k,n}$, set $\theta_J=argmax_{[-\pi,\pi]} \Re g_J(re^{i\theta})$. Recall the action $(t,J)\mapsto t\cdot J$ from Section \ref{subsec:summary_bounds_alpha}. Since $g_{1\cdot J}(z)=g_{J}(ze^{2i\pi\theta/n})+i\frac{l}{k}\theta \frac{2\pi}{n}$, there exists $t\in [0,n[$ such that 
$\theta_{t\cdot J}\in [-\pi/(2n),\pi/(2n)]$.

\begin{lemma}\label{lem:Case_J1}
There exists $\kappa>0$ such that for any $\omega\in \mathcal{J}_1$,
$$\vert \lambda_{\omega}^{\ell}\vert=o\left(e^{-\kappa C_1\frac{rk}{\log k}}\right)$$
when $\ell\geq c\log k^2$ for some $c$ only depending on $C_1$.
\end{lemma}
\begin{proof}
Let $\omega\in\mathcal{J}_1$ and choose $J\in \omega$ such that $\theta_J\in[-\pi/(2n),\pi/(2n)]$. Then,
$$\vert \lambda_{J}^{\ell}\vert=\frac{\vert\alpha^{\ell} _J\vert}{\vert \alpha^{\ell} _{I_0}\vert}.$$
By definition of $\theta_J$, for all $\theta\in [-\pi,\pi]$,
$$\Re g_{J}(re^{i\theta})\leq \Re g_{J}(re^{i\theta_J})\leq g(r)+\Re g(re^{i\theta_J})-g(r)+\Re g_{J}(re^{i\theta_J})-\Re g(re^{i\theta_J}).$$
First, since $\theta_J\in [-\pi/(2n),\pi/(2n)]$, $\Re g(re^{i\theta_J})-g(r)\leq 0$ by Taylor expansion at $r$, see Lemma \ref{lem:study_g}. Then, 
\begin{align*}
\Re g_{J}(re^{i\theta_J})-\Re g(re^{i\theta_J})\leq \frac{1}{k}\sum_{i=1}^k\Re\log(1+r&e^{i\theta_J+ 2i\pi(J(i)-(k+1)/2)/n})\\
&\hspace{1cm}-\frac{1}{k}\sum_{i=1}^k\Re\log(1+re^{i\theta_J+2i\pi(I_0(i)-(k+1)/2)/n}).
\end{align*}
Remark that 
$$\Re\log(1+re^{i\theta})=\log(\vert 1+re^{i\theta}\vert)=\frac{1}{2}\log(1+r^2+2r\cos\theta)$$
which is smooth on $]-\pi,\pi[$.  Since $\theta_J\in[-\pi/(2n),\pi(2n)]$,
\begin{align*}
\Re g_{J}(re^{i\theta_J})-\Re g(re^{i\theta_J})=& \frac{1}{2k}\sum_{i=1}^k\log(1+r^2+2r\cos(2\pi(J(i)-(k+1)/2)/n))\\
&\hspace{2cm}-\frac{1}{2k}\sum_{i=1}^k\log(1+r^2+2r\cos(2\pi(I_0(i)-(k+1)/2)/n)+O(\frac{1}{n})\\
=& \frac{1}{2k}\int_{\mathbb{R}}\log(1+r^2+2r \cos(\vert\theta\vert))d\mu_J(\theta)\\
&\hspace{2cm}-\frac{1}{2k}\int_{\mathbb{R}}\log(1+r^2+2r \cos(\vert\theta\vert)d\mu_{I_0}(\theta)+O(\frac{1}{n}).
\end{align*}
Let $T$ be an optimal transport map from $\mu_{I_0}$ to $\mu_J$ with $\vert T(x)\vert\geq \vert x\vert$ for all $x\in \widehat{I_0}$, the existence of which is given by Lemma \ref{lem:increasing_transport_map}. Then, since $\log(1+r^2+2r \cos(x))$ is decreasing on $[0,\pi]$,
\begin{align*}
\frac{1}{2k}&\int_{\mathbb{R}}\log(1+r^2+2r \cos(\vert\theta\vert))d\mu_J(\theta)-\frac{1}{2k}\int_{\mathbb{R}}\log(1+r^2+2r \cos(\vert\theta\vert))d\mu_{I_0}(\theta)\\
=&\frac{1}{2k}\int_{\mathbb{R}}\log(1+r^2+2r \cos(\vert T(\theta)\vert))-\log(1+r^2+2r \cos\theta)d\mu_{I_0}(\theta)\\
\leq &\frac{1}{2k}\int_{\mathbb{R}}\mathbf{1}_{\vert T(\theta)\vert>k\pi/n+\frac{C_2}{\log n}}\left(\log(1+r^2+2r \cos(\vert T(\theta)\vert))-\log(1+r^2+2r \cos\theta)\right)d\mu_{I_0}(\theta).
\end{align*}
By concavity of $\cos$ on there is a constant $c$ only depending on $k/n$ such that for $y\geq k\pi/n+\frac{C_2}{\log n}$ and $x\leq k\pi/n$,
$$\log(1+r^2+2r \cos(y))-\log(1+r^2+2r x)d\mu_{I_0}\cos(x)\leq c(x-y).$$
Hence, by hypothesis $\omega\in\mathcal{J}_1$,
\begin{align*}
&\frac{1}{2k}\int_{\mathbb{R}}\mathbf{1}_{\vert T(\theta)\vert>k\pi/n+\frac{C_2}{\log n}}\left(\log(1+r^2+2r \cos(\vert T(\theta)\vert))-\log(1+r^2+2r \cos\theta)\right)d\mu_{I_0}(\theta)\\
\leq& \frac{1}{2k}\int_{\mathbb{R}}\mathbf{1}_{\vert T(x)\vert>k\pi/n+\frac{C_2}{\log n}}cr(\vert \theta\vert-\vert T(\theta)\vert)d\mu_{I_0}(\theta)\\
\leq &-crC_1\frac{1}{\log k}.
\end{align*}
Hence, by Lemma \ref{lem:first_bound_integral},
$$ \frac{\vert\alpha_{\ell}^J\vert}{\alpha_{\ell}^{I_0}}\leq e^{-cC_1\frac{kr}{\log k}+\frac{1}{2}\log (kr)}\leq e^{-\kappa C_1\frac{kr}{\log k}},$$ 
provided $r\geq c\frac{\log k^2}{k}$ for $c$ large enough. Since $r$ is of order $\frac{\ell}{k}$ up to an absolute numeric, the result is deduced.
\end{proof}
\subsubsection*{Case $\mathcal{J}_3$}
Before tackling the case $\mathcal{J}_3$, let us first relate $r$ to $\gamma_{\ell}$.
\begin{lemma}\label{lem:relation_r_gammal}
For $1\leq \ell\leq k/2$
$$\gamma_{\ell}=\frac{2\pi}{n}\Im\left(\frac{re^{-i\frac{k\pi}{n}}}{1+re^{-i\frac{k\pi}{n}}}\right)\left(1+O\left(\frac{1}{n}\right)\right).$$
\end{lemma}
\begin{proof}
First, by  \eqref{eq:expression_gammal},
$$\gamma_{\ell}=2\frac{\pi}{n}\frac{\sin\left(\frac{\ell \pi}{n}\right)\sin\left(\frac{(k-\ell)\pi}{n}\right)}{\sin\left(\frac{k\pi}{n}\right)}\left(1+O\left(\frac{1}{n}\right)\right).$$
Then, by Lemma \ref{lem:stationary_point} and the fact that $\Im \left(\frac{re^{-i\theta}}{1+re^{-i\theta}}\right)= \frac{r\sin(\theta)}{1+r^2+2r\cos(\theta)}$ for $\theta\in \mathbb{R}$ and \eqref{eq:first_rel_r_gamma}
\begin{align*}
\Im \left(\frac{re^{-i\frac{k\pi}{n}}}{1+re^{-i\frac{k\pi}{n}}}\right)^{-1}=&\frac{r^2\sin\left(\frac{k\pi}{n}\right)^2+\left(1+r\cos\left(\frac{k\pi}{n}\right)\right)^2}{r\sin\left(\frac{k\pi}{n}\right)}\left(1+O\left(\frac{1}{n}\right)\right)\\
=&\frac{r^2\sin\left(\frac{k\pi}{n}\right)^2+\frac{r^2\sin\left(\frac{k\pi}{n}\right)^2}{\tan\left(\frac{\ell\pi}{n}\right)^2}}{r\sin\left(\frac{k\pi}{n}\right)}\left(1+O\left(\frac{1}{n}\right)\right)\\
=&r\sin\left(\frac{k\pi}{n}\right)\left(1+\cot\left(\frac{\ell\pi}{n}\right)^2\right)\left(1+O\left(\frac{1}{n}\right)\right)\\
=&\frac{\sin\left(\frac{k\pi}{n}\right)}{\sin\left(\frac{\ell \pi}{n}\right)\sin\left(\frac{(k-\ell)\pi}{n}\right)}\left(1+O\left(\frac{1}{n}\right)\right),
\end{align*}
so that
$$\gamma_{\ell}=\frac{2\pi}{n}\Im\left(\frac{re^{-i\frac{k\pi}{n}}}{1+re^{-i\frac{k\pi}{n}}}\right)\left(1+O\left(\frac{1}{n}\right)\right).$$
\end{proof}
If $\omega\in\mathcal{J}_3$, recall from Section \ref{subsec:summary_bounds_alpha} that there exists $J_{\omega}\in \omega$ such that there exist a pair of partitions $\tau^{\omega}=(\mu^{\omega},\mu^{\omega})$ defined by 
$$\mu^{\omega}=\frac{2\pi}{n}Sort\left(T^{t\cdot J_\omega}(x)-x, x\in \widehat{I_0}\cap \mathbb{R}^+\right),\,\nu^{\omega}=Sort\left(x-T^{t\cdot J_\omega}(x),x\in \widehat{I_0}\cap \mathbb{R}^-\right),$$
with $\max(\mu^{\omega}_1,\nu^{\omega}_1)\leq \frac{2C_2 n}{\pi \log n}$ and $\max(\ell(\mu^{\omega}),\ell(\nu^{\omega}))\leq \frac{2C_2 n}{\pi \log n}$. We set $\vert \tau^{\omega}\vert=\vert \tau^{\omega}\vert$. In this section, we prove the following estimate.
\begin{lemma}\label{lem:Case_J3}
For $\omega\in \mathcal{J}_3$ and $\ell\geq \log k^5$,
$$\left\vert \lambda_{\omega}^{\ell} \right\vert=e^{-\gamma_{\ell}\vert \tau^{\omega}\vert\left(1+O\left(\frac{1}{\log k}\right)\right)}.$$
If $\vert \tau^{\omega}\vert\leq \frac{1}{\gamma_{\ell}\log n}$, then for $J_{\omega}$,
$$\alpha_{J_{\omega}}^{\ell} =e^{-\gamma_{\ell}\vert \tau^{\omega}\vert\left(1+O\left(\frac{1}{\log k}\right)\right)+iO\left(\gamma_\ell\vert \tau^{\omega}\vert\right)}\alpha_{I_0}^{\ell} .$$
\end{lemma}

\begin{proof}
Throughout this proof, we simply set $J=J_{\omega}$. First $\vert \tau^{\omega}\vert\leq C\frac{n^2}{(\log n)^2}$ implies that $W_1(\mu_J,\mu_{I_0})\leq C\frac{n}{\log n^2}$. Hence, for any $c>0$
\begin{align*}
\frac{1}{2\pi}\int_{-\pi}^{-C'/\log n}e^{kg_{J}(re^{i\theta})-l\log r-li\theta}d\theta=&\frac{1}{2\pi}\int_{-\pi}^{-C/\log n}e^{k\left(f(\theta)+O\left(\frac{r}{\log n^2}\right)\right)}d\theta\\
=&\frac{1}{2\pi}\int_{-\pi}^{-C'/\log n}e^{kf(0)-kf''(r)\theta^2)+O\left(\frac{kr}{\log n^2}\right)}d\theta\\
=&\frac{1}{2\pi}\int_{-\pi}^{-C'/\log n}e^{kf(0)-C^{(2)}kr/(\log n^2)+O\left(\frac{kr}{\log n^2}\right)}d\theta\\
=&O\left(e^{kf(0)-c'\frac{kr}{\log n^2}}\right)=O( \alpha^{\ell} _{I_0} e^{-c\frac{kr}{\log n^2}})
\end{align*}
for $C'$ large enough and some constant $c$ only depending on $C'>0$ and under the condition $r\geq \log n^3$. Similarly,
$$\frac{1}{2\pi}\int_{C'/\log n}^{\pi}e^{kg_{J}(re^{i\theta})-l\log r-li\theta}d\theta=O\left(\alpha^{\ell} _{I_0} e^{-c\frac{kr}{\log n^2}}\right).$$
Since $\vert \tau^{\omega}\vert\leq \frac{Cn^2}{\log n ^2}$, $\frac{r\vert \tau^{\omega}\vert}{n}=O(\frac{rn}{\log n^2})$ and for $c$ large enough 
\begin{align*}
e^{-c\frac{kr}{\log n^2}}=e^{-\gamma_{\ell}\vert \tau^{\omega}\vert-\tilde{c}\frac{kr}{\log n^2}},
\end{align*}
where we used Lemma \ref{lem:relation_r_gammal} to relate $r$ to $\gamma_{\ell}$. For $r\geq \frac{\log n^3}{k}$ and $\tilde{c}$ large enough, $e^{-\tilde{c}\frac{kr}{\log n^2}}\leq n^{-3}$, so that 
$$\frac{1}{2\pi}\int_{C'/\log n}^{\pi}e^{kg_{J}(re^{i\theta})-l\log r-li\theta}d\theta=\alpha^{\ell} _{I_0}e^{-\gamma_{\ell}\vert \tau^{\omega}\vert}O\left(\frac{1}{n^3}\right).$$
Recall that $g_{J}(re^{i\theta})=\sum_{j=1}^k\log(1+re^{i(\theta+\theta_{j}^{J})})$ and set $\theta_0=\frac{k\pi}{n}$. First, using that $\mu_i,\nu_i=O(n/\log n)$, for $\theta=O(1/\log(n))$,
\begin{align}
k\Re(g_{J}(re^{i\theta})-g(re^{i\theta}))=&\sum_{i=1}^{l_1}\Re\log\left(\frac{1+re^{i(\theta_i+2\pi\lambda_i/n+\theta)}}{1+re^{i(\theta_i+\theta)}}\right)+\sum_{i=1}^{l_2}\Re\log\left(\frac{1+re^{i(\theta_{k-i+1}+2\pi\mu_{i}/n+\theta)}}{1+re^{i(\theta_{k-i+1}+\theta)}}\right)\nonumber\\
=&\sum_{i=1}^{l_1}\frac{2\pi\lambda_i}{n}\Re\left(\frac{ire^{i\theta_0}}{1+re^{i\theta_0}}\right)\left(1+O\left(\frac{1}{\log n}\right)\right)+O\left(\frac{r\lambda_i^2}{n^2}\right)\nonumber\\
&\hspace{2cm}-\sum_{i=1}^{l_2}\frac{2\pi\mu_i}{n}\Re\left(\frac{ire^{-i\theta_0}}{1+re^{-i\theta_0}}\right)\left(1+O\left(\frac{r}{\log n}\right)\right)+O\left(\frac{r\mu_i^2}{n^2}\right)\nonumber\\
=&-\frac{\vert \tau^{\omega}\vert}{n}2\pi\Im\left(\frac{re^{-i\theta_0}}{1+re^{-i\theta_0}}\right)\left(1+O\left(\frac{1}{\log n}\right)\right),\label{eq:J3_real_part}
\end{align}
Hence, using Lemma \ref{lem:relation_r_gammal},
$$\frac{1}{2\pi}\int_{-C'/\log n}^{C'/\log n}e^{k\Re g_{J}(re^{i\theta})-l\log r}d\theta=e^{-\vert \tau^{\omega}\vert\gamma_{\ell}\left(1+O\left(\frac{1}{\log n}\right)\right)}\frac{1}{2\pi}\int_{-C'/\log n}^{C'/\log n}e^{k\Re f(\theta)}d\theta.$$

By Lemma \ref{lem:first_bound_integral}
$$\frac{1}{2\pi}\int_{-C'/\log n}^{C/\log n}e^{k\Re g(re^{i\theta})-l\log r}d\theta\leq \alpha^{\ell} _{I_0}\left(1+O\left(\frac{1}{kr}\right)\right),$$
and finally,
\begin{equation}\label{eq:ineq_Reg_J3}
\frac{1}{2\pi}\int_{-\pi}^{\pi}e^{k\Re g_{J}(re^{i\theta})-l\log r}d\theta\leq e^{-\gamma_{\ell}\vert \tau^{\omega}\vert\left(1+O\left(\frac{r}{\log k}\right)\right)}\alpha^{\ell} _{I_0}\left(1+O\left(\frac{1}{kr}\right)+O\left(\frac{1}{n^3}\right)\right).
\end{equation}
In particular, if $\vert \tau^{\omega}\vert\geq \frac{n}{r\log n}$, the result is deduced since for $r\geq \frac{\log n^5}{k}$,
$$1+O\left(\frac{1}{kr}\right)+O\left(\frac{1}{n^3}\right)=1+O\left(\frac{1}{\log n^2}\right)=e^{\gamma_{\ell}\vert \tau^{\omega}\vert O\left(\frac{1}{\log k}\right)}.$$
In the case where $\vert \lambda\vert,\vert\mu\vert=O(\frac{n}{r\log n})$, we have
$$g_{J}(re^{i\theta})=g(re^{i\theta})+O\left(\gamma_{\ell}\vert \tau\vert\right),$$
with $O\left(\gamma_{\ell}\vert \tau\vert\right)=O\left(\frac{1}{\log n}\right)$. Hence, applying Lemma \ref{lem:integral_small_pertub_f} with $\delta(\theta)=g_{J}(re^{i\theta})-g(re^{i\theta})$ and $\epsilon=\gamma_{\ell}\vert \tau^{\omega}\vert$ yields
\begin{align*}
 \alpha^{\ell} _{J}&=\exp\left(-\gamma_{l}\vert \tau^{\omega}\vert\left(1+O\left(\frac{1}{\log k}\right)\right)+iO(\gamma_{l}\vert \tau^{\omega}\vert)\right)\alpha^{\ell} _{I_0}
\end{align*}
when $\vert \tau^{\omega}\vert\leq \frac{1}{\gamma_{l}\log n}$. 
\end{proof}

\subsubsection*{Case $\mathcal{J}_2$} 
For each $\omega\in\mathcal{J}_2$, we choose $J\in\omega$ such that 
$$\int\mathbf{1}_{\vert T^J(x)\vert>k\pi/n+\frac{C_2}{\log n}}(\vert T^J(x)\vert-\vert x\vert)d\mu_{I_0}(x)\leq C_1\frac{k}{\log k},$$
where $T^J$ is the optimal transport map from $\mu_{I_0}$ to $\mu_{J}$, see Section \ref{subsec:summary_bounds_alpha}. We will use the following useful bound. Recall that $d_{TV}(\mu,\nu)$ denotes the total-variation distance between $\mu$ and $\nu$.
\begin{lemma}\label{lem:J2_Wasserstein}
There exists $C>0$ only depending on $C_2,C_1$ such that, for $\omega\in \mathcal{J}_2$, there exists $J\in \omega$ with
$$d_{TV}(\mu_J,\mu_{I_0})\leq C\frac{k}{\sqrt{\log k}}.$$
\end{lemma}
\begin{proof}
Suppose that $\omega\in \mathcal{J}_2$ and let $J\in \omega$ be such that 
$$\int\mathbf{1}_{\vert T(x)\vert>k\pi/n+\frac{C_2}{\log n}}(\vert T(x)\vert-\vert x\vert)d\mu_{I_0}(x)\leq C_1\frac{k}{\log k},$$
where $T$ be the corresponding increasing transport map from $\mu_{I_0}$ to $\mu_{J}$. Let $N$ be the number of element $x$ of $J$ such that $\vert x\vert\geq k\pi/n+\frac{C_2}{\log n}$. Since at most two elements of $J$ have the same modulus and all atoms of $\mu_{I_0}$ are in $[-k\pi/n,k\pi/n]$,
$$\int\mathbf{1}_{\vert T(x)\vert>k\pi/n+\frac{C_2}{\log n}}(\vert T(x)\vert-\vert x\vert)d\mu_{I_0}(x)\geq 2\sum_{j=1}^{N/2}\left\vert k\pi/n+\frac{C_2}{\log n}+\frac{2j\pi}{n}-k\pi/n\right\vert\geq \frac{NC_2}{\log n}+\frac{\pi N(N+2)}{2n}.$$
We thus have 
$$\frac{NC_2}{\log n}+\frac{\pi N(N+2)}{2n}\leq C_1\frac{k}{\log k},$$
which implies $N\leq C\frac{k}{\sqrt{\log k}}$ for some constant only depending on $C_2$ and $C_1$. Hence, there are at most $C\frac{k}{\sqrt{\log k}}+\frac{2nC_2}{\log n}$ atoms of $\mu_J$ outside of $[-k\pi/n,k\pi/n]$. By the Pigeon's hole principle, we thus have 
$$d_{TV}(\mu_J,\mu_{I_0})\leq C\frac{k}{\sqrt{\log k}}+\frac{2nC_2}{\log n}\leq C\frac{k}{\sqrt{\log k}},$$
for some constant $C>0$.
\end{proof}
For $J\in B_{k,n}$, set 
$$N(J):=\#\{x\in I_0,\,\min(\vert T^J(x)\vert-k\pi/n,k\pi/n-\vert x\vert)\geq\frac{C_2}{\log n}\}.$$
Remark that by definition, see Section \ref{subsec:summary_bounds_alpha}, $N(J)=0$ if and only if $J\in \omega$ for some $\omega\in \mathcal{J}_3$.
\begin{lemma}\label{lem:J2}
Let $c,t>0$ and $\eta$ small enough universal constant. For $C_2$ large enough independent of $\ell$, for any $J\in B_{k,n}$ with $d_{TV}(\mu_J,\mu_{I_0})\leq C\frac{k}{\sqrt{\log k}}$, there exists $\tilde{J}\in B_{k,n}$ with $N(\tilde{J})=0$ (independent of $r,c,t$ and $C_2$) such that, for $\theta\in[-\eta,\eta]$,
$$d(J)^{\frac{2r}{tk\log k}}e^{k\Re g_{J}(re^{i\theta})}\leq c^{N(J)\frac{r}{t\log k}}d(\tilde{J})^{\frac{2r}{tk\log k} }\max\left(e^{k\Re g_{\tilde{J}}(re^{i\theta})}, e^{-krc}\alpha^{\ell} _{I_0}\right)$$
and 
$$\Im g_{\tilde{J}}(re^{i\theta})=\Im g_{J}(re^{i\theta})+O\left(\frac{r}{k}N(J)\right).$$
\end{lemma}
\begin{proof}
We will prove by recursion on $N(J)$ that there exists $\tilde{J}\in B_{k,n}$ with $N(\tilde{J})=0$ such that 
$$d(J)^{\frac{2r}{tk\log k}}e^{k\Re g_{J}(re^{i\theta})}\leq c^{\frac{r}{t\log k}}d(\tilde{J})^{\frac{2r}{tk\log k} }e^{k\Re g_{\tilde{J}}(re^{i\theta})}$$
for $\theta \in[-\eta,\eta]$ for $C_2$ large enough.

If $N(J)=1$, there exists a pair $(x,y)$ with $y\in J\setminus I_0, x\in I_0\setminus J$ with either $\vert y\vert\geq \pi k/n+\frac{C_2}{\log n}$ and $\vert x\vert\leq \frac{k\pi}{n}$,  or $\vert y\vert\geq \frac{k\pi}{n}$ and $\vert x\vert\leq \pi k/n-\frac{C_2}{\log n}$. Denote by $\tilde{J}$ the configuration obtained by replacing $y$ by $x$. Then,
$$\#\{x\in \tilde{J},\,\min(\vert T(x)\vert-k\pi/n,k\pi/n-\vert x\vert)\geq\frac{C_2}{\log n}\}=N(J)-1=0,$$
so that $N(\tilde{J})=0$. Let $\eta>0$ be small enough (and universal). For $\theta\in]-\eta,\eta[$, there exists $C>0$ such that
\begin{align*}
\Re g_{J}(re^{i\theta})-\Re g_{\tilde{J}}(e^{i\theta})=&\frac{1}{k}\left(\Re\log(1+re^{i(\theta+y)})-\Re\log(1+re^{i(\theta+x)}\right)\\
\leq& \frac{1}{k}\log\left\vert\frac{1+re^{i(\theta+y)}}{1+re^{i(\theta+x)}}\right\vert\leq \frac{Cr}{k}(x-y)\leq -\frac{CrC_2}{k\log n}.
\end{align*}
Hence, 
$$e^{k\Re g_{J}(re^{i\theta})}\leq e^{-\frac{CrC_2}{\log k}}e^{k\Re g_{\tilde{J}}(re^{i\theta})}.$$
On the other hand, let $l,l'$ be such that $\frac{2\pi}{n}J_{\ell}-\frac{k+1}{2}=x$ and $\frac{2\pi}{n}\tilde{J}_{l'}-\frac{k+1}{2}=y$. Since $\{J_s,s\not=l\}=\{\tilde{J}_s,s\not=l'\}$,
$$\left\vert \frac{d(\tilde{J})}{d(J)}\right\vert =\frac{\prod_{s\not =l'} \sin(\pi \vert J_s-\tilde{J}_{l'}\vert/n)}{\prod_{s\not =l} \sin(\pi \vert J_s-J_{l}\vert/n)}\geq \prod_{s\not =l'} \sin(\pi \vert J_s-\tilde{J}_{l'}\vert/n)\geq \prod_{j=1}^{k/2}\sin(\pi j/n)^2.$$
Remark that 
$$\prod_{j=1}^{k/2}\sin(\pi j/n)^2=\exp\left(k\left(\frac{2}{k}\sum_{j=1}^{k/2} \log \sin(\pi j/n)\right)\right)=\exp\left(2n\left(\int_0^{1/2}\log\sin(\frac{k\pi}{n} t)dt+o(1)\right)\right)\geq \kappa^{-n}$$
for some universal constant $\kappa>0$. 
Hence, 
$$\vert d(J)\vert^2\leq \kappa^{2k}\vert d(\tilde{J})^2,$$
so that 
$$d(J)^{\frac{2r}{tk\log k}}e^{k\Re g_{J}(re^{i\theta})}\leq \alpha^{\frac{2r}{t\log k}} e^{-\frac{CrC_2}{\log k}}d(\tilde{J})^{\frac{2r}{tk\log k} }e^{k\Re g_{\tilde{J}}(re^{i\theta})}.$$
For $C_2$ large enough, $CC_2-\frac{2}{t}\log \kappa>\frac{\log c}{t}$, and then 
$$d(J)^{\frac{2r}{tk\log k}}e^{k\Re g_{J}(re^{i\theta})}\leq c^{\frac{r}{t\log k}}d(\tilde{J})^{\frac{2r}{tk\log k} }e^{k\Re g_{\tilde{J}}(re^{i\theta})},$$
which gives the first step of the recursion.

If $J\in B_{k,n}$ is such that $d_{TV}(\mu_J,\mu_{I_0})\leq C\frac{k}{\sqrt{\log k}}$ and $N(J)>1$, doing the same reasoning as before yields 
$$d(J)^{\frac{2r}{tk\log k}}e^{k\Re g_{J}(re^{i\theta})}\leq c^{\frac{r}{t\log k}}d(\hat{J})^{\frac{2r}{tk\log k}}e^{k\Re g_{\tilde{J}}(re^{i\theta})},$$
with $\hat{J}$ such that $N(\hat{J})=N(J)-1$. Remark moreover that 
$$d_{TV}(\mu_{\hat{J}},\mu_{I_0})\leq d_{TV}(\mu_{J},\mu_{I_0})\leq  C\frac{k}{\sqrt{\log k}}$$
 By induction, there exists $\tilde{J}\in B_{k,n}$ with $N(\tilde{J})=0$ such that 
$$d(\hat{J})^{\frac{2r}{tk\log k}}e^{k\Re g_{\hat{J}}(re^{i\theta})}\leq c^{N(\hat{J})\frac{r}{t\log k}}d(\tilde{J})^{\frac{2r}{tk\log k} }e^{k\Re g_{\tilde{J}}(re^{i\theta})}.$$
Combining the two latter inequalities yields thus 
$\tilde{J}$ such that 
\begin{equation}\label{eq:J2_bound_in_eta}
d(J)^{\frac{2r}{tk\log k}}e^{k\Re g_{J}(re^{i\theta})}\leq c^{N(J)\frac{r}{t\log k}}d(\tilde{J})^{\frac{2r}{tk\log k} }e^{k\Re g_{\tilde{J}}(re^{i\theta})}
\end{equation}
for $\theta\in[-\eta,\eta]$. Then, since $d_{TV}(\mu_J,\mu_{\tilde{J}})=O(N(J))$, $k\Im g_{\tilde{J}}(re^{i\theta})=k\Im g_{J}(re^{i\theta})+O(rN(J))$. 

For $\theta\not\in[-\eta,\eta]$, remark that we have by Lemma \ref{lem:J2_Wasserstein} $g_{J}(re^{i\theta})\leq g(re^{i\theta})+\frac{\tilde{C}}{\sqrt{\log k}}\leq g(r)-cr$ for some $c>0$ and $k$ large enough. Since $\frac{d(J)^{\frac{2r}{t\log k}}}{d(\tilde{J})^{\frac{2r}{t\log k}}}\leq \kappa^\frac{2rk}{t\log k}$, adjusting $c$ and taking $k$ large enough yields
\begin{equation}\label{eq:J2_bound_out_eta}
d(J)^{\frac{2r}{tk\log k}}e^{k\Re g_{J}(re^{i\theta})-l\log r}\leq c^{\frac{r}{t\log k}}d(\tilde{J})^{\frac{2r}{tk\log k} }e^{-crk}\alpha^{\ell} _{I_0}
\end{equation}
for $\theta\not\in[-\eta,\eta]$. Combining \eqref{eq:J2_bound_in_eta} and \eqref{eq:J2_bound_out_eta} yields
\begin{align*}
d(J)^{\frac{2r}{tk\log k}}\vert \alpha_J^{\ell} \vert&\leq d(J)^{\frac{2r}{tk\log k}}\left\vert\frac{1}{2\pi}\int_{-\pi}^{\pi}e^{k\Re g_{J}(re^{i\theta})-l\log r}d\theta\right\vert\\
&\leq d(\tilde{J})^{\frac{2r}{tk\log k} }\max\left(c^{N(J)\frac{r}{t\log k}}\frac{1}{2\pi}\int_{-\pi}^{\pi}e^{k\Re g_{\tilde{J}}(re^{i\theta})-l\log r}d\theta,e^{-krc}\alpha^{\ell} _{I_0}\right).
\end{align*}
\end{proof}
\begin{proposition}\label{prop:case_J2}
Let $c_2,t_0>0$. For $C_2$ large enough, for all $\omega\in\mathcal{J}_2$ there exists $\tilde{\omega}\in \mathcal{J}_3$ such that for any $t\geq t_0$, denoting $\tau$ the corresponding pair of partitions given by $J_{\tilde{\omega}}$,
$$d(\omega)^{\frac{2r}{tk\log k}}\vert \lambda_{\omega}^{\ell} \vert\leq d(\tilde{\omega})^{\frac{2r}{tk\log k}}\exp\left(-\gamma_{\ell}\vert \tau^{\tilde{\omega}}\vert\left(1+O\left(\frac{1}{\log k}\right)\right)-\frac{rc_2}{t\log k}\right),$$
provided $\ell\geq \log k^{7}$.
\end{proposition}
\begin{proof}
Let $\omega\in\mathcal{J}_2$ and $J\in \omega$ the corresponding particle configuration introduced at the beginning of the section. By Lemma \ref{lem:J2_Wasserstein} and Lemma \ref{lem:J2}, for $C_2$ large enough, there exists $\tilde{J}$ with $N(\tilde{J})=0$ such that
$$d(J)^{\frac{2r}{tk\log k}}\vert \alpha_J^{\ell} \vert\leq  e^{-cN(J)\frac{r}{t\log k}}d(\tilde{J})^{\frac{2r}{tk\log k} }\max\left(\frac{1}{2\pi}\int_{-\pi}^{\pi}e^{kg_{\tilde{J}}(re^{i\theta})+iO(rN(J))-l\log r}d\theta, e^{-krc}\alpha^{\ell} _{I_0}\right).$$
Let $\tilde{\omega}\in\mathcal{J}_3$ be the equivalence class of $\tilde{J}$ and $J_{\tilde{\omega}}$ the corresponding representative given in Section \ref{subsec:summary_bounds_alpha}. Denote by $\tau^{\tilde{\omega}}=(\mu^{\tilde{\omega}},\nu^{\tilde{\omega}})$ the corresponding pair of partitions. Let us first assume that either $\frac{r\vert \tau^{\tilde{\omega}}\vert}{n}\geq \frac{1}{\log n}$ or $N(J)\geq \frac{t\tilde{c}}{cr\log k}$, with $\tilde{c}$ numeric large enough to be specified later. By \eqref{eq:ineq_Reg_J3}, 
\begin{align*}
e^{-cN(J)\frac{r}{t\log k}}\left\vert \frac{1}{2\pi}\int_{-\pi}^{\pi}e^{kg_{\tilde{J}}(re^{i\theta})+iO(rN(J))-l\log r}d\theta\right\vert\leq & \frac{1}{2\pi}\int_{-\pi}^{\pi}e^{k\Re g_{\tilde{J}}(re^{i\theta})-l\log r-cN(J)\frac{r}{t\log k}}d\theta\\
\leq& e^{-\gamma_{\ell}\vert \tau^{\tilde{\omega}}\vert\left(1+O\left(\frac{1}{\log k}\right)\right)-cN(J)\frac{r}{t\log k}+O\left(\frac{1}{kr}\right)}\alpha^{\ell} _{I_0}.
\end{align*} 
If $N(J)\geq \frac{t\tilde{c}}{cr\log k^2}$, then $cN(J)\frac{r}{t\log k}\leq\frac{\tilde{c}}{\log k^3}$ and for $l\geq \log k^4$, which guarantees $\frac{1}{kr}=o\left(\frac{1}{\log k^3}\right)$,
$$-cN(J)\frac{r}{t\log k}+O\left(\frac{1}{kr}\right)\leq  -\frac{c_2 r N(J)}{t\log k}$$
for $c_2=c/2$. Likewise, if $\frac{r\vert \tau^{\tilde{\omega}}\vert}{n}\geq \frac{1}{\log n}$ then $\frac{1}{kr}=O\left(\frac{r\vert \tau^{\tilde{\omega}}\vert}{n\log k}\right)$. Hence, in any case,
$$\exp^{-\gamma_{\ell}\vert \tau^{\tilde{\omega}}\vert\left(1+O\left(\frac{1}{\log k}\right)\right)-cN(J)\frac{r}{t\log k}+O\left(\frac{1}{kr}\right)}\alpha^{\ell} _{I_0}\leq e^{-\gamma_{\ell}\vert \tau^{\tilde{\omega}}\vert\left(1+O\left(\frac{1}{\log k}\right)\right)-\frac{c_2}{\log k}}\alpha^{\ell} _{I_0},$$
and the result is deduced. 

Assume now that $N(J)\leq \frac{t\tilde{c}}{cr\log k}$ and $\frac{r\vert \tau^{\tilde{\omega}}\vert}{n}\leq \frac{1}{\log n}$. Then, setting $\delta(\theta)=k(g_J(re^{i\theta})-g(re^{i\theta}))-cN(J)\frac{r}{t\log k}+iO\left(rN(J)\right)$ and $\epsilon=cN(J)\frac{r}{t\log k}+\gamma_{\ell}\vert \tau^{\tilde{\omega}}\vert$, we have $\Vert\delta\Vert_{\infty}=O\left(\epsilon\log k\right)$. Moreover, if $\frac{r\vert \tau^{\tilde{\omega}}\vert}{n}\geq cN(J)\frac{r}{t}$, then the condition $\frac{r\vert \tau^{\tilde{\omega}}\vert}{n}\leq \frac{1}{\log n}$ ensures 
$$\Vert\delta\Vert_{\infty}^2=O\left(\left(\frac{r\vert \tau^{\tilde{\omega}}\vert}{n}\right)^2\right)=O\left(\frac{1}{\log k}\frac{r\vert \tau^{\tilde{\omega}}\vert}{n}\right)=O\left(\frac{\epsilon}{\log k}\right).$$
If $cN(J)\frac{r}{t}\geq \frac{r\vert \tau^{\tilde{\omega}}\vert}{n}$, then, since $\frac{N(J)r}{t}=O\left(\frac{1}{\log k}\right)$,
$$\Vert\delta\Vert_{\infty}^2=O\left(\left(cN(J)\frac{r}{t}\right)^2\right)=\frac{1}{\log k^2}O\left(\left(cN(J)\frac{r}{t}\right)\right)=\frac{1}{\log k}O\left(\left(cN(J)\frac{r}{t\log k }\right)\right)=O\left(\frac{\epsilon}{\log k}\right).$$
Finally, $\Re\delta(\theta)\leq -\epsilon\left(1+\frac{1}{\log k}\right)$ on $[-\frac{\tau\sqrt{\log k}}{\sqrt{kr}},\frac{\tau\sqrt{\log k}}{\sqrt{kr}}]$ with $\tau$ large enough. Hence, applying Lemma \ref{lem:integral_small_pertub_f} yields
\begin{align*}
e^{-cN(J)\frac{r}{t\log k}}\left\vert \frac{1}{2\pi}\int_{-\pi}^{\pi}e^{kg_{\tilde{J}}(re^{i\theta})+iO(rN(J))-l\log r}d\theta\right\vert\leq\exp\left(-\left(\gamma_{\ell}\vert \tau^{\tilde{\omega}}\vert+\frac{cN(J)r}{t\log k}\right)\left(1+O\left(\frac{1}{\log k}\right)\right)\right)\alpha^{\ell} _{I_0},
\end{align*}
provided $\ell\geq \log n^7$. Since 
$\vert \tau^{\tilde{\omega}}\vert\leq \frac{n^2}{\log n^2}$ for $\tilde{J}\in\mathcal{J}_3$, and $N(J)\leq C\frac{k}{\sqrt{\log k}}$ by the proof of Lemma \ref{lem:J2_Wasserstein},
$$\gamma_{\ell}\vert \tau^{\tilde{\omega}}=O\left(\frac{rn^2}{n^2\log n^2}\right)=O\left(\frac{r}{\log n^2}\right),\, \frac{cN(J)r}{t\log k}=O\left(\frac{rk}{\log k^{3/2}}\right),$$
and thus 
$$e^{-krc}=O\left(\exp\left(-\left(\gamma_{\ell}\vert \tau^{\tilde{\omega}}\vert+\frac{cN(J)r}{t\log k}\right)\left(1+O\left(\frac{1}{\log k}\right)\right)\right)\right).$$
The result is then deduced.
\end{proof}
\section{Estimates in the sub-logarithmic regime}\label{sec:estimates_lowerlog}
The goal of this section is to prove Proposition \ref{prop:estimates_J} in the following cases :
\begin{enumerate}
\item $\omega\in\mathcal{J}_1$ and $\ell\leq C_1'\left(\log k\right)^2$ for some $C_1'>0$,
\item $\omega\in\mathcal{J}_2$ for $C_2$ large enough and $\ell\geq (\log n)^s$ for some $s>0$,
\item $\omega\in\mathcal{J}_3$ and $\ell\geq (\log n)^7$.
\end{enumerate}
As in the former section, we assume that 
$0\leq\ell\leq k/2$. Let us recall the Newton's identities
$$e_{\ell}(x_1,\ldots,x_k)=\sum_{\substack{\lambda\vert \ell\\\lambda=1^{m_1}\ldots k^{m_k}}}(-1)^{n-\sum_{s=1}^km_s}\frac{\prod_{s=1}^kp_s(x_1,\ldots,x_k)^{m_s}}{\prod_{s=1}^km_s!s^{m_s}}.$$
\begin{lemma}\label{lem:asymptotic_el_p1}
Suppose that $J\in B_{k,n}$ is such that $d_{TV}(\mu_J,\mu_{I_0})\leq C\frac{k}{\sqrt{\log k}}$ and $\ell\leq \log k^s$ for some $s\geq 1$. Then, writing $p_i=p_i(\xi(J))$ for $i\geq 1$,
$$e_{\ell}(\xi(J))=\frac{1}{\ell!}\left(p_1^{\ell} -\frac{\ell(\ell-1)}{2}p_1^{\ell-2}p_2+a_1(\ell)p_1^{\ell-3}p_3-a_2(\ell)p_1^{\ell-4}p_2^2+ p_1^{\ell} O\left(\frac{\log k^{s'}}{k^3}\right)\right),$$
with $a_1(\ell),a_2(\ell)$ polynomials in $\ell$, $s'$ only depending on $s$ and $O(.)$ only depending on $C$.
\end{lemma}
\begin{proof}
Following Newtons identities, we write 
$$e_{\ell}(x_1,\ldots,x_k)=\frac{1}{\ell!}p_1^{\ell} -\frac{p_1^{\ell-1}p_{2}}{2(l-2)!}+\frac{p_1^{\ell-3}p_3}{3(l-3)!}+\frac{p_1^{\ell-4}p_2^2}{2^3(\ell-4)!}+
\sum_{\substack{\lambda\vert l\\\lambda=1^{m_1}\ldots l^{m_{\ell}}\\\sum_{i=1}^{\ell} m_i\leq l-3}}(-1)^{\ell-\sum_{s=1}^{\ell} m_s}\frac{\prod_{s=1}^{\ell} p_s(x_1,\ldots,x_k)^{m_s}}{\prod_{s=1}^{\ell} m_s!s^{m_s}}.$$
Next, remark that $p_s(\xi(J))=\int_{-\pi}^\pi e^{isx}d\mu_j(x)$. Since $d_{TV}(\mu_J,\mu_{I_0})\leq C\frac{k}{\sqrt{\log k}}$ for some constant $C>0$, for $1\leq s\leq k$,
$$\vert p_s(\xi(I_0))\vert-C\frac{k}{\sqrt{\log k}}\leq \vert p_s(\xi(J))\vert\leq \vert p_s(\xi(I_0))\vert+C\frac{k}{\sqrt{\log k}}.$$ 
Since $\vert p_s(\xi(I_0))\vert=\frac{\vert \sin(sk\pi/n)\vert}{\sin(s\pi/n)}$, for $s\geq 2$ we have $\vert p_s(\xi(I_0))\vert\leq \alpha \vert p_1(\xi(I_0))\vert$ for some constant $\alpha$ only depending on $k/n$. Hence, for $n$ large enough,
$$\vert p_s(\xi(J))\vert\leq \left\vert p_1(\xi(J))\right\vert$$
for $2\leq s\leq k$. For $s=1$, the inequality is trivially satisfied, so that
\begin{align}
\left\vert \sum_{\substack{\lambda\vert l\\\lambda=1^{m_1}\ldots l^{m_{\ell}}\\\sum_{i=1}^{\ell} m_i\leq l-3}}(-1)^{\ell-\sum_{s=1}^{\ell} m_s}\frac{\prod_{s=1}^{\ell} p_s(x_1,\ldots,x_k)^{m_s}}{\prod_{s=1}^{\ell} m_s!s^{m_s}}\right\vert\leq&\sum_{\substack{\lambda\vert l\\\lambda=1^{m_1}\ldots l^{m_{\ell}}\\\sum_{i=1}^{\ell} m_i\leq l-3}}\frac{\left\vert p_1(\xi(J))\right\vert^{\sum_{s=1}^{\ell} m_s}}{\prod_{s=1}^{\ell} m_s!s^{m_s}}\nonumber\\
=&\left(\sum_{\substack{\lambda\vert l\\\lambda=1^{m_1}\ldots l^{m_{\ell}}}}-\sum_{\substack{\lambda\vert l\\\lambda=1^{m_1}\ldots l^{m_{\ell}}\\\sum_{i=1}^{\ell} m_i\geq l-2}}\right)\frac{\left\vert p_1(\xi(J))\right\vert^{\sum_{s=1}^{\ell} m_s}}{\prod_{s=1}^{\ell} m_s!s^{m_s}}\nonumber\\
=&\frac{1}{\ell!}\left(\sum_{\sigma\in S_{\ell}}p_1(\xi(J))^{l(\sigma)}-\sum_{\substack{\sigma\in S_{\ell}\\l(\sigma)\geq l-2}}\left\vert p_1(\xi(J))\right\vert^{l(\sigma)}\right).\label{eq:small_l_bound_el}
\end{align}
Using that the generating series of cycle lengths of the permutation group is $\sum_{\sigma\in S_{\ell}}\left\vert p_1(\xi(J))\right\vert^{l(\sigma)}=\prod_{i=0}^{\ell-1}(\left\vert p_1(\xi(J))\right\vert+i)$, we get, setting $x=\left\vert p_1(\xi(J))\right\vert$
\begin{align*}
\left(\sum_{\sigma\in S_{\ell}}x^{l(\sigma)}-\sum_{\substack{\sigma\in S_{\ell}\\l(\sigma)\geq l-2}}x^{l(\sigma)}\right)&=x^{l}\left(\prod_{i=0}^{\ell-1}\left(1+\frac{i}{x}\right)-\sum_{\substack{\sigma\in S_{\ell}\\l(\sigma)\geq l-2}}x^{l(\sigma)-l}\right).
\end{align*}
Then,
\begin{align*}
\prod_{i=0}^{\ell-1}\left(1+\frac{i}{x}\right)=&\exp(\sum_{i=0}^{\ell-1}\log(1+\frac{i}{x}))\\
=&\exp\left(\frac{1}{x}\sum_{i=0}^{\ell-1}i-\frac{1}{2x^2}\sum_{i=0}^{\ell-1}i^2+o\left(\frac{1}{x^3}\sum_{i=0}^{\ell-1}i^3\right)\right)\\
=&1+\frac{\ell(\ell-1)}{2x}+\frac{a_1(\ell)+a_2(\ell)}{x^2}+O\left(\frac{l^4}{x^3}\right).
\end{align*}
Since $\sum_{\substack{\sigma\in S_{\ell}\\l(\sigma)\geq \ell-2}}x^{l(\sigma)-\ell}=1+\frac{\ell(\ell-1)}{2x}+\frac{a_1(\ell)+a_2(\ell)}{x^2}$, we get 
$$\left(\sum_{\sigma\in S_{\ell}}x^{l(\sigma)-\ell}-\sum_{\substack{\sigma\in S_{\ell}\\l(\sigma)\geq l-2}}x^{l(\sigma)-\ell}\right)=O\left(\frac{\ell^4}{x^3}\right)=O\left(\frac{\log k^{s'}}{k^3}\right).$$
The statement of the lemma is then deduced by plugging the latter estimate in \eqref{eq:small_l_bound_el}. 
\end{proof}
\begin{proposition}\label{prop:case_J3_small}
When $\ell\leq \log n^7$ and $\omega\in\mathcal{J}_3$,
$$\lambda_{J_{\omega}}=\exp\left(-\frac{\gamma_{\ell}\vert \tau^{\tilde{\omega}}\vert}{n}(1+O(1/\log n))+iO\left(\gamma_{\ell}\vert \tau^{\omega}\vert\right)\right).$$
\end{proposition}
\begin{proof}
We suppose here $\ell\geq 4$, the proof being exactly similar for $l\leq 3$. Let $\omega\in \mathcal{J}_3$ and consider $J_{\omega}$ introduced in Section \ref{subsec:summary_bounds_alpha}. First, $d_{TV}(\mu_{J_{\omega}},\mu_{I_0})\leq C\frac{k}{\log k}$, using Lemma \ref{lem:asymptotic_el_p1} with ,
\begin{align}
&\frac{\alpha_{J_{\omega}}^{\ell} }{\alpha_{I_0}^{\ell} }=\frac{ p_1^{\ell-4}}{ q_1^{\ell-4}}\cdot\frac{p_1^4-\frac{\ell(\ell-1)}{2}p_1^{2}p_2+a_1(\ell)p_1p_3-a_2(\ell)p_2^2+ p_1^4O\left(\frac{\log k^{s'}}{k^3}\right)}{q_1^4-\frac{\ell(\ell-1)}{2}q_1^{2}q_2+a_1(\ell)q_1q_3-a_2(\ell)q_2^2+ q_1^4O\left(\frac{\log k^{s'}}{k^3}\right)},\label{eq:ratio_small_l}
\end{align}
where $p_i:=p_i(\xi(J_{\omega}))$, $q_i=p_i(\xi(I_0))$ for $1\leq i\leq 3$. Recall from Section \ref{subsec:summary_bounds_alpha} that $\hat{J}_{\omega}=Supp(\mu_{J_{\omega}})$ is decomposed as 
$$\hat{J}_{\omega}=\frac{2\pi}{n}\left\{\frac{k+1}{2}-j+\mu^{\omega}_j,1\leq j\leq \frac{k+1}{2}\right\}\sqcup \left\{\frac{2\pi}{n}\{-\frac{k+1}{2}+j-\nu^{\omega}_j,1\leq j< \frac{k+1}{2}\right\},$$
so that
$$p_1(\xi(J))=p_1(\xi(I_0))+\sum_{j=1}^{\frac{k+1}{2}}e^{\frac{2i\pi}{n} \left(\frac{k+1}{2}-j+\mu^{\omega}_j\right)}-e^{\frac{2i\pi}{n}\left(\frac{k+1}{2}-j\right)}+\sum_{j=1}^{\frac{k}{2} }e^{\frac{2i\pi}{n} \left(-\frac{k+1}{2}+j-\nu^{\omega}_j\right)}-e^{\frac{2\pi}{n}\left(-\frac{k+1}{2}+j\right)}.$$
Since $\mu^{\omega}_1\leq \frac{2C_2n}{\pi\log n}$,
\begin{align*}
e^{\frac{2i\pi}{n} \left(\frac{k+1}{2}-j+\mu^{\omega}_j\right)}-e^{\frac{2i\pi}{n}\left(\frac{k+1}{2}-j\right)}=&e^{2\pi i(j-(k-1)/2)/n}(2\pi i\mu^{\omega}_j/n(1+O(1/\log n)))\\
=&e^{i\pi k/n}\frac{2\pi i\mu^{\omega}_j}{n}(1+O(1/\log n)),
\end{align*}
with $O(\cdot)$ only depending on $c''$ and $O(\cdot)=0$ if $\mu^{\omega}_j=0$ and the same holds for $\nu^{\omega}$ Hence, since $\max(\ell(\mu^{\omega}),\ell(\nu^{\omega}))\leq \frac{2 C_2 n}{\pi\log n}$
\begin{align*}
p_1(\xi(J))=&p_1(\xi(I_0))+2\pi i\left[e^{i\pi k/n}\vert \lambda\vert-e^{-i\pi k/n}\vert \mu\vert\right]\left(1+O(1/\log n)\right)\\
=&\frac{\sin(k\pi/n)}{\sin(\pi/n)}+2\pi i\left[e^{i\pi k/n}\vert \lambda\vert-e^{-i\pi k/n}\vert \mu\vert\right]\left(1+O(1/\log n)\right)\\
=&p_1(\xi(I_0))\left(1-\gamma_1\vert \tau^{\omega}\vert(1+O(1/\log n))+iO\left(\gamma_1\vert \tau^{\omega}\vert\right)\right),
\end{align*}
with $\gamma_1=\frac{2\pi^2}{n^2 \sin\left(\frac{k\pi}{n}\right)}$ and $\vert \tau^{\omega}\vert=\vert \mu^{\omega}\vert+\vert \nu^{\omega}\vert$, so that 
\begin{equation}\label{eq:ratio_sum_small_l}
\frac{p_1(\xi(J))}{p_1(\xi(I_0))}=1-\frac{\gamma_1\vert \tau^{\omega}\vert}{n}\left(1+O\left(\frac{1}{\log n}\right)\right)+iO\left(\gamma_1\vert \tau^{\omega}\vert\right).
\end{equation}
One proves similarly that 
\begin{equation}\label{eq:ratio_power_sum_small_l}\frac{p_r(\xi(J))}{p_r(\xi(I_0))}=1+O\left(\frac{\gamma_1\vert \tau^{\omega}\vert}{n}\right)
\end{equation}
for $r=2,3$. 

First, by \eqref{eq:ratio_sum_small_l},
\begin{align*}
\frac{ p_1^{\ell-4}}{ q_1^{\ell-4}}=&\left( 1-\frac{\gamma_1\vert \tau^{\omega}\vert}{n}\left(1+O\left(\frac{1}{\log n}\right)\right)+iO\left(\gamma_1(\vert \tau^{\omega}\vert\right)\right)^{\ell-4}\\
=& \exp\left(-\frac{(\ell-4)\gamma_1\vert \tau^{\omega}\vert}{n}(1+O(1/\log n))+iO\left(\gamma_\ell(\vert \tau^{\omega}\vert\right)\right).
\end{align*}
Then, using \eqref{eq:ratio_small_l} and \eqref{eq:ratio_power_sum_small_l}, we get 
\begin{align*}
&p_1^4-\frac{\ell(\ell-1)}{2}p_1^{2}p_2+a_1(\ell)p_1p_3-a_2(\ell)p_2^2+ p_1^4O\left(\frac{\log k^{s'}}{k^3}\right)\\
&=q_1^4\left(1-\frac{4\gamma_1\vert \tau^{\omega}\vert}{n}\left(1+O\left(\frac{1}{\log n}\right)+4iO\left(\gamma_1\vert \tau^{\omega}\vert\right)\right)\right)\\
&\hspace{1cm}+\left[-\frac{\ell(\ell-1)}{2}q_1^{2}q_2+a_1(\ell)q_1q_3-a_2(\ell)q_2^2+ q_1^4O\left(\frac{\log k^{s'}}{k^3}\right)\right]\left(1+O\left(\frac{\gamma_1\vert \tau^{\omega}\vert}{n}\right)\right).
\end{align*}
Since 
\begin{equation}\label{eq:small_l_negligible_rest}
\left[-\frac{\ell(\ell-1)}{2}q_1^{2}q_2+a_1(\ell)q_1q_3-a_2(\ell)q_2^2+ q_1^4O\left(\frac{\log k^{s'}}{k^3}\right)\right]=q_1^4O\left(\frac{1}{\log n}\right),
\end{equation}
we thus have
\begin{align*}
&p_1^4-\frac{\ell(\ell-1)}{2}p_1^{2}p_2+a_1(\ell)p_1p_3-a_2(\ell)p_2^2+ p_1^4O\left(\frac{\log k^{s'}}{k^3}\right)\\
=&q_1^4\left(1-\frac{4\gamma_1\vert \tau^{\omega}\vert}{n}\left(1+O\left(\frac{1}{\log n}\right)\right)+4iO\left(\gamma_1\vert \tau^{\omega}\vert\right)\right),
\end{align*}
so that, using again \eqref{eq:small_l_negligible_rest},
\begin{align*}
\left\vert\frac{p_1^4-\frac{\ell(\ell-1)}{2}p_1^{2}p_2+a_1(\ell)p_1p_3-a_2(\ell)p_2^2+ p_1^4O\left(\frac{\log k^{s'}}{k^3}\right)}{q_1^4-\frac{\ell(\ell-1)}{2}q_1^{2}q_2+a_1(\ell)q_1q_3-a_2(\ell)q_2^2+ q_1^4O\left(\frac{\log k^{s'}}{k^3}\right)}\right\vert=&1-\frac{q_1^4\frac{4\gamma_1\vert \tau^{\omega}\vert}{n}\left(1+O\left(\frac{1}{\log n}\right)\right)}{q_1^4\left(1+O\left(\frac{1}{\log n}\right)\right)+4iO\left(\gamma_1\vert \tau^{\omega}\vert\right)}\\
=&1-\frac{4\gamma_1\vert \tau^{\omega}\vert}{n}\left(1+O\left(\frac{1}{\log n}\right)\right)\\
&\hspace{2cm}+4iO\left(\gamma_1\vert \tau^{\omega}\vert\right).
\end{align*}
Combining both estimates yields
\begin{align*}
\lambda_{J_{\omega}}=\frac{\alpha_J^{\ell} }{\alpha_{I_0}^{\ell} }=& \exp\left(-\frac{\ell\gamma_1\vert \tau^{\omega}\vert}{n}(1+O(1/\log n))+iO\left(\gamma_{\ell}\vert \tau^{\omega}\vert\right)\right)\\
=& \exp\left(-\frac{\gamma_{\ell}\vert \tau^{\omega}\vert}{n}(1+O(1/\log n))+iO\left(\gamma_{\ell}\vert \tau^{\omega}\vert\right)\right).
\end{align*}
\end{proof}
\begin{proposition}\label{prop:case_j2_small}
Let $\ell\leq \log n^s$. Let $\eta>0$. For $C_2$ large enough, for all $\omega\in\mathcal{J}_2$ there exists $\tilde{\omega}\in \mathcal{J}_3$ such that for any $t\geq t_0$, 
$$d(\omega)^{\frac{\gamma_{\ell}}{tk^2\log k}}\left\vert \lambda_{\omega}^{\ell} \right\vert\leq d(\tilde{\omega})^{\frac{\gamma_{\ell}}{tk^2\log k}}\exp\left(-\frac{lc_2}{tk\log k}\right)\left\vert\lambda^{\ell} _{\tilde{\omega}}\right\vert.$$
\end{proposition}
\begin{proof}
Let $\omega\in\mathcal{J}_2$ and $J\in \omega$ satisfying 
$$ \int\mathbf{1}_{\vert T^J(x)\vert>k\pi/n+\frac{C_2}{\log n}}(\vert T^J(x)\vert-\vert x\vert)d\mu_{I_0}(x)\geq C_1\frac{k}{\log k}.$$
First, since $d_{TV}(\mu_J,\mu_{I_0})\leq \frac{C k}{\sqrt{\log k}}$ by Lemma \ref{lem:J2_Wasserstein}, 
\begin{equation}\label{eq:J_2_small_Re_lowerbound}\Re p_1(\xi(J))\geq \frac{\sin((k+1)\pi/n)}{\sin(\pi/n)}- \frac{C k}{\sqrt{\log k}}\geq ck
\end{equation}
for some constant $c>0$.  Recall that for $I\in B_{k,n}$,
$$N(I):=\#\{x\in I_0,\,\min(\vert T^I(x)\vert-k\pi/n,k\pi/n-\vert x\vert)\geq\frac{C_2}{\log n}\}$$
and that, when $N(I)=0$, there exists $\tilde{\omega}\in \mathcal{J}_3$ such that $I\in\tilde{\omega}$, see the definition of $\mathcal{J}_3$ in Section \ref{subsec:summary_bounds_alpha}.

Construct from $J$ a sequence $(J^{(m)})_{m\geq 0}$ of elements of $B_{k,n}$ as follows :
\begin{itemize}
\item $J^{(0)}=J$,
\item if $N(J^{(m)}>0$, let $(x,y)$ with $y\in \widehat{J^{(m)}}\setminus \widehat{I_0}, x\in \widehat{I_0}\setminus \widehat{J^{(m)}}$ and either $\vert y\vert\geq \pi k/n+\frac{C_2}{\log n}$ and $\vert x\vert\leq \frac{k\pi}{n}$,  or $\vert y\vert\geq \frac{k\pi}{n}$ and $\vert x\vert\leq \pi k/n-\frac{C_2}{\log n}$. Let $\tilde{J}$ be the configuration obtained by replacing $y$ by $x$. Then, set $J^{(m+1)}=\tilde{J}+s\mathbf{1}_k$ as the unique rotated version of $\tilde{J}$ such that $\arg p_1(\xi(J^{(m+1)}))\in ]-\frac{\pi}{2n},\frac{\pi}{2n}]$.
\item if $N(J^{(m)})=0$, $J^{(m+1)}=J^{(m)}$.
\end{itemize}
The second step simply stop the algorithm if it encounters some configuration belonging to some $\tilde{\omega}\in\mathcal{J}_3$. Let $\varrho=\min(n\geq 1, N(J^{(m)})=0)$, with the convention $\min \emptyset=+\infty$. Let us prove that  for $m<\varrho$,
\begin{enumerate}
\item $\vert p_1(\xi(J^{(m+1)}))\vert \geq \vert p_1(\xi(J^{(m)}))\vert+\frac{c'C_2}{\log n}\geq ck$,
\item $\vert d(J^{(m)})\vert \leq \kappa^k \vert d(J^{(m+1)})\vert$,
\end{enumerate}
for some $\kappa,c'>0$ independent of $n$ and $C_2$. Let us prove the first assumption by recursion on $0\leq n<\varrho$.
Let $n<\varrho$. Since $N(J^{(m)})>0$ there exists $(x,y)$ with $y\in \widehat{J^{(m)}}\setminus \widehat{I_0}, x\in \widehat{I_0}\setminus \widehat{J^{(m)}}$ and either $\vert y\vert\geq \pi k/n+\frac{C_2}{\log n}$ and $\vert x\vert\leq \frac{k\pi}{n}$,  or $\vert y\vert\geq \frac{k\pi}{n}$ and $\vert x\vert\leq \pi k/n-\frac{C_2}{\log n}$. Let $\tilde{J}$ be the configuration obtained by replacing $y$ by $x$. Either by \eqref{eq:J_2_small_Re_lowerbound} for the initialization or by induction otherwise, 
$$\Re p_1(\xi(J^{(m)}))\geq ck.$$
Then, since $\arg p_1(\xi(J^{(m)}))\in ]-\pi/(2n),\pi/(2n)]$, 
$$\vert \Im p_1(\xi(J^{(m)}))\vert \leq \epsilon$$
for some $\epsilon$ only depending on $k$ and $n$. Since $p_1(\xi(\tilde{J}))=p_1(\xi(J^{(m)}))+e^{ix}-e^{iy}$,
$$\Re p_1(\xi(\tilde{J}))\geq \Re p_1(\xi(J^{(m)}))+\frac{C}{\log n}.$$ Hence,
\begin{align*}
\left\vert p_1(\xi(\tilde{J}))\right\vert^2\geq \Re p_1(\xi(\tilde{J}))^2\geq \Re p_1(\xi(J^{(m)}))^2+2ck\frac{C_2}{\log n}\geq \left\vert p_1(\xi(J^{(m)}))\right\vert^2+2ck\frac{C_2}{\log n}-\epsilon ^2.
\end{align*}
Therefore, 
$$\left\vert p_1(\xi(\tilde{J}))\right\vert\geq\left\vert p_1(\xi(J^{(m)}))\right\vert^2+c\frac{C_2}{\log n}$$
for some $c>0$, and by rotation and induction,
$$\vert p_1(\xi(J^{(m+1)}))\vert \geq \vert p_1(\xi(J^{(m)}))\vert+\frac{c'C_2}{\log n}\geq ck,$$
where we used the induction hypothesis on the last inequality.
For the second assertion, let $l,l'$ be such that $\frac{2\pi}{n}\left(J^{(m)}_{\ell}-\frac{k+1}{2}\right)=x$ and $\frac{2\pi}{n}\left(\tilde{J}_{l'}-\frac{k+1}{2}\right)=y$. Since $\{J^{(m)}_s,s\not=l\}=\{\tilde{J}_s,s\not=l'\}$,
$$\left\vert \frac{d(\tilde{J})}{d(J^{(m)})}\right\vert =\frac{\prod_{s\not =l'} \sin(\pi \vert J^{(m)}_s-\tilde{J}_{l'}\vert/n)}{\prod_{s\not =l} \sin(\pi \vert J^{(m)}_s-J_{l}\vert/n)}\geq \prod_{s\not =l'} \sin(\pi \vert J^{(m)}_s-\tilde{J}_{l'}\vert/n)\geq \prod_{j=1}^{k/2}\sin(\pi j/n)^2.$$
Remark that 
$$\prod_{j=1}^{k/2}\sin(\pi j/n)^2=\exp\left(k\left(\frac{2}{k}\sum_{j=1}^{k/2} \log \sin(\pi j/n)\right)\right)=\exp\left(2n\left(\int_0^{1/2}\log\sin\left(\frac{k\pi}{n} t\right)dt+o(1)\right)\right)\geq \kappa^{-k}$$
for some universal constant $\kappa>0$. 
Hence, 
$$\vert d(J^{(m)})\vert\leq \kappa^{k}\vert d(\tilde{J})\vert.$$
Since $d(J^{m+1})=d(\tilde{J})$, the second assertion is also proven.

By (1), $\vert p_1(\xi(J^{(m)}))\vert\geq ck +m\frac{c'C_2}{\log n}$ for $m\leq \tau$. Since $\vert p_1(\xi(J^{(m)}))\leq k$, this implies that $\tau <+\infty$, and thus there exists $m_0$ such that $N(J^{(m_0)})=0$. Let $\tilde{\omega}\in\mathcal{J}_3$ be such that $J^{(m_0)}\in \tilde{\omega}$.

Then, by Lemma \ref{lem:asymptotic_el_p1},
$$\frac{\left\vert \lambda_{\tilde{\omega}}\right\vert}{\left\vert \lambda_{\omega}\right\vert}=\left\vert\frac{\alpha_{J^{(m_0)}}^{\ell} }{\alpha_{J}^{\ell} }\right\vert=\left\vert\frac{e_{\ell}(\xi(J^{(m_0)}))}{e_{\ell}(\xi(J))}\right\vert=\frac{\left\vert p_1(\xi(J^{(m_0)}))^{\ell-2}\left(p_1(\xi(J^{(m_0)}))^2-\frac{\ell(\ell-1)}{2}p_2(\xi(J^{(m_0)}))+O\left(\frac{\log k^s}{k^2}\right)\right)\right\vert}{\left\vert p_1(\xi(J))^{\ell-2}\left(p_1(\xi(J))^2-\frac{\ell(\ell-1)}{2}p_2(\xi(J))+O\left(\frac{\log k^s}{k^2}\right)\right)\right\vert}.$$
Since $\vert p_2(\xi(J^{(m_0)})-p_2(\xi(J))\vert \leq m_0$ and $p_1(\xi(J))\geq ck$,
\begin{align*}
p_1(\xi(J^{(m_0)}))^2-\frac{\ell(\ell-1)}{2}p_2(\xi(J^{(m_0)}))\geq& (p_1(\xi(J))+\frac{mc'C_2}{\log n})^2-\frac{\ell(\ell-1)}{2}(p_2(\xi(J))+m)\\
\geq& p_1(\xi(J))^2-\frac{\ell(\ell-1)}{2}p_2(\xi(J)) +2p_1(\xi(J))\frac{mc'C_2}{\log n}-\frac{\ell(\ell-1)}{2}m\\
\geq &p_1(\xi(J))^2-\frac{\ell(\ell-1)}{2}p_2(\xi(J)) +p_1(\xi(J))\frac{mc'C_2}{\log n}
\end{align*}
for $n$ large enough, where we used that $\ell\leq\log k^s$ in the last inequality. Finally, there exists $c>0$ such that 
$$\left\vert \lambda^{\ell} _{\omega}\right\vert\leq \left(1-\frac{mcC_2}{k\log k}\right)^{\ell} \left\vert \lambda^{\ell} _{\tilde{\omega}}\right\vert.$$
Since $d(\tilde{\omega})\geq \kappa^{km}d(\omega)$ by $(2)$ the latter inequality yields for $t\geq t_0$
$$d(\omega)^{\frac{ \ell}{tk^2\log k}}\left\vert \lambda_{\omega}^{\ell} \right\vert\leq e^{ \frac{m \ell}{k\log k}\left(\frac{\log(\theta)}{t_0}-cC_2\right)}d(\tilde{\omega})^{\frac{\ell}{tk^2\log k}}\left\vert \lambda_{\tilde{\omega}}^{\ell} \right\vert,$$
and, for $c_2>0$, choosing $C_2$ large enough (only depending on $t_0$ and $\theta$ and $c$) yields 
$$d(\omega)^{\frac{\ell}{t k^2\log k}}\left\vert \lambda_{\omega}^{\ell} \right\vert\leq e^{-\frac{cl}{k\log k}}d(\tilde{\omega})^{\frac{\ell}{t k^2\log k}}\left\vert \lambda_{\tilde{\omega}}^{\ell} \right\vert$$
for $t\geq t_0$. Since $\gamma_{\ell}=\frac{2\pi^2\ell}{k^2}\left(1+O\left(\frac{\log ^sn}{n}\right)\right)$, the result is deduced.
\end{proof}

\begin{lemma}\label{lem:subgroup}
For all $J$ and for all $s\in\llbracket 1,\ell\rrbracket$, we have
\[
\left|p_s\left(\xi(J)\right)\right|\leq \left|p_1\left(\xi(I_0)\right)\right|\left(1+O\left(\frac{s}{n}\right)\right)\, .
\]
\end{lemma}
\begin{proof}[Proof of Lemma~\ref{lem:subgroup}]
Let $\Phi$ be the group morphism which maps $\omega\in U_n$ to $\omega^s$, where $U_n$ is the cyclic group of order $n$. The image $\Phi(U_n)$ is a subgroup of $U_n$, so it is of the form $\{\omega^{is'}\}_{0\leq i <r}$, for some $s'\leq s$ and $r=n/s'$. The kernel of $\Phi$ is also a subgroup, with size $s'$. Hence, since each $\omega^{J_i s}$ is in the image of $\Phi$,
\begin{align*}
p_s\left(\xi(J)\right)&= \sum_{i=1}^k \omega^{J_i s}=\sum_{i=0}^{r-1}d_i\omega^{is'}\, ,
\end{align*}
with $d_i\in\llbracket 0,s'\rrbracket$. Now since $\sum_{i=0}^{r-1}d_i=k$, we have
\[
\left|p_s\left(\xi(J)\right)\right|\leq  s'\left|\sum_{i=0}^{k/s'}\omega^{s'i}\right|=\frac{s'\sin\left(\frac{(k+s')\pi}{n}\right)}{\sin\left(\frac{s'\pi}{n}\right)}=\frac{n\sin\left(\frac{k\pi}{n}\right)}{\pi}\left(1+O\left(\frac{s}{n}\right)\right)\, .
\]
The result then follows from the fact that
\[
\left|p_1\left(\xi(I_0)\right)\right|=\frac{n\sin\left(\frac{k\pi}{n}\right)}{\pi}\left(1+O\left(\frac{1}{n}\right)\right)\, .
\]
\end{proof}

\begin{lemma}\label{lem:bound_J_1_small}
If there exists $C_1'>0$ such that $\ell\leq C_1' (\log k)^2$, then for all $\omega\in\mathcal{J}_1$, 
\[
\left\vert \lambda^{\ell}_{\omega}\right\vert=\left(1+O\left(\frac{\ell}{\sqrt{k}}\right)\right)e^{-\frac{C_1n^2\gamma_{\ell}}{2\pi^2\log k}}\, .
\]
\end{lemma}
\begin{proof}[Proof of Lemma~\ref{lem:bound_J_1_small}]
We first observe that for $\ell\leq c (\log k)^2$, we have
\[
\left|e_\ell\left(\xi(I_0)\right)\right|=\frac{\prod_{j=k-\ell+1}^k \sin\left(\frac{\pi j}{n}\right)}{\prod_{j=1}^\ell \sin\left(\frac{\pi j}{n}\right)}=\left(1+O\left(\frac{\ell^2}{n}\right)\right)\frac{1}{\ell !}\left(\frac{\sin\left(\frac{\pi k}{n}\right)}{\sin\left(\frac{\pi }{n}\right)}\right)^\ell=\left(1+O\left(\frac{\ell^2}{n}\right)\right)\frac{1}{\ell !} \left|p_1\left(\xi(I_0)\right)\right|^\ell\, .
\]
On the other hand, using Newton's identities and the triangle inequality, we have
\[
\left|e_\ell\left(\xi(J)\right)\right|\leq \frac{1}{\ell !}\sum_{\sigma\in S_\ell}\prod_{s=1}^\ell \left|p_s\left(\xi(J)\right)\right|^{m_s}\, ,
\]
where $m_s$ denotes the number of cycles of length~$s$ in permutation~$\sigma$. Using Lemma~\ref{lem:subgroup}, we get
\[
\ell ! \left|e_\ell\left(\xi(J)\right)\right|\leq \sum_{\sigma\in S_\ell}\left|p_1\left(\xi(J)\right)\right|^{m_1} \Lambda^{\sum_{s=2}^\ell m_s}\, ,
\]
where $\Lambda=\left|p_1\left(\xi(I_0)\right)\right|\left(1+O\left(\frac{\ell}{n}\right)\right)$. In the sum above, we consider separately the terms corresponding to permutations $\sigma$ with $C(\sigma)\geq \ell -c$ and with $C(\sigma)<\ell -c$, where $C(\sigma)=\sum_{s=1}^\ell m_s$ is the total number of cycles of~$\sigma$, and where $c$ is a fixed integer to be specified later. First, if $C(\sigma)\geq \ell -c$, then there at least $\ell-2c$ fixed points. Using that there are less than $\ell^{r}$ permutations with $\ell-r$ fixed points, we thus have
\[
\sum_{\substack{\sigma\in S_\ell\\ C(\sigma)\geq \ell -c}}\left|p_1\left(\xi(J)\right)\right|^{m_1} \Lambda^{\sum_{s=2}^\ell m_s}\leq \sum_{r=0}^{2c}\ell^r \left|p_1\left(\xi(J)\right)\right|^{\ell-r} \Lambda^{r/2}\, .
\]
Now, if $J\in\mathcal{J}_1$, then $\left|p_1\left(\xi(J)\right)\right|\leq \left(1-\frac{C_1}{\log k}\right)\left|p_1\left(\xi(I_0)\right)\right|\leq \left(1-\frac{C_1}{\log k}\right)\Lambda$, which gives
\begin{align*}
\sum_{\substack{\sigma\in S_\ell\\ C(\sigma)\geq \ell -c}}\left|p_1\left(\xi(J)\right)\right|^{m_1} \Lambda^{\sum_{s=2}^\ell m_s} &\leq \Lambda^\ell e^{-\frac{C_1\ell}{\log k}}\sum_{r=0}^{2c} \left(\frac{\ell}{\sqrt{\Lambda}\left(1-\frac{C_1}{\log k}\right)}\right)^r\\
&\leq \Lambda^\ell e^{-\frac{C_1\ell}{\log k}} \frac{1}{1-\frac{2\ell}{\sqrt{\Lambda}}}=\Lambda^\ell e^{-\frac{C_1\ell}{\log k}} \left(1+O\left(\frac{\ell}{\sqrt{\Lambda}}\right)\right)
\end{align*}
Moreover, for the terms corresponding to $C(\sigma)<\ell -c$, we have
\begin{align*}
\sum_{\substack{\sigma\in S_\ell\\ C(\sigma)< \ell -c}}\left|p_1\left(\xi(J)\right)\right|^{m_1} \Lambda^{\sum_{s=2}^\ell m_s} &\leq \sum_{\substack{\sigma\in S_\ell\\ C(\sigma)< \ell -c}} \Lambda^{C(\sigma)} \\
&\leq \sum_{j=1}^{\ell-c-1}\Lambda^j{\ell \choose j}\ell^{\ell -j}\, ,
\end{align*}
where we used the fact that $\sum_{\sigma \in S_\ell}\Lambda^{C(\sigma)}=\prod_{i=0}^{\ell -1}(\Lambda+i)$, so that the coefficient in front of $\Lambda^j$ is less than ${\ell \choose j}\ell^{\ell -j}$. Therefore,
\begin{align*}
\sum_{\substack{\sigma\in S_\ell\\ C(\sigma)< \ell -c}} \Lambda^{C(\sigma)}&\leq \Lambda^\ell \sum_{j=1}^{\ell-c-1}{\ell \choose j}\left(\frac{\ell}{\Lambda}\right)^{\ell-j}\\
&=  \Lambda^\ell \sum_{j=c+1}^{\ell-1} {\ell \choose j} \left(\frac{\ell}{\Lambda}\right)^{j}\\
&\leq \Lambda^\ell \sum_{j=c+1}^{\ell-1} \left(\frac{\ell^2}{\Lambda}\right)^{j}\leq  \Lambda^\ell \left(\frac{\ell^2}{\Lambda}\right)^{c}\, \cdot
\end{align*}
Note that taking $c$ large enough with respect to $C_1$ and $C_1'$, we can always ensure that $\left(\frac{\ell^2}{\Lambda}\right)^{c}\leq \tfrac{1}{\Lambda}e^{-\frac{C_1\ell}{\log k}}$. All in all, we obtain
\[
\frac{\left|e_\ell\left(\xi(J)\right)\right|}{\left|e_\ell\left(\xi(I_0)\right)\right|}\leq \left(1+O\left(\frac{\ell}{\sqrt{\Lambda}}\right)\right)e^{-\frac{C_1\ell}{\log k}}\, .
\]
Since $\gamma_{\ell}=\frac{2\pi^2\ell}{n^2}\left(1+O\left(\frac{\log k^2}{n}\right)\right)$ by \eqref{eq:expression_gammal}, the result is deduced.
\end{proof}

\bibliographystyle{abbrvnat}
\bibliography{biblio}

\end{document}